\pdfoutput=1
\documentclass [11pt,reqno]{amsart}
\usepackage[english]{babel}
\usepackage[utf8]{inputenc}
\usepackage[T1]{fontenc}
\usepackage{lmodern}
\usepackage{amssymb,amsmath,amsthm,amsfonts}
\usepackage{thmtools,mathtools}
\usepackage{microtype}
\allowdisplaybreaks
\usepackage{tikz,tikz-cd}
\usepackage[breaklinks,pdfencoding=auto,psdextra]{hyperref}
\usepackage{geometry}
\usepackage{tikz-cd}
\usepackage{comment}
\usepackage{xfrac}
\usepackage{tabularx}

\usepackage{chngcntr}

\def\bA{{\mathbb{A}}}

\def\bC{{\mathbb{C}}}

\def\bN{{\mathbb{N}}}

\def\bP{{\mathbb{P}}}
\def\bQ{{\mathbb{Q}}}
\def\bR{{\mathbb{R}}}

\def\bZ{{\mathbb{Z}}}
\DeclareMathAlphabet{\mathmybb}{U}{bbold}{m}{n}

\def\cE{{\mathcal{E}}}
\def\cF{{\mathcal{F}}}
\def\cG{{\mathcal{G}}}

\def\cK{{\mathcal{K}}}
\def\cL{{\mathcal{L}}}
\def\cM{{\mathcal{M}}}
\def\cN{{\mathcal{N}}}
\def\cO{{\mathcal{O}}}
\def\cP{{\mathcal{P}}}

\def\cU{{\mathcal{U}}}

\def\cX{{\mathcal{X}}}
\def\cY{{\mathcal{Y}}}

\def\fo{{\mathfrak{o}}}

\def\m{\mathfrak m}
\def\p{\mathfrak p}

\def\alg{\operatorname{alg}}
\def\an{\operatorname{an}}
\def\ar{\operatorname{ar}}
\def\Aut{\operatorname{Aut}}

\def\Ban{\operatorname{Ban}}
\def\Berk{\operatorname{Berk}}

\def\diag{\operatorname{diag}}
\def\disc{\operatorname{disc}}

\def\Frac{\operatorname{Frac}}

\def\GL{\operatorname{GL}}

\def\Hom{\operatorname{Hom}}
\def\hyb{\operatorname{hyb}}

\def\ord{\operatorname{ord}}

\def\rank{\operatorname{rank}}

\def\red{\operatorname{red}}

\def\rr{\operatorname{rat.rank}}
\def\s{\operatorname{s}}

\def\sn{\operatorname{sn}}

\def\sp{\operatorname{sp}}
\def\Spec{\operatorname{Spec}}
\def\Spm{\operatorname{Spm}}

\def\trdeg{\operatorname{tr.deg}}
\def\triv{\operatorname{triv}}
\def\um{\operatorname{um}}

\def\ZR{\operatorname{ZR}}

\def\tilde{\widetilde}
\def\setminus{\smallsetminus}
\def\emptyset{\varnothing}

\addto\extrasenglish{}
\addto\extrasenglish{}
\addto\extrasenglish{\def\equationautorefname~#1\null{\textnormal{(#1)}\null}}
\declaretheorem[name=Theorem,refname={Theorem},style=plain,numberwithin=subsection]{theorem}
\declaretheorem[name=Theorem,refname={Theorem},style=plain,numbered=no]{theorem*}
\declaretheorem[name=Proposition-Definition,refname={Proposition-Definition},style=plain,sibling=theorem]{proposition-definition}
\declaretheorem[name=Proposition,refname={Proposition},style=plain,sibling=theorem]{proposition}
\declaretheorem[name=Proposition,refname={Proposition},style=plain,numbered=no]{proposition*}
\declaretheorem[name=Lemma,refname={Lemma},style=plain,sibling=theorem]{lemma}
\declaretheorem[name=Lemma,refname={Lemma},style=plain,numbered=no]{lemma*}
\declaretheorem[name=Definition,refname={Definition},style=definition,sibling=theorem]{definition}
\declaretheorem[name=Definition,refname={Definition},style=definition,numbered=no]{definition*}
\declaretheorem[name=Remark,refname={Remark},style=definition,sibling=theorem]{remark}
\declaretheorem[name=Remark,refname={Remark},style=remark,numbered=no]{remark*}
\declaretheorem[name=Corollary,refname={Corollary},style=plain,sibling=theorem]{corollary}
\declaretheorem[name=Corollary,refname={Corollary},style=plain,numbered=no]{corollary*}
\declaretheorem[name=Example,refname={Example},style=definition,sibling=theorem]{example}
\declaretheorem[name=Notation,refname={Notation},style=definition,sibling=theorem]{notation}

\declaretheorem[name=Claim,refname={Claim},style=plain,numbered=no]{claim*}
\declaretheorem[name=Théorème,refname={Théorème},style=plain,numbered=no]{theoremfr*}

\newtheorem{theo}{Theorem}

\newcommand{\va}{|\cdot|}

\newcommand{\fonction}[5]{\begin{array}{l|rcl}
#1: & #2 & \longrightarrow & #3 \\
    & #4 & \longmapsto & #5 \end{array}}

 \thanks{The author was supported by the collaborative research 
	center SFB 1085 \emph{Higher Invariants - Interactions between Arithmetic Geometry and Global Analysis} funded by the Deutsche Forschungsgemeinschaft.}

\begin{document}

\title{Pseudo-absolute values: foundations}
\date{\today}
\author{Antoine Sédillot}
\address{A. Sédillot, Mathematik, Universit{\"a}t 
		Regensburg, 93040 Regensburg, Germany}
	\email{antoine.sedillot@mathematik.uni-regensburg.de}

\begin{abstract}
In this article, we introduce pseudo-absolute values, which generalise usual absolute values. Roughly speaking, a pseudo-absolute value on a field $K$ is a map $|\cdot| : K \to [0,+\infty]$ satisfying axioms similar to those of usual absolute values. This notion allows to include "pathological" absolute values one can encounter trying to incorporate the analogy between Diophantine approximation and Nevanlinna theory in an Arakelov theoretic framework. It turns out that the space of all pseudo-absolute values can be endowed with a compact Hausdorff topology in a similar way as the Berkovich analytic spectrum of a Banach ring. Moreover, we introduce both local and global notions of analytic spaces over pseudo-valued fields and interpret them as analytic counterparts to Zariski-Riemann spaces.
\end{abstract}

%\keywords{Valuations, valued fields, Berkovich spaces} 
%\subjclass{{Primary 12J10, 13A18, 32P05; Secondary 14G40}}

\maketitle

\tableofcontents

\section*{Introduction}
%rework

\subsection*{Motivations and background}

\subsubsection*{Arakelov geometry, adelic curves} 

Let us start by stating the following guiding principle. Given a field $K$, together with a distinguished set of absolute values satisfying a product formula, one can perform Diophantine geometry. The first occurrence of this principle is the case of global fields. A turning point in this philosophy is due to Arakelov's seminal work \cite{Arakelov74}, which led to the development of many arithmetic analogues over number fields of classical tools in algebraic geometry. Moreover, several other instances of this principle have been studied in the literature (e.g. trivially valued fields, finitely generated extensions of $\bQ$, infinite algebraic extensions of $\bQ$).

In \cite{ChenMori},Chen and Moriwaki introduced an Arakelov theory over an arbitrary countable field. The central object of the theory is called an \emph{adelic curve}. Namely, an adelic curve is the data $S=(K,(\Omega,\nu),(\va_{\omega})_{\omega\in\Omega})$, where $K$ is a field, $(\Omega,\nu)$ is a measure space and $(\va_{\omega})_{\omega\in\Omega}$ is a family of absolute values on $K$. Moreover, an adelic curve $S$ is called \emph{proper} if a product formula holds (see \cite{ChenMori}, \S 3 for more details).

Adelic curves arise naturally in various number theoretic situations. In particular, any global field can be naturally equipped with an adelic structure. More generally, any countable field can be endowed with an adelic structure. Furthermore, adelic curves allow to study the examples cited above in a uniform manner.

Most tools arising in classical Arakelov geometry have a counterpart in the world of adelic curves: e.g. geometry of numbers, arithmetic intersection theory, Hilbert-Samuel formula \cite{ChenMori,ChenMori21,ChenMori24}.

%Considering the development of Berkovich's non-Archimedean analytic geometry \cite{Berko90,ChambertLoir06,CLD12,GK17}, it is natural to expect a "global analytic" approach to Arakelov geometry \cite{Paugam09,YuanZhang21}.

\subsubsection*{Nevanlinna theory and $M$-fields} 

Another instance of the guiding principle is encountered in the analogy between Diophantine geometry and Nevanlinna theory. This analogy was spotted first by Osgood \cite{Osgood81} and further explored by Vojta in \cite{Vojta87}. Roughly speaking, Nevanlinna theory can be seen as a generalisation to the fundamental theorem of algebra to entire functions. It builds on two fundamental theorem. Through the analogy, the first one corresponds to a suitable construction of a height function. The second one is seen as an analogue of Roth's theorem \cite{Roth55}. 

In \cite{Gubler97}, Gubler introduced the notion of \emph{M-fields}, with the idea of including Nevanlinna theory in an Arakelov theoretic framework. More precisely, an $M$\emph{-field} $K$ is the data of a field $K$ together with a measure space $(M,\nu)$ and maps defined $\nu$-almost everywhere 
\begin{align*}
(v\in M) \mapsto |a|_v \in\bR_{+} 
\end{align*}
for all $a\in K$ satisfying
\begin{itemize}
	\item[(i)] $|a+b|_v\leq |a|_v+|b|_v$ $\nu$-ae,
	\item[(ii)] $|a\cdot b|_v =|a|_v\cdot|b|_v$ $\nu$-ae,
	\item[(iii)] $v\mapsto \log|c|_v \in L^1(M,\nu)$ and $|0|_v=0$ $\nu$-ae,
\end{itemize} 
for all $a,b\in K$ and $c\in K^{\ast}$. This definition is notably motivated by the following example in Nevanlinna theory. Consider the field $\cM(\bC)$ of meromorphic complex functions. Consider a real number $R>0$ and set $M_R:=\{z\in \bC : |z|_{\infty}\leq R\}$ where the boundary $\{z\in \bC: |z|_{\infty}=R\}$ is equipped with the Lebesgue measure and the open disc $\{z\in \bC: |z|_{\infty}<R\}$ is equipped with a counting measure. For any $f\in \cM(\bC)$, consider the map
\begin{align*}
(z\in M_R) \mapsto  \left\{\begin{matrix}
|f(z)|_{\infty} \quad \text{if }|z|_{\infty}=R, \\
e^{-\ord(f,z)}\quad \text{if }|z|_{\infty}=R.
\end{matrix}\right.
\end{align*}
Note that the above map is well-defined everywhere except poles of $f$ on the circle of radius $R$. Then one can check that we have a $M_R$-field $\cM(\bC)$.

Using $M$-fields, Gubler obtains the construction of a height function for fields of arithmetic nature, including the above example and thus generalising Nevanlinna's first main theorem. 

\begin{remark*}
Since the notions introduced above will not be further used in the following presentation, we exposed them very succinctly. For the interested reader, a more detailed exposition can be found in the introduction of \cite{Sedillotthese}.
\end{remark*}

\subsection*{Goal}

Broadly speaking, the goal is to obtain a framework in the spirit of adelic curves allowing to both include Nevanlinna theory and study possibly uncountable fields of arithmetic interest. 

First of all, let us mention that in the case of $M$-fields, it seems to be a complicated problem to obtain further results, e.g. geometry of numbers. This is notably due to the fact that the "absolute values" appearing are not well-defined everywhere. Second, in the theory of adelic curves, the countability condition on the base field is imposed by the fact that the parameter space is a measure space and various construction for adelic vector bundles thus force this countability assumption (cf. \cite{ChenMori}, \S 4 and \S 6). Another hope is being able to add some analytic flavoured features in a similar spirit to the theory of global Berkovich spaces initiated by  \cite{Poineau07,Poineau10,Poineau13a,LemanissierPoineau24}.

To address these issues, it is aimed to work with a topological parameter space, also equipped with some analytic data. To this purpose, the idea is to authorise "absolute values" with singularities. These are the objects we introduce in this article. 

The content of this article is to be seen as the local aspect of my PhD thesis \cite{Sedillotthese}. The global counterpart (namely Chapters II-IV in \emph{loc. cit.}) will be introduced in a subsequent work. Since the tools we introduce next have an intrinsic interest, we decided to publish them in an independent paper. 

\subsection*{Content of the article}

In this article, we introduce the notion of pseudo-absolute value. Let $K$ be a field. A \emph{pseudo-absolute value} on $K$ is a map $|\cdot| : K \to [0,+\infty]$ satisfying 
\begin{itemize}
\item[(i)] $|1| = 1$ and $|0|=0$;
	\item[(ii)] for all $a,b\in K$, $|a+b| \leq |a|+|b|$;
	\item[(iii)] for all $a,b\in K$ such that $\{|a|,|b|\}\neq \{0,+\infty\}$, $|ab| = |a||b|$. 
\end{itemize}
This notion includes "pathological" absolute values that one encounters in the context of Diophantine geometry or Nevanlinna theory, e.g. maps of the form $(f\in \bQ(T)) \mapsto |f(0)|_{\infty} \in [0,+\infty]$ or $(f\in\cM(\bC)) \mapsto |f(z)|_{\infty}$, for $z\in \bC$. Moreover, if $|\cdot|$ is a pseudo-absolute value on a field $K$, then $A_{|\cdot|}:=\{a\in K : |a|<+\infty\}$ is a valuation ring of $K$ with maximal ideal is $\m_{|\cdot|}:=\{a\in A : |a|=0\}$ and $|\cdot|$ induces a usual absolute value on the residue field $A_{|\cdot|}/\m_{|\cdot|}$. In other words, pseudo-absolute values are the composition of a general valuation with a (possibly Archimedean) absolute value. 

The first appearance of pseudo-absolute values in the literature seems to be in (\cite{Weil51}, \S 9). What we call pseudo-absolute values are called \emph{absolute values}. In this seminal work, he gives the description of pseudo-absolute values as the composition of a valuation and a residue absolute value. 

In \cite{Temkin11}, Temkin studies Zariski-Riemann spaces in a relative context. The notion of \emph{semi-valuation} or \emph{Manis valuation} is central in the valuative interpretation of these Zariski-Riemann spaces. Pseudo-absolute values can be seen as variants of semi-valuations where the residue valuation is allowed to be Archimedean. In this article, we only limit ourselves to the classical case where the algebraic morphism is just the inclusion of the generic point of an integral projective scheme. Nonetheless, more general versions should be studied using this point of view. 

More recently, this notion has appeared in the context of Ben Yaacov-Hrushovski's framework of globally valued field. They obtain several related facts to those proved in that article (cf. the independent work \cite{GVF24} by Ben Yaacov-Destic-Hrushovski-Szachniewicz). More details will be given in the following of this introduction.

Using this interpretation, we obtain results on the behaviour of pseudo-absolute values with respect to extensions of the base field as well as some "Galois theoretic" results, namely we describe the action of the Galois group on sets of extensions of pseudo-absolute values (\S \ref{sec:algebraic_extension_of_sav}-\ref{sec:transcendental_extension_sav}). It is also possible to complete a field equipped with a pseudo-absolute value (\S \ref{sec:completion_sav}). Note that in general, the completion is not canonical. 

After that, we introduce the analogue of a normed vector space in the case where the base field is equipped with a pseudo-absolute value (\S \ref{sec:local_pseudo-norm}). The generalisation of a norm on a vector space is called a \emph{pseudo-norm} and the latter can be described as the data of a free module over the underlying valuation ring of the pseudo-absolute value together with a usual norm on the extension of scalars of this free module to the completion of the residue field. In this context, it is possible to generalise the usual algebraic constructions that are performed over normed vector spaces (e.g. subspaces, quotients, duals, tensor products, exterior products, direct sums). This step is necessary in view of developing slope theory in the global context (cf. \cite{Sedillotthese}, Chapter 3).

The core of this article consists in exploring the analytic geometry of spaces of pseudo-absolute values. First , note that pseudo-absolute values are intrinsically related to multiplicative semi-norm on a ring: indeed, if $|\cdot|$ is a pseudo-absolute value on a field $K$, then $|\cdot|$ induces a multiplicative semi-norm on the valuation ring $A_{|\cdot|}=\{a\in K : |a|<+\infty\}$. It is possible to equip the space of pseudo-absolute values with a topology similar to that of the Berkovich spectrum of a Banach ring \cite{Berko90}. It turns out that the set of all pseudo-absolute values on a field $K$ behaves as the analytic spectrum of the Zariski-Riemann space of $K$, namely the set of valuation rings of $K$, as suggested by the following result.

\begin{theo}
\label{th:pav_intro}
Let $K$ be a field with prime subring $k$. Denote by $M_{K}$ the set of all pseudo-absolute values on $K$. Then the topological space $M_K$ is non-empty, compact and Hausdorff. Moreover, we have a specification map $j : M_K \to \ZR(K/k)$ which is  continuous, where $\ZR(K/k)$ denotes the Zariski-Riemann space ok $K/k$ equipped with the Zariski topology.
\end{theo}
This will be proved in Theorem \ref{prop:compactness_set_sav}. This shows that the space $M_K$ is a choice of compactification of the set of usual absolute values over $K$ and plays the role of an analytic spectrum, the corresponding algebraic space being the Zariski-Riemann space of $K$. 

Then, in \S \ref{sec:local_analytic_spaces}, we introduce \emph{local analytic spaces}. We give two approaches. The first one is called the model approach. More precisely, let $v$ be a pseudo-absolute value on a field $K$ with valuation ring $A$ and residue field $\kappa$. Denote by $\widehat{\kappa}$ the completion of $\kappa$. Let $X$ be a projective $K$-scheme and let $\cX \to \Spec(A)$ be a projective model of $X$ over $A$. Namely $\cX$ is a projective scheme over $A$ given with an isomorphism between $X$ and the generic fibre of $\cX$. Then the \emph{model local analytic space} attached to this data is the analytification in the sense of Berkovich of the extension of scalars to $\widehat{\kappa}$ of the special fibre of $\cX$, namely $(\cX \otimes_{A} \widehat{\kappa})^{\an}$. The more general approach, i.e. without fixing a projective model over the valuation ring, illustrates the birational nature of pseudo-absolute values. In fact, it turns out that the relevant space to consider is the set of all pseudo-absolute values on a finite type extension of the base field $K$. This can be described as a projective limit of the local model analytic spaces over all the possible valuation rings of the extension. 

\begin{theo}
\label{th:local_intro}
Let $v$ be a pseudo-absolute value on $K$ defining a valuation ring $A$ with residue field $\kappa$. Let $K'/K$ be a finitely generated field extension. Let $M_{K',v}$ denote the set of all pseudo-absolute values on $K'$ extending $v$.
\begin{itemize}
	\item We have a continuous map 
	\begin{align*}
	M_{K',v} \to \ZR(K'/A).
	\end{align*}
	\item We have a commutative diagram of topological spaces
		\begin{center}
		\begin{tikzcd}[ampersand replacement=\&]
M_{K',v} \arrow[d] \arrow[r, "\cong"] \& \varprojlim_{\cX} (\cX \otimes_{A} \kappa)^{\an} \arrow[d] \\
\ZR(K'/A) \arrow[r, "\cong"]            \& \varprojlim_{\cX} \cX  \\                                 
\end{tikzcd},
		\end{center}
	where $\cX$ runs over all projective models of $K'/A$.
	\end{itemize}
\end{theo}

This will be proved in Theorem \ref{prop:local_analytic_space_projective_limit}. This result can be seen as an analytic counterpart to the usual algebraic description of Zariski-Riemann spaces, namely the bottom homeomorphism in the above diagram. It turns out that in the case where $v$ in a usual absolute value, Theorem \ref{th:local_intro} can be viewed as a birational counterpart to (\cite{Goto24}, Theorem 1.2). 

Finally, we introduce global analytic spaces. Similarly to the local case, we first introduce a "model" approach. Namely, we specify an integral model of the base field, which restricts the choice of relevant pseudo-absolute values. To do so, we introduce the notion of integral structure (\S \ref{sec:integral_structure}). Namely an \emph{integral structure} for a field $K$ is a Prüfer Banach ring $(A,\|\cdot\|)$ with fraction field $K$. It turns out that the Berkovich spectrum $\cM(A,\|\cdot\|_A)$ can be interpreted as a closed subset of $M_K$. This notion allows in some cases to describe explicitely some parts of spaces of pseudo-absolute values. The appearance of Prüfer domains in the theory is not surprising since they play an important role in the theory of Zariski-Riemann spaces. Namely they characterise the so-called \emph{affine subsets} of Zariski-Riemann spaces (cf. e.g. \cite{Olberding21}).

Using these integral structures, we introduce \emph{model global analytic spaces} can be introduced using the theory of global Berkovich spaces introduced in \cite{LemanissierPoineau24} (\S \ref{sec:global_analytic_spaces}). These spaces will be used notably in the implementation of Nevanlinna theory in the global construction. We also have a general notion of global analytic space which is linked with the model approach when we choose an integral structure on the base field. 

\begin{theo}
\label{th:global_intro}
Let $K'/K$ be a finitely generated field extension. Let $(A,\|\cdot\|_A)$ be an integral structure for $K$ such that $(A,\|\cdot\|)$ is a geometric base ring. Let $V'_A$ denote the set of all pseudo-absolute values on $K'$ restricting to an element of $\cM(A,\|\cdot\|_A)$.
\begin{itemize}
	\item We have a continuous map 
	\begin{align*}
	V'_A \to \ZR(K'/A).
	\end{align*}
	\item We have a commutative diagram
		\begin{center}
		\begin{tikzcd}[ampersand replacement=\&]
V'_A \arrow[d] \arrow[r, "\cong"] \& \varprojlim_{\cX} \cX^{\an} \arrow[d] \\
\ZR(K'/A) \arrow[r, "\cong"]            \& \varprojlim_{\cX} \cX  \\                                 
\end{tikzcd},
		\end{center}
	where $\cX$ runs over all projective, coherent models of $K'/A$.
	\end{itemize}
\end{theo}

This will be proved in Theorem \ref{th:link_global_analytic_spaces}. Note that this result recovers (\cite{GVF24}, Proposition 1.5) in the case where $K$ is a number field equipped with the integral structure from Example \ref{example:integral_structure} (2). In \cite{GVF24}, when the field $K$ is of characteristic zero, the authors obtain a description of the space $M_K$ as a projective limit of Berkovich analytifications of \emph{submodels} of $K$, i.e. schemes $\cX$ over $\bZ$ that are models of a finitely generated field contained in $K$. This approach is related to the technics used by Gubler in \cite{Gubler97}.

\subsection*{Acknowledgements}

We would like to thank Huayi Chen for his support and discussions during the elaboration of this paper. We also thank Keita Goto, Walter Gubler, Klaus Künnemann and Jérôme Poineau for numerous remarks and suggestions. Finally, we thank Micha\l{} Szachniewicz for pointing out \cite{Weil51} and that variants of Theorems \ref{th:pav_intro} and \ref{th:global_intro} were independently obtained in \cite{GVF24}.

\begin{comment}
We use the following notation for metrics on line bundles.
\begin{itemize}
	\item[(i)] If $L$ is a line bundle on $X$, we denote by $C^0(L)$ the class of continuous metrics on $L$.
	\item[(ii)] If $L,M$ are two line bundles on $X$, respectively endowed with metrics $\varphi,\psi$, the family $\varphi+\psi:=(\va_{\varphi+\psi}(x))_{x\in X^{\an}}$ defined by
	\begin{align*}
	\forall x\in X^{\an},\quad\va_{\varphi+\psi}(x) = \va_{\varphi}(x) \otimes \va_{\psi}(x),
	\end{align*}
is a metric on $L+M$. Likewise, the family $-\varphi$ defines a metric on $-L$, considering the corresponding dual norm family. If $(\varphi,\psi)\in C^0(L)\times C^0(M)$, resp. $\varphi \in C^0(L)$, then $\varphi + \psi \in C^0(L+M)$, resp. $-\varphi\in C^0(-L)$.
	\item[(iii)] Let $\varphi$ be a metric on $\cO_X$. We identify $\varphi$ with the function $-\ln |1|_{\varphi} : X^{\an} \to \bR$.
\end{itemize}
\end{comment}

\section*{Conventions and notation}

\begin{itemize}
	\item All rings considered in this article are commutative with unit.
	\item Let $A$ be a ring. We denote by $\Spm(A)$ the set of maximal ideals of $A$.
	\item By a local ring $(A,\m)$, we mean that $A$ is a local ring and $\m$ is its maximal ideal. In general, if $A$ is a local ring, the maximal ideal of $A$ is denoted by $\m_A$. 
	\item By an \emph{algebraic function field} $K/k$, we mean a finitely generated field extension $K/k$.   
	\item Let $A$ be a ring and let $X \to \Spec(A)$ be a scheme over $A$. Let $A\to B$ be an $A$-algebra. Then we denote $X \otimes_{A} B := X \times_{\Spec(A)} \Spec(B)$.
	\item Let $k$ be a field. We denote by $|\cdot|_{\triv}$ the trivial absolute value on $k$. If we have an embedding $k \hookrightarrow \bC$, we denote by $|\cdot|_{\infty}$ the restriction of the usual Archimedean absolute value on $\bC$.
	\item Let $(k,|\cdot|)$ be a valued field. Unless mentioned otherwise and when no confusion may arise, we will denote by $\widehat{k}$ the completion of $k$ w.r.t. $|\cdot|$.
	\item Throughout this article, unless specified otherwise, all valuations are considered up to equivalence. 
	\item We assume the \emph{axiom of universes}, which will allow taking inverse and direct limits over collections. 
\end{itemize}

\section{Preliminaries}
\label{sec:preliminaries}

\subsection{Valuation rings and Prüfer domains}
\label{sub:valuation_rings}

\subsubsection{Rank, rational rank, composite valuations and Gauss valuations}
\label{subsub:rank_rat_rank_etc}

Let $V$ be a valuation ring. We denote its \emph{value group} by $\Gamma_V:=\Frac(V)^{\times}/V^{\times}$. Let $V$ be a valuation ring with fraction field $K$ whose underlying valuation is denoted by $v: K \to \Gamma_V\cup\{\infty\}$. 
\begin{itemize}
	\item The \emph{rank} of $V$ is defined as the ordinal type of the totally ordered set of prime ideals in $V$ and is denoted by $\rank(V)$.
	\item The \emph{rational rank} of $V$ is defined by $\rr(V):= \dim_{\bQ}(\Gamma_V \otimes_{\bZ} \bQ)\in \bN\cup \{\infty\}$. In general we have $\rank(V) \leq \rr(V)$ (\cite{BouAC}, Chap. VI, \S10.2, Proposition 3).  
\end{itemize}

Let $V$ be a valuation ring with fraction field $K$, value group $\Gamma$ maximal ideal $\m$ and associated valuation $v : K \to \Gamma \cup \{\infty\}$. Let $\overline{v}$ be a valuation on the residue field $\kappa:=V/\m$. Then (\cite{Vaquie}, Proposition 1.12) implies that 
\begin{align*}
V' := \{a\in V : \overline{v}(\overline{a}) \geq 0\}
\end{align*}
is a valuation ring of $K$. The valuation attached to $V'$ is denoted by $v':= v \circ \overline{v}$. In that case, the residue field of $V'$ equals the residue field of $\overline{v}$. Moreover, we have the equalities
\begin{align*}
\rank(v') = \rank(v) + \rank(\overline{v}),\\
\rr(v') = \rr(v) + \rr (\overline{v}),
\end{align*}
as well as a short exact sequence of Abelian groups
\begin{align*}
0 \longrightarrow \overline{\Gamma} \longrightarrow \Gamma' \longrightarrow \Gamma \longrightarrow 0,
\end{align*}
where $\overline{\Gamma},\Gamma'$ denote respectively the value groups of $\overline{v},v'$. In the particular case where $\Gamma$ is order isomorphic to $\bZ^{n}$ for some integer $n\in \bN_{>0}$ (in that case $\rank(v)=n$), the the above short exact sequence splits and $\Gamma'$ is isomorphic to $\Gamma \times \overline{\Gamma}$ equipped with the lexicographic order. 

\begin{definition}
\label{def:composite_valuation}
With the same conditions as above, the valuation $v':= v \circ \overline{v}$ called the \emph{composite valuation} with $v$ and $\overline{v}$.
\end{definition}

\begin{example}
\label{example:composite_valuation}
Let $(V,\m)$ be a valuation ring with fraction field $K$, denote by $v$ the corresponding valuation. Assume that $V\subset V'$, where $(V',\m')$ is a valuation ring of $K$, denote by $v'$ the corresponding valuation. Let $\kappa,\kappa'$ denote respectively the residue fields of $V,V'$. Then $\m'\cap V$ is a prime ideal of $V$ and the quotient ring $\overline{V} := V/(\m'\cap V)$ is a valuation ring of $\kappa'$, denote by $\overline{v}$ the corresponding valuation. Then we have the equality $v = v' \circ \overline{v}$.
\end{example}

\begin{example}[Gauss valuations]
\label{example:Gauss_valuations}
Let $K$ be a field and $x$ be transcendental over $K$. Let $v$ be a valuation of $K$ with value group $\Gamma$ and residue field $\kappa$. Fix an extension of totally ordered Abelian groups $\iota : \Gamma \hookrightarrow \Gamma'$. Let $a\in K$ and $\gamma \in \Gamma'$. Let $P\in K[x]$ of degree $n$. Write 
\begin{align*}
P = \displaystyle\sum_{i=0}^{n} c_{i}(x-a)^i,
\end{align*}
where $c_0,...,c_n \in K$, and set
\begin{align*}
v_{a,\gamma}(P) := \displaystyle\min_{0\leq i\leq n} \{v(c_{i})+ i\gamma\}.
\end{align*}
For any $P/Q\in K(X)$, where $P,Q\in K[x]$ with $Q\neq 0$, set $v_{a,\gamma}(P/Q)=v_{a,\gamma}(P)-v_{a,\gamma}(Q)$. Then $v_{a,\gamma} : K(X) \to \Gamma + \bZ\gamma$ defines a valuation of $K(X)$ extending $v$. From (\cite{Kuhlmann04}, Lemma 3.10), if $\gamma$ is non-torsion modulo $\Gamma$, then $v_{a,\Gamma}$ has value group $\Gamma + \bZ\gamma$ and residue field $\kappa$. Otherwise, $v_{a,\gamma}$ has value group $\Gamma + \bZ\gamma$ and its residue field is a purely transcendental extension of transcendence degree $1$ of $\kappa$. Note that in both cases, $v_{a,\gamma}$ is Abhyankar. 
\end{example}

\subsubsection{Connection to algebraic geometry : specialisation on a scheme}

Valuation rings are of particular importance in algebraic geometry. An illustration of this fact is the following result.

\begin{proposition}[\cite{stacks}, \href{https://stacks.math.columbia.edu/tag/00PH}{Lemma 00PH}]
\label{prop:Krull-Akizuki}
Let $R$ be a Noetherian local domain which is not a field. Let $K:=\Frac(R)$. Let $L/K$ be a finitely generated extension. Then there exists a DVR $V$ with fraction field $L$ which dominates $R$. %modif Klaus
\end{proposition}

\begin{proposition}
\label{prop:DVR_specialisation}
Let $X$ be a locally Noetherian scheme and let $x,x'\in X$ be such that $x$ is a specialisation of $x'$, namely $x \in \overline{\{x'\}}$. Then there exist a discrete valuation ring $V$ and a morphism $\Spec(V) \to X$ such that the generic point of $\Spec(V)$ is mapped to $x'$ and the closed point of $\Spec(V)$ is mapped to $x$. Moreover, for any finitely generated extension $K/\kappa(x')$ we may choose $V$ such that the extension $\Frac(V)/\kappa(x')$ is isomorphic to the given extension $K/\kappa(x')$.
\end{proposition}

\begin{proof}
Let $K/\kappa(x')$ be a finitely generated extension, this induces ring morphisms $\cO_{X,x} \to \kappa(x') \to K$. Now Proposition \ref{prop:Krull-Akizuki} yields the existence of a DVR $V$ with fraction field $K$ which dominates $\cO_{X,x}$. Therefore, the morphism $\cO_{X,x} \to V$ yields the desired scheme morphism $\Spec(V) \to X$. 
\end{proof}

Roughly speaking, Proposition \ref{prop:DVR_specialisation} implies the fact that any specialisation of a locally Noetherian scheme can be encoded through a discrete valuation ring. In general, we have no information on the extension of residue fields of the special points. The following result gives a partial answer in this direction. 

\begin{proposition}
\label{prop:valuation_rings_specialisation_smooth}
Let $X$ be a locally Noetherian integral scheme and let $x\in X$ be a regular point. Then there exists a valuation ring $V$ dominating $\cO_{X,x}$ with residue field $\kappa$ such that $\Frac(R) = K(X)$ and $\kappa = \kappa(x)$.
\end{proposition}

\begin{proof}
Since $x$ is regular, there exist $a_1,...,a_r\in \cO_{X,x}$ whose images in $\m_x/\m_x^2$ are linearly independent over $\kappa(x)$ such that the maximal ideal $\m_x$ of $\cO_{X,x}$ is $(a_1,...,a_r)$. Moreover, for any $i=1,...,r$, the image of $a_i$ is an irreducible element of the regular ring $\cO_{X,x}/(a_1,...,a_{i-1})$ which is a UFD, and thus is a prime element of $\cO_{X,x}/(a_1,...,a_{i-1})$. This yields a prime divisor $D_i \subset \Spec(\cO_{X,x}/(a_1,...,a_{i-1}))$ and thus a discrete valuation $v_i$ of the field $\Frac(\cO_{X,x}/(a_1,...,a_{i-1}))$. We now define a valuation $v : \Frac(\cO_{X,x})\cong K(X) \to \bZ^{r}_{\mathrm{lex}}$ by sending any $f\in \cO_{X,x}$ to $(v_i(f \mod (a_1,...,a_{i-1})))_{1\leq i \leq r}$. Then the valuation ring $V$ of $v$ is a rank $r$ valuation ring of $K(X)$ dominating $\cO_{X,x}$ with residue field $\kappa(x)$.
\end{proof}

\begin{proposition}
\label{prop:valuation_rings_specialisation_non_smooth}
Let $X$ be a locally Noetherian integral scheme over a field $K$. Assume that there exists a proper birational morphism $\pi : X' \to X$ of $K$-schemes such that $X'$ is smooth. Then, for any $x\in X$, there exists a valuation ring $V$ dominating $\cO_{X,x}$ with residue field $\kappa$ such that $\Frac(R) = K(X)$ and $\kappa/\kappa(x)$ is finite.
\end{proposition}

\begin{proof}
We may assume that $x\in X$ is a closed point. Then, for any $x'\in \pi^{-1}(X')$, $x'$ is a closed regular point of $X'$ and Proposition \ref{prop:valuation_rings_specialisation_smooth} yields a valuation ring $V\subset K(X')\cong K(X)$ dominating $\cO_{X',x'}$ with residue field $\kappa(x')$. Since the extension $\kappa(x')/\kappa(x)$ is finite, we obtain the desired property for $V$.
\end{proof}

\begin{remark}
\label{rem:valuation_rings_specialisation_non_smooth}
In particular, Proposition \ref{prop:valuation_rings_specialisation_non_smooth} holds when the base field is of characteristic zero (cf. \cite{Hir64}).
\end{remark}

\subsubsection{Prüfer domains}
\label{subsub:Prufer_domains}

Let $A$ be an integral domain with fraction field $K$. $A$ is said to be a \emph{Prüfer domain} if, for any prime ideal $\p \in \Spec(A)$, the localisation $A_{\p}$ is a valuation ring. There are many characterisations of Prüfer domains (see e.g. \cite{FHP97}, Theorem 1.1.1).

\begin{proposition}
\label{prop:prop_Prufer_domains}
Let $A$ be a Prüfer domain with fraction field $K$. 
\begin{itemize}
	\item[(1)] Let $V$ be a valuation ring of $K$ containing $A$, and denote by $\m$ the maximal ideal. Then $\m\cap A$ is a prime ideal of $A$ and $V = A_{\m \cap A}$.
	\item[(2)] Let $L/K$ be an algebraic extension. Then the integral closure of $A$ in $L$ is a Prüfer domain (\cite{Fuchs01}, Chap. III, Theorem 1.2).
	\item[(3)] An $A$-module is flat if and only if it is torsion-free (\cite{BouAC}, Chapitre VII, \S 2, Exercices 12 et 14).
	\item[(4)] A finitely generated $A$-module is projective if and only if it is torsion-free (\cite{Fuchs01}, Chapter V, Theorem 2.7).
	\item[(5)] Assume that $A$ is Bézout. Then any projective $A$-module is free (\cite{Fuchs01}, Chapter VI, Theorem 1.11). 
	\item[(6)] Let $(A_i)_{i\in I}$ be direct system of Prüfer domains with injective arrows. Then $\varinjlim_{i\in I} A_i$ is a  Prüfer domain (\cite{Fuchs01}, Proposition 1.8).
\end{itemize}
\end{proposition}

\begin{proof}
By definition, $A_{\m \cap A} = \{a/b : a,b\in A \text{ and } b\notin \m\cap A\}$. Thus $A_{\m \cap A} \subset V_{\m}=V$. Since $(\m\cap A) \subset \m$, the inclusion $A_{\m \cap A} \to V$ is a local morphism of local rings whose fraction fields are $K$. From the fact that $A_{\m\cap A}$ is a valuation ring, we get $A_{\m\cap A} = V$.
\end{proof}

\begin{comment}
\begin{proposition}
\label{prop:overring_Prufer}
Let $A$ be a Prüfer domain with fraction field $K$. Let $A'$ be an overring of $A$, namely $A \subset A'$ and $A'$ is a subring of $K$. Then $A'$ is a Prüfer domain.
\end{proposition}

\begin{proof}
Let $A$ be an overring of $A$. Let $\p'$ be a prime ideal of $B$ and set $\p := \p' \cap A$, it is a prime ideal of $A$. Now the inclusion $A_{\p} \to A'_{\p'}$ is local $A_{\p}$ is a valuation ring of $K$, therefore $A'_{\p'}$ is a valuation ring of $K$ and $A'$ is Prüfer.
\end{proof}

\begin{proposition}
\label{prop:integral_closure_Prufer_ring}
Let $A$ be a Prüfer domain with fraction field $K$. 
\end{proposition}

\begin{proposition}[]
\label{prop:colimit_of_Prufer_rings}

\end{proposition}

\begin{proposition}
\label{prop:modules_over_Prufer_rings}
Let $A$ be a Prüfer domain. 
\begin{itemize}
	
\end{itemize}
\end{proposition}
\end{comment}

\begin{example}
\label{example:Prufer_rings}
\begin{itemize}
	\item[(1)] Any field is a Prüfer domain.	
	\item[(2)] Any Dedekind ring is Prüfer. Indeed, these are exactly the Noetherian Prüfer domains since they are locally discrete valuation rings, i.e. Noetherian valuation rings. In particular, the ring of integers of a number field is a Prüfer domain, and so is its absolute integral closure (cf. Proposition \ref{prop:prop_Prufer_domains} (2)). 
	\item[(3)] Let $X$ be a connected non-compact Riemann surface. Then the ring $\cO(X)$ of analytic functions on $X$ is a Prüfer domain (\cite{Royden56}, Proposition 1). Moreover, it is a Bézout domain. 
	\item[(4)] Let $C \subset U$ be a connected Stein subset of a connected non-compact Riemann surface $U$, namely $C$ has a basis of Stein open subset neighbourhoods in $U$. Denote $A=\cO(C)$ the ring of germs of holomorphic functions on $C$. It is an integral ring whose fraction field $K:=\cM(C)$ is the field of germs of meromorphic functions on $C$. Then $A$ is Prüfer. Indeed, for any Stein open set $C\subset U' \subset U$, $\cO(U')$ is a Prüfer domain and $A = \varinjlim \cO(U')$, where $U'$ runs over the open sets $U'$ such that $C\subset U'\subset U$. Hence Proposition \ref{prop:prop_Prufer_domains} (6) implies that $A$ is Prüfer.
\end{itemize}
\end{example}

\subsection{Application to rings of analytic functions}
\label{subsub:rings_of_analytic_functions}

In this subsection, we recall useful algebraic properties of rings of analytic functions that are studied throughout this article. %References are \cite{Helmer40,Henriksen52,Royden56,Alling63,Isssa66,Frisch67,Alling68,Siu69,Dales74,Alling79}.

\begin{proposition}
\label{prop:ring_of_analytic_functions_open_Riemann_surface}
Let $X$ be a non-compact Riemann surface.
\begin{itemize}
 \item[(1)] Any finitely generated ideal of $A:=\cO(X)$ is principal. Furthermore, any such ideal $I\subset
  A$ is prime iff it is maximal iff any generator of $I$ has exactly one zero, namely $I$ is the ideal of analytic functions vanishing at a point of $X$.
 \item[(2)] If $\m$ is a maximal ideal of $A$, then $\m$ is principal iff $\m$ is the kernel of a $\bC$-algebra morphism $\pi : A \to \bC$.
\end{itemize}
\end{proposition}

\begin{proof}
Let $I=(f_1,...,f_n)$ be a finitely generated ideal of $A$. Denote $d := \gcd(f_1,...,f_n) \in A$. Then there exist $e_1,...,e_n \in A$ such that $d= e_1f_1+\cdots+e_nf_n$ (cf. \cite{Royden56}, Proposition 1). Therefore $I=(d)$, i.e. $I$ is principal. Then (1) is exactly Proposition 2 of \cite{Royden56}. (2) is Proposition 3 of \cite{Royden56}. 
\end{proof}

\begin{proposition}
\label{prop:absolute_values_on_fields_of_meromorphic_functions}
Let $X$ be a non-compact Riemann surface. Denote by $A$ the ring of holomorphic functions and by $K$ the fraction field of $A$. Let $\va$ be an absolute value on $K$ such that $A\subset \{f\in K : |f|\leq 1\}$. Then $\va$ is either trivial or there exists $z\in X$ such that $\va$ is equivalent to an absolute value of the form $e^{-\ord(\cdot,z)}$.  
\end{proposition}

\begin{proof}
Let $\va$ be such an absolute value, it is necessarily non-Archimedean. Denote by $V$ the valuation ring of $V$ and by $\m$ the maximal ideal. The hypothesis ensures that $A\subset V$. Then $\p := \m \cap A \in \Spec(A)$ and $A_{\p}$ is a valuation ring of $K$ and the injection $A_{\p} \to V$ is local. Therefore, we have $A_{\p} = V$. 

Now assume that $\p$ does not contain any $\m_z$, for $z\in X$, and is not $(0)$. Let us show that any element $f$ in $\p$ has an infinite number of zeros. Assume that $f$ has a finite number of zeros $z_1,...,z_n$. By hypothesis on $\p$, there exist $f_1,...,f_n \in \p$ such that, for all $i=1,...,n$, $f_i\notin \m_{z_{i}}$. Thus $\gcd(f,f_1,...,f_n)=1 \in \p$, which gives a contradiction. Let $N = \{f\in A : f \text{ has a finite number of zeros}\}$, it is a multiplicative subset of $A$. From (\cite{Gilmer72}, \S 13, Exercise 21), we get that the localisation $A_N$ is a ring whose complete integral closure is $K$. By the above remark, we have an inclusion $A_N \subset A_{\p}$. Let $S$ denote the complete integral closure of $A_{\p}$. Then $K\subset S \subset K$ and thus $A_{\p}$ is not completely integrally closed. Then (\cite{Fuchs01}, Chapter II, Exercise 1.12) implies that $A_{\p}$ is a valuation of rank greater than $1$. This contradicts the fact that $A_{\p}$ is the valuation ring of an absolute value on $K$. 

Therefore either $\p$ is $(0)$ or there exists $z\in X$ such that $\m_z\subset \p$. In the first case, we get $V=A_{\p}=K$ and $\va$ is trivial. In the second case, let $z\in X$ such that $\m_{z}\subset \p$. Then we have an inclusion $A_{\m_{z}}\subset A_{\p}$ of rank $1$ valuation rings with fraction field $K$. Therefore $A_{\p}=A_{\m_{z}}$ and $\va$ is equivalent to the absolute value $e^{-\ord(\cdot,z)}.$
\end{proof}

We now describe prime ideals of $A=\cO(\bC)$, the ring of entire functions on $\bC$. For any entire function $f\in A$, denote by $Z(f)$ the set of zeroes of $f$. For any ideal $I \subset A$, if $\bigcap_{f\in I} Z(f) \neq \emptyset$, $I$ is called \emph{fixed}. Otherwise, the ideal $I$ is called \emph{free}. 

\begin{proposition}
\label{prop:holomorphic_functions_non-compact_RS_ideals}
Let $\p$ be a prime ideal of $A$.
\begin{itemize}
	\item[(1)] If $\p$ is fixed, it is maximal and of the form $\m_z := \{f\in A : f(z)=0\}$ for some $z\in\bC$.
	\item[(2)] If $\p$ is a free maximal ideal, then $A_{\p}$ is a valuation ring of rank at least $2^{\aleph_1}$.
	\item[(3)] If $\p$ is free, it is contained in a unique (free) maximal ideal $\m$ and 
	\begin{align*}
	\p^{\ast} := \displaystyle \bigcap_{n>0} \m^k
	\end{align*}
	is the largest non-maximal prime ideal contained in $\m$. Moreover, for any $f\in \p$, $f$ has an infinite number of zeroes.
	\item[(4)] If $\p$ is free and $\m$ is the unique maximal ideal containing $\p$, then $A/\p$ is a valuation ring whose maximal ideal $\m/\p$ is principal. 
	\item[(5)] Assume that $\p$ is free. Then $A/\p$ is a complete DVR iff $\p=\p^{\ast}$.
\end{itemize}
\end{proposition}

\begin{proof}
(1), (2) and (3) follow from (\cite{Henriksen52}, \S 3, Theorems 1-5). (4) is (\emph{loc. cit.}, Theorem 6) and (5) is (\emph{loc. cit.}, Theorems 7 and 8).
\end{proof}

\begin{remark}
\label{rem:holomorphic_functions_non-compact_RS_ideals}
More generally, by looking at the proof of Proposition \ref{prop:holomorphic_functions_non-compact_RS_ideals}, one can prove that the same conclusions as the above proposition hold by replacing $A$ by the ring of global analytic functions on a non-compact Riemann surface. 
\end{remark}

Let $X$ be a complex analytic space. Denote by $\cO_{X}$ its structure sheaf. Let $A\subset X$ be any subset. Then the space of germs of analytic functions on $A$ is
\begin{align*}
\Gamma(A,\cO_{X}) := \displaystyle\varinjlim_{U\supseteq A} \Gamma(U,\cO_{X}),
\end{align*}
where $U$ runs over the open subset of $X$ containing $A$. Conditions on $A$ to study algebraic properties of $\Gamma(A,\cO_{X})$ can be found in \cite{Frisch67,Alling68,Siu69,Dales74}. 

\begin{proposition}
\label{prop:holomorphic_functions_closed_disc_PID}
Let $R>0$ and let $\overline{D(R)}$ denote the closed disc of radius $R$ in $\bC$. Then the ring $\cO(\overline{D(R)})$ of germs of holomorphic functions on $\overline{D(R)}$ is a principal ideal domain (hence a Dedekind domain). 
\end{proposition}

\begin{proof}
The fact that $\cO(\overline{D(R)})$ is Noetherian is a consequence of (\cite{Frisch67}, Théorème I.9), see also (\cite{Siu69}, Theorem 1). $\cO(\overline{D(R)})$ is a unique factorisation domain by (\cite{Dales74}, Corollary to Theorem 1). Now Example \ref{example:Prufer_rings} (4) implies that $\cO(\overline{D(R)})$ is a Noetherian Prüfer domain, hence is Dedekind. Moreover, a Dedekind unique factorisation domain is a principal ideal domain.
\end{proof}

\begin{proposition}
\label{prop:holomorphic_functions_closed_disc_ideals}
Let $R>0$ and let $\overline{D(R)}$ denote the closed disc of radius $R$ in $\bC$. Then the maximal ideals of the ring $\cO(\overline{D(R)})$ are of the form $\m_z := \{f\in \cO(\overline{D(R)}) : f(z) = 0\}$, for some $z\in \overline{D(R)}$.
\end{proposition}

\begin{proof}
Denote $\cO(\overline{D(R)})$. For any $R'>R$, let $D(R')$ denote the open disc of radius $R'$ and $A_{R'}:=\cO(D(R'))$ the ring of holomorphic functions on $D(R')$. Then we have an isomorphism
\begin{align*}
\cO(\overline{D(R)}) \cong \displaystyle\varinjlim_{R'>R} A_{R'},
\end{align*}
where, for any $R<R''<R'$, we consider the inclusion $A_{R'} \subset A_{R''}$. It follows that we have an isomorphism of schemes
\begin{align*}
\Spec(\cO(\overline{D(R)})) \cong \displaystyle\varprojlim_{R'>R} \Spec(A_{R'}).
\end{align*}
Let $\p = (\p_{R'})_{R'>R} \in \Spec(\cO(\overline{D(R)}))$ be a non-zero prime ideal, namely, for any $R'>R$, $\p_{R'} \in \Spec(A_{R'})$ and, for any $R<R''<R'$, $\p_{R''}\cap A_{R'} = \p_{R'}$. Let us prove that, for any $R<R'$, $\p_{R'}$ is fixed. Assume that $\p_{R'}$ is free and non-zero for some $R<R'$. Let $f\in \p_{R'}\setminus\{0\}$. Then Proposition \ref{prop:holomorphic_functions_non-compact_RS_ideals} (3) together with Remark \ref{rem:holomorphic_functions_non-compact_RS_ideals} imply that $f$ has an infinite number of zeroes written as a sequence $(a_n)_{n\geq 0}$. Discreteness of $(a_n)_{n\geq 0}$ yields $|a_n| \to_{n\to+\infty} = R'$. Now for any $R<R''<R'$, $\p_{R''}$ is free. Thus the restriction of $f$ to $A_{R''}$ yields a non zero element of $\p_{R''}$, which has an infinite number of zeroes in $D(R'')$ thus accumulating at the boundary of $D(R'')$ which is included in the interior of $D(R')$. Hence we get a contradiction. Therefore, for any $R<R'$, $\p_{R'}$ is fixed and corresponds to some $\m_{z_{R'}} := \{f\in A_{R'} : f(z_{R'})=0\}$ for some $z_{R'}\in D(R')$. Since $(\p_{R'})_{R'>R}$ is a projective system, we obtain that there exits $z\in \overline{D(R)}$ such that, for any $R<R'$, we have $z_{R'}=z$. Conversely, for any $z\in \overline{D(R)}$, $\m_z$ is a maximal ideal of $\cO(\overline{D(R)})$.
\end{proof}

\subsection{Models over a Prüfer domain}
\label{subsub:models_over_Prüfer_domain}

Throughout this paragraph, we fix a Prüfer domain $A$ with fraction field $K$. 

Let $X \to \Spec(K)$ be a separated $K$-scheme of finite type. By a \emph{model} of $X$ over $A$, we mean a separated $A$-scheme $\cX\to \Spec(A)$ of finite type such that the generic fibre of $\cX$ is isomorphic to $X$. A model $\cX$ of $X$ over $A$ is respectively called \emph{projective}, \emph{flat}, \emph{coherent}, if the structural morphism $\cX \to \Spec(A)$ is projective, flat, of finite presentation. Note that if $\cX$ is a projective model of $X$, then $X$ is projective. 

If $A$ is a valuation ring with residue field $\kappa$, we denote by $\cX_{s}:=\cX \otimes_{A} \kappa$ the special fibre of $\cX$. In general, for any $y\in \Spec(A)$, we denote $\cX_{y} := \cX \otimes_{A}A_{y}$ and by $\cX_{y,s}$ the special fibre of $\cX_{y}$. 

Let $L$ be an invertible $\cO_{X}$-module. By a model $(\cX,\cL)$ of $(X,L)$ over $A$ we mean the data of a model $\cX$ of $X$ over $A$ together with an invertible $\cO_{\cX}$-module $\cL$ with generic fibre isomorphic to $L$. A model $(\cX,\cL)$ of $(X,L)$ over $\cL$ is respectively called flat, coherent, if the corresponding model $\cX$ of $X$ over $A$ is so.  If $A$ is a valuation ring, we denote by $\cL_s$ the restriction of $\cL$ to the special fibre $\cX_s$. In general, for any $y\in \Spec(A)$, we denote by $\cL_{y}$ the pullback of $\cL$ to $\cX_{y}$ and by $\cL_{y,s}$ the restriction of $\cL_y$ to the special fibre of $\cX_y$. 

\begin{example}
\label{example:model_line_bundle}
Let $L$ be a very ample line bundle on a projective $K$-scheme $X$. Denote by $\iota : X \hookrightarrow \mathbb{P}^{n}_{K}$ a corresponding closed immersion. Let $\cX$ denote the schematic closure of $X$ in $\mathbb{P}_{A}^{n}$. Then $\cX \to \Spec(A)$ is a model of $X$ over $A$ and the pullback $\cL$ of $\cO_{\mathbb{P}^{n}_{A}}$ on $\cX$ yields a model of $L$ over $\cX$, namely $(\cX,\cL)$ is a model of $(X,L)$. 
\end{example}

\begin{proposition} %TODO la preuve devrait marcher pour un anneau de Prüfer.
We assume that $A$ is a valuation ring with residue field $\kappa$. Let $X \to \Spec(K)$ be a projective $K$-scheme and $L$ be an invertible $\cO_{X}$-module. Let $(\cX,\cL)$ be a projective model of $(X,L)$ over $A$. 
\begin{itemize}
	\item[(1)] There exists a flat projective model $(\cX',\cL')$ of $(X,L)$ such that $\cL'=\cL_{|\cX'}$ and the special fibres of $\cX'$ and $\cX$ coincide.
	\item[(2)] There exists a coherent projective model $(\cX',\cL')$ of $(X,L)$ such that
	\begin{itemize}
		\item[(i)] $\cX$ is a closed subscheme of $\cX'$;
		\item[(ii)] the special fibres of $\cX'$ and $\cX$ coincide;
		\item[(iii)] $\cL'_{\cX}=\cL$.
	\end{itemize}
	\item[(3)] Assume that the restriction of $\cL$ to every fibre of $\cX\to \Spec(A)$ is ample. Then $\cL$ is ample.   
\end{itemize}
\end{proposition}

\begin{proof}
Let $\cO_{\cX,\mathrm{tors}}$ denote the torsion part of $\cO_{\cX}$ as an $\cO_{\Spec(A)}$-module. Then the closed subscheme $\cX'$ of $\cX$ defined by the ideal sheaf $\cO_{\cX,\mathrm{tors}}$ is a projective model of $X$ with special fibre $\cX' \times_{\Spec(A)} \Spec(\kappa) \cong \cX \times_{\Spec(A)} \Spec(\kappa)$. Moreover, the morphism $\cX' \to \Spec(A)$ is flat since $\cO_{\cX'}$ is torsion free (cf. Proposition \ref{prop:prop_Prufer_domains} (3)).  By setting $\cL'_{|\cX'}$, we conclude the proof of (1).

(2) is (\cite{ChenMori21}, Lemma 3.2.17), by noting that the proof does not use the fact that the rank of the valuation ring is less than $1$.

Let us prove (3). Let $(\cX',\cL')$ be a model such that conditions (i)-(iii) of (2) hold. Since $\cX' \to \Spec(A)$ is a proper and finitely presented and $\cL'$ is ample along the fibres, (\cite{EGAIV}, Corollaire (9.6.4)) gives the ampleness of $\cL'$. Thus $\cL = \cL'_{\cX}$ is ample.
\end{proof}

\begin{proposition}
\label{prop:models_Prüfer_domains}
Let $X \to \Spec(K)$ be a projective $K$-scheme and $L$ be an invertible $\cO_{X}$-module. Let $(\cX,\cL)$ be a projective model of $(X,L)$ over $A$. 
\begin{itemize}
	\item[(1)] We have an isomorphism $H^{0}(X,L) \cong H^{0}(\cX,\cL)\otimes_{A} K$. 
	\item[(2)] Assume that $(\cX,\cL)$ is a flat model. Then the following hold.
		\begin{itemize}
			\item[(i)] $H^{0}(\cX,\cL)$ is a flat $A$-module.
			\item[(ii)] Let $y\in \Spec(A)$, denote by $\cX_y$ the fibre of $\cX \to \Spec(A)$ over $p$ and by $\cL_y$ the restriction of $\cL_y$. Then we have an injection $H^0(\cX,\cL) \otimes_{A} \kappa \hookrightarrow H^{0}(\cX_y,\cL_y)$.
		\end{itemize}	
	\item[(3)] Assume that $(\cX,\cL)$ is a coherent model. Then $H^{0}(\cX,\cL)$ is a finitely generated $A$-module. 
	\item[(4)] Assume that $(\cX,\cL)$ is a flat and coherent model. Then $H^{0}(\cX,\cL)$ is a projective $A$-module of finite type. In particular, if $A$ is Bézout, then $H^{0}(\cX,\cL)$ is a free module of finite rank.
\end{itemize}
\end{proposition}

\begin{proof} 
\textbf{(1)} Since $A\to K$ is flat and $\cX\to \Spec(A)$ is qcqs, (1) follows from the flat base change theorem (\cite{Gortz10}, Corollary 12.8). 

\textbf{(2.i)} Since $\cX \to \Spec(A)$ is flat, we obtain that $H^{0}(\cX,\cL)$ is a torsion-free $A$-module. Now Proposition \ref{prop:prop_Prufer_domains} (3) implies that $H^{0}(\cX,\cL)$ is a flat $A$-module.

\textbf{(2.ii)} We first treat the case where $A$ is a valuation ring with maximal ideal $\m$ and residue field $\kappa$. As $f:\cX \to \Spec(A)$ is a flat, we have an exact sequence of $\cO_{X}$-modules 
\begin{center}
% https://tikzcd.yichuanshen.de/#N4Igdg9gJgpgziAXAbVABwnAlgFyxMJZABgBpiBdUkANwEMAbAVxiRGJAF9T1Nd9CKAIzkqtRizYAdKQFsuPEBmx4CRAEyjq9Zq0QgAggt4qBRAMxbxu6VIDWdNGjrGlfVYOQAWKzsn6OTjEYKABzeCJQADMAJwh5RDIQHAgkIW5ouISRZNTEdQyQWPikTVykc0LihMtyxC8gziA
\begin{tikzcd}
0 \arrow[r] & f^{\star}\m \otimes \cL \arrow[r] & \cL \arrow[r] & \cL_s \arrow[r] & 0
\end{tikzcd}
\end{center}
and thus an injection $H^{0}(\cX,\cL)/H^{0}(\cX,\cL \otimes f^{\ast}\m) \hookrightarrow H^{0}(\cX_s,\cL_s)$. Now $\m$ is a torsion-free $A$-module, hence is flat. By the projection formula (\cite{Gortz23}, Proposition 22.81), $H^{0}(\cX,\cL \otimes f^{\ast}\m) \cong \m H^{0}(\cX,\cL)$. As $H^{0}(\cX,\cL)$ is a flat $A$-module we have 
\begin{align*}
H^{0}(\cX,\cL) \otimes_{A} \kappa \cong H^{0}(\cX,\cL)/\m H^{0}(\cX,\cL) \hookrightarrow H^{0}(\cX_s,\cL_s).
\end{align*}
We finally treat the general case. Let $y\in \Spec(A)$ and denote by $\p_y$ the corresponding prime ideal of $A$. Then $A_{\p_y}$ is a valuation ring and $(\cX \otimes_{A} A_{\p_y},\cL \otimes_{A} A_{\p_y})$ is a flat model of $(X,L)$ over $A_{\p_y}$ whose special fibre coincides with $\cX_s$. As $A\to A_{\p_y}$ is flat, the previous case combined with the flat base change theorem yields
\begin{align*}
H^0(\cX,\cL) \otimes_{A} \kappa \cong (H^{0}(\cX,\cL) \otimes_{A} A_{\p_y}) \otimes_{A_{\p_y}} \kappa \cong H^{0}(\cX \otimes_{A} A_{\p_y},\cL \otimes_{A} A_{\p_y}) \otimes_{A_{\p_y}} \kappa \hookrightarrow H^{0}(\cX_y,\cL_y),
\end{align*}
which gives the conclusion.

\textbf{(3)} First mention that Prüfer domains and valuation domains are stably coherent, namely any polynomial algebra with finitely many indeterminates over a Prüfer domain is coherent (cf. \cite{Glaz89}, Theorem 7.3.3 and Corollary 7.3.4). Since $\cX \to \Spec(A)$ is projective and of finite presentation, (\cite{Ullrich95}, Theorem 3.5) implies that $H^{0}(\cX,\cL)$ is a finitely generated $A$-module. 

\textbf{(4)} Finally, (4) is a consequence of (2.i) and (3) together with Proposition \ref{prop:prop_Prufer_domains} (4)-(5).
\end{proof}

\subsection{Zariski-Riemann spaces}
\label{sub:Zariski-Riemann_space}

Let $K$ be a field and let $k$ be a subring of $K$ (we do not necessarily assume that $k$ is a domain with quotient field $K$). Define the \emph{Zariski-Riemann} space of $K/k$, denoted by $\ZR(K/k)$ as the set of valuation rings of $K$ containing $k$. The set $\ZR(K/k)$ is equipped with a topology. It is defined as follows. For any sub-$k$-algebra $A\subset K$, let $E(A)$ denote the set of valuation rings on $\ZR(K/k)$ containing $k$. Then the sets $E(A)$, where $A$ runs over the set of sub-$k$-algebra of finite type of $K$, form a basis for a topology on $\ZR(K/k)$, it is called the \emph{Zariski topology}. There is a \emph{center map} $\ZR(K/k) \to \Spec(k)$ sending any $V\in \ZR(K/k)$ to $\m_V \cap k$.

By a \emph{projective model} $X$ of $K/k$, we mean a projective integral $k$-scheme $X$ with fraction field $K(X)=K$. We define the \emph{domination relation} between projective models of $K/k$ as follows. Let $X,Y$ be two projective models of $K/k$. We say that $Y$ \emph{dominates} $X$ if there exists a birational $k$-morphism of schemes $Y \to X$. The collection of all projective models of $K/k$ form an inverse system with respect to the domination relation. By the valuative criterion of properness, for any $V\in \ZR(K/k)$, for any projective model $X$ of $K/k$, there exists a unique $\xi_V \in X$ such that $V$ dominates $\cO_{X,\xi_V}$. This defines a map $\ZR(K/k) \to X$ which is compatible with the domination relation. Hence we have a map $d : V\in\ZR(K/k) \to (d(V))_{X}\in\varprojlim_{X} X$, where $X$ runs over the projective models of $K/k$. Moreover, one can defined a structure sheaf on $\ZR(K/k)$. For any open subset $U \subset \ZR(K/k)$, define $\cO_{\ZR(K/k)}(U)$ to be the intersection of the valuation rings $V \in U$. This defines sheaf of rings on $\ZR(K/k)$ such that $(\ZR(K/k),\cO_{\ZR(K/k)})$ is a locally ringed space and the previously defined map $d$ is a morphism of locally ringed spaces.

\begin{theorem}[\cite{Olberding15ZR}, Proposition 3.3]
\label{th:Zariski-Riemann_projective_limite}
Assume that $k$ is a domain with fraction field $\Frac(k)$ and that the extension $K/\Frac(k)$ is finitely generated.
\begin{itemize}
\item[(1)] The domination map $d$ defined above is a homeomorphism.
\item[(2)] Let $V\in \ZR(K/k)$. Then $V$ is the union of the local rings $\cO_{X,d(V)_X}$, where $X$ runs over the collection of projective models of $K/k$. 
\item[(3)] The map $d$ defined above is an isomorphism of locally ringed spaces.
\end{itemize} 
\end{theorem}

\begin{proof}
(1) and (2) follow from the proof of (\cite{ZariskiSamuelII}, Chap. VI, \S 17, Theorem 41).

(3) is (\cite{Kahn08}, Theorem 3.2.5).
\end{proof}

\subsection{Berkovich spaces}

\subsubsection{Banach rings}

Recall that a \emph{Banach ring} is a pair $(A,\|\cdot\|_A)$, where $A$ is a ring and $\|\cdot\|_A$ is a norm on $A$ such that $
A$ is complete metrised space. 

\begin{example}
\begin{itemize}
	\item[(1)] Let $(K,\va)$ be a complete valued fields. Then $(K,\va)$ is a Banach ring.
	\item[(2)] Let $K$ be a field, let $\va$, be an absolute value on $K$. Denote by $\va_{\triv}$ the trivial absolute value on $K$. Then $\|\cdot\|_{\hyb} := \max\{\va,\va_{\triv}\}$ is a norm on $K$ such that $(K,\|\cdot\|_{\hyb})$ is a Banach ring. The norm $\|\cdot\|_{\hyb}$ is called the \emph{hybrid norm} associated to $\va$. 
	\item[(3)] The construction of (2) can be generalised to the case of a ring. Let $A$ be an arbitrary ring, let $\|\cdot\|$ be any norm on $A$. Denote by $\|\cdot\|_{\triv}$ the trivial norm on $A$, namely $\|0\|_{\triv}:=0$ and for any $a\in A \setminus\{0\}$, we have $\|a\|_{\triv}:=1$. Define $\|\cdot\|_{\hyb}:=\max\{\|\cdot\|,\|\cdot\|_{\triv}\}$. Then $\|\cdot\|_{\hyb}$ is a Banach norm on $A$. In this articles, such a norm is called a \emph{hybrid norm} and the corresponding normed ring is called a \emph{hybrid ring}.
\end{itemize}
\end{example}

\begin{definition}
\label{def:uniform_Banach_ring}
Let $(A,\|\cdot\|)$ be a Banach ring. We define the \emph{spectral semi-norm} on $A$ by
\begin{align*}
\fonction{\|\cdot\|_{\sn}}{A}{\bR_{\geq 0}}{f}{\inf_{k\in \bN^{\times}}\|f^k\|^{\frac{1}{k}}.}
\end{align*}
Theorem 1.3.1 of \cite{Berko90} yields, for all $f\in A$, the equality $\|f\|_{\sp} = \max_{x\in \cM(A)} |f(x)|$. The norm $\|\cdot\|$ is called \emph{uniform}, and that $(A,\|\cdot\|)$is a \emph{uniform} Banach ring, if $\|\cdot\|$ is equivalent to the spectral semi-norm. In that case, $\|\cdot\|_{\sp}$ is a norm and we have a homeomorphism
\begin{align*}
\cM(A,\|\cdot\|) \cong \cM(A,\|\cdot\|_{\sp})
\end{align*}
induced by the identity on $A$. In practice, unless explicitly mentioned, when a ring $A$ is equipped with a uniform norm, we always assume that the norm is the spectral norm.
\end{definition}

\subsubsection{Analytification in the sense of Berkovich: completely valued field case}

In this paragraph, we fix a completely valued field $(k,|\cdot|)$.

Let $X$ be a $k$-scheme. Its \emph{analytification in the sense of Berkovich}, denoted by $X^{\an}$, is defined in the following way: a point $x\in X^{\an}$ is the data $(p,\va_{x})$ where $p \in X$ and $\va_{x}$ is an absolute value on $\kappa(x)$ extending the absolute value on $k$. $X^{\an}$ can be endowed with the Zariski topology: namely the coarsest topology on $X^{\an}$ such that the first projection $j : X^{\an} \to X$ is continuous. There exists a finer topology on $X^{\an}$, called the Berkovich topology: it is the initial topology on $X^{\an}$ with respect to the family defined by $j : X^{\an} \to X$ and the applications
\begin{align*}
\fonction{|f|_{\cdot}}{U^{\an}}{\bR_{\geq 0},}{x}{|f|_x,}
\end{align*}
where $U^{\an}$ is of the form $U^{\an}:=j^{-1}(U)$, with $U$ a Zariski open subset of $X$, and $f\in \cO_{X}(U)$. Endowed with this topology, $X^{\an}$ is a locally compact topological space. There are GAGA type results: namely, $X$ is separated, resp. proper iff $X^{\an}$ is Hausdorff, resp. compact Hausdorff. If $X$ is a scheme of finite type, $X^{\an}$ can be endowed with a sheaf of analytic functions. 

Let $L$ be a line bundle on $X$. A \emph{metric} on $L$ is a family $\varphi :=(\va_{\varphi}(x))_{x\in X^{\an}}$, where 
\begin{align*}
\forall x\in X^{\an},\quad \va_{\varphi}(x) : L(x) := L \otimes_{\cO_X} \widehat{\kappa}(x) \to \bR_{\geq 0}
\end{align*}
is a norm on the $\widehat{\kappa}(x)$-vector space $L(x)$. The metric $\varphi$ is called \emph{continuous} if for all $U\subset X$ open and for all $s\in H^{0}(U,L)$, the map $|s|_{\varphi} : U^{\an} \to \bR_{\geq 0}$ is continuous with respect to the Berkovich topology. 

Let $\varphi$ be a continuous metric on a line bundle $L$. Let 
\begin{align*}
\forall s\in H^{0}(X,L),\quad \|s\|_{\varphi} := \displaystyle\sup_{x\in X^{\an}} |s|_{\varphi}(x) \in \bR_{\geq 0},
\end{align*}
it defines a seminorm on $H^{0}(X,L)$. Moreover, if $X$ is reduced then $\|\cdot\|_{\varphi}$ defines a norm on $H^{0}(X,L)$. In general (\cite{ChenMori}, Proposition 2.1.16) implies that if a section $s\in H^{0}(X,L)$ satisfies $\|s\|_{\varphi}=0$, then there exists an integer $n\geq 1$ such that $s^{\otimes n}=0$.

\subsubsection{Analytification in the sense of Berkovich: global case}
\label{subsub:global_analytification}

We now briefly recall the global counterpart of the last paragraph. The relevant class of base global analytic objects are the so-called \emph{geometric base rings} (cf. \cite{LemanissierPoineau24}, Définition 3.3.8 for more details). This class includes many usual examples of Banach rings studied in analytic geometry (e.g. rings of integers of number fields, hybrid fields, discretely valued Dedekind rings).

Let $(A,\|\cdot\|)$ be a Banach ring. In (\cite{LemanissierPoineau24}, Chapitre 2), the authors define the category of analytic spaces over $(A,\|\cdot\|)$, which is denoted by $(A,\|\cdot\|)$-an. 

\begin{theorem}[\cite{LemanissierPoineau24}, Corollaire 4.1.3, Lemme 6.5.1 and Proposition 6.5.3]
\label{th:global_analytification}
Let $(A,\|\cdot\|)$ be a geometric base ring. Let $\cX \to \Spec(A)$ be an $A$-scheme which is locally of finite presentation. Then the functor 
\begin{align*}
\fonction{\Phi_{\cX}}{(A,\|\cdot\|)\mathrm{-an}}{\mathrm{Sets}}{Y}{\Hom(Y,\cX)}
\end{align*}
is representable. The $(A,\|\cdot\|)$-analytic space which represents $\Phi_{\cX}$ is called the \emph{analytification} of $\cX$ and is denoted by $\cX^{\an}$. Moreover, if $\cX$ is projective, then $\cX^{\an}$ is a compact Hausdorff topological space.
\end{theorem}

%Section ok et cohérente.
\section{Pseudo-absolute values}
\label{sec:sav}

In this section, we introduce the main object of this paper: pseudo-absolute values.

\subsection{Definitions}

\begin{definition}
\label{def:semi-absolute_value}
Let $K$ be a field. A \emph{pseudo-absolute value} on $K$ is any map $\va : K \to [0,+\infty]$ such that the following conditions hold:
\begin{itemize}
	\item[(i)] $|1| = 1$ and $|0|=0$;
	\item[(ii)] for all $a,b\in K$, $|a+b| \leq |a|+|b|$;
	\item[(iii)] for all $a,b\in K$ such that $\{|a|,|b|\}\neq \{0,+\infty\}$, $|ab| = |a||b|$. 
\end{itemize}
The set of all pseudo-absolute values on $K$ is denoted by $M_K$.
\end{definition}

\begin{proposition}
\label{prop:sav_valuation_ring}
Let $\va$ be a pseudo-absolute value on a field $K$. Then $A_{\va} := \{a\in K : |a| \neq +\infty\}$ is a valuation ring of $K$ with maximal ideal $\m_{\va} := \{a\in A : |a|=0\}$. Further, $\va$ induces a multiplicative semi-norm on $A_{\va}$ with kernel $\m_{\va}$.
\end{proposition}

\begin{proof}
From $(i),(ii)$, $A_{\va}$ and $\m_{\va}$ are Abelian subgroups of $K$. From $(i),(iii)$, $A_{\va}$ is an integral subring of $K$ and $\m_{\va}\subset A_{\va}$ is an ideal. 

We show that $A_{\va}$ is a valuation ring of $K$. It is enough to treat the $A_{\va}\neq K$ case. Let $x\in K\setminus A_{\va}$. If $|x^{-1}|=+\infty$, $(iii)$ yield $1=|1|=|x\cdot x^{-1}| = +\infty$, contradicting $(i)$. Hence $|x^{-1}| \neq +\infty$, i.e. $x^{-1}\in A_{\va}$. 

We now prove that $\m_{\va}$ is the maximal ideal of $A_{\va}$. Let $\overline{a}\in A_{\va}/\m_{\va}\setminus \{0\}$ and fix a representative $a\in A_{\va}$ of $\overline{a}$. Then $a\in A_{\va}^{\times}$ and $a^{-1}$ is a representative of $\overline{a}$. Hence $\m_{\va}$ is the maximal ideal of $A_{\va}$. The last statement is a direct consequences of $(i)-(iii)$. 
\end{proof}

\begin{notation}
\label{notation:sav}
Let $K$ be a field. 
\begin{itemize}
	\item[(1)] Let $\va$ be a pseudo-absolute value on $K$. We call
	\begin{itemize}
		\item $A_{\va}$ the \emph{finiteness ring} of $\va$;
		\item $\m_{\va}$ the \emph{kernel} of $\va$;
		\item $\kappa_v := A_{\va}/\m_{\va}$ the \emph{residue field} of $v$.
		\item $v_{\va} : K \to \Gamma_{v_{\va}} \cup\{\infty\}$ the \emph{underlying valuation} of $\va$.
	\end{itemize}
	\item[(2)] By "let $(\va,A,\m,\kappa)$ be a pseudo-absolute value", we mean that $\va$ is a pseudo-absolute value on $K$ with finiteness ring $A$, kernel $\m$, residue field $\kappa$. 
	\item[(3)] By abuse of notation, by "let $v$ be a pseudo-absolute value" on $K$, we mean the pseudo-absolute value $(\va_v,A_v,\m_v,\kappa_v)$.
	\item[(4)] By abuse of notation, if $v$ is a pseudo-absolute value on $K$, we denote by $v:K \to \Gamma_{v} \cup\{\infty\}$ the underlying valuation of $\va_v$.
	\item[(5)] Let $v$ be a pseudo-absolute value on $K$. The map $\va_v : K \to [0,+\infty]$ is uniquely determined by the \emph{residue absolute value} on the residue field $\kappa_v$, namely the absolute value defined by $\tilde{|\overline{x}|_v} := |x|$ for all $\overline{x}\in \kappa_v$, where $x_v$ denotes any representative of $\overline{x}$ in $A_v$. By abuse of notation, we shall denote by $\va_v$ the residue absolute value. We denote by $\widehat{\kappa_v}$ the \emph{completed residue field}, namely the completion of $\kappa_v$ with respect to the residue absolute value
\end{itemize}
\end{notation}

\begin{remark}
\label{rem:characterisation_sav}
The construction of Proposition \ref{prop:sav_valuation_ring} can be reversed. Let $A$ be a valuation ring of a field $K$ with maximal ideal $\m$. Let $\va$ be a multiplicative semi-norm on $A$. Then $\va$ can be extended to $K$ by setting $|x| = \infty$ if $x\in K\setminus A$. Then $\va : K \to [0,+\infty]$ defines a pseudo-absolute value on $K$ with finiteness ring $A$ and kernel $\m$. 
\end{remark}

In the following, we shall implicitly use Remark \ref{rem:characterisation_sav} and we shall often describe a pseudo-absolute value by specifying the finiteness ring and the residue absolute value.

\begin{definition}
\label{def:properties_sav}
Let $v$ be a pseudo-absolute value on a field $K$. 
\begin{itemize}
	\item[(1)] The \emph{rank} of $v$ is defined as the rank of the finiteness ring $A_v$. It is denoted by $\rank(v)$.
	\item[(2)] Likewise, the \emph{rational rank} of $v$ is defined as the rational rank of the the finiteness ring $A_v$. It is denoted by $\rr(v)$.
	\item[(3)] $v$ is respectively called \emph{Archimedean}, \emph{non-Archimedean}, \emph{residually trivial} if the residue absolute value is Archimedean, non-Archimedean, trivial. 
\end{itemize}
\end{definition}

\begin{remark}
\label{rem:relation_sav_composite_valuation}
The notion of pseudo-absolute value is related the notion of composite valuations (cf. Definition \ref{def:composite_valuation}). Indeed, a non-Archimedean pseudo-absolute value on a field $K$ is nothing else that the data of a valuation $v$ of $K$ with a specified decomposition $v = v' \circ \overline{v}$, where $v'$ is a (general) valuation of $K$ and $\overline{v}$ is a rank one valuation of the residue field of $v'$. Likewise, an Archimedean pseudo-absolute value on $K$ is the "composition" of a valuation of $K$ with an Archimedean absolute value. Roughly speaking, a pseudo-absolute value can be seen as the composition of a valuation with a "real valuation".
\end{remark}

\subsection{Example of pseudo-absolute values}

\begin{example} 
\label{ex:sav}
\begin{itemize}
	\item[(1)] Any usual absolute value $\va$ on a field $K$ defines a pseudo-absolute value with finiteness ring $K$ and trivial underlying valuation.
	\item[(2)] Let $z\in \bP^{1}_{\bC}$ and $\epsilon \in ]0,1]$. Denote by $\va_{\infty}$ the usual absolute value on $\bC$. Then the map 
	\begin{align*}
	\fonction{\va_{z,\epsilon}}{\bC(T)}{[0,+\infty]}{P}{|P(z)|_{\infty}^{\epsilon}}
	\end{align*}
defines a pseudo-absolute value on $\bC(T)$ denoted by $v_{z,\epsilon,\infty}$, where $P(\infty) \in \bP^{1}_{\bC}$ denotes the evaluation in $0$ of $Q(T):=P(1/T)$. Its finiteness ring is $A_{z} := \{f\in \bC(T) : \ord(f,z) \geq 0\}$, its kernel is $\m_{z} := \{f\in \bC(T) : \ord(f,z) > 0\}$, and its residue field is $\bC$ endowed with the absolute value $\va_{\infty}^{\epsilon}$.
	\item[(3)] Let $K=\cM(U)$ be the field of meromorphic functions on a non-compact Riemann surface $U$ and denote by $\cO(U)$ the ring of analytic functions on $U$. Then $\cO(U)$ is a Prüfer domain (cf. \S \ref{subsub:Prufer_domains}). Let $z\in U$, then the localisation $A_z := \cO(U)_{\m_z}$ of $\cO(U)$ at the maximal ideal of functions on $U$ vanishing in  $z$ is a valuation ring of $K$ with residue field $\bC$. For all $\epsilon\in]0,1]$, let $v_{z,\epsilon,\ar} \in M_{K}$ be the pseudo-absolute value on $K$ with valuation ring $A_z$ and residue absolute value $\va_{\ar}^{\epsilon}$. 
	\item[(4)] Let $K=\bQ(T)$ and let $t\in [0,1]$ such that $e^{2i\pi t}\in \overline{\bQ}$. Denote by $\m_t$ the ideal of $\bQ[T]$ generated by the minimal polynomial of $e^{2i\pi t}$. For all $\epsilon\in]0,1]$, let $v_{t,\epsilon,\infty}$ be the pseudo-absolute value on $K$ defined by
\begin{align*}
\forall P\in\bQ[T]_{\m_t},\quad |P|_{t,\epsilon,\ar} := |P(e^{2\pi i t})|_{\infty},
\end{align*}
where $|\cdot|_{\infty}$ denotes the usual absolute value on $\bC$.
\end{itemize}
	\item[(5)] Let $K$ be a field and $(A,\m)$ be a valuation ring of $K$. Then there is a residually trivial pseudo-absolute value $v_{A,\triv}$ whose finiteness ring is $(A,\m)$ and the residue absolute value is the trivial absolute value on $A/\m$. Conversely, all residually trivial pseudo-absolute value arise this way. 
\end{example}

\subsection{Extension of pseudo-absolute values}

\begin{definition}
\label{def:extension_of_sav}
Let $K'/K$ be a field extension. Let $v=(\va,A,\m,\kappa)$, resp. $v'=(\va',A',\m',\kappa')$, be a pseudo-absolute value on $K$, resp. on $K'$. If $|x|' = |x|$ for all $x\in K$, we say that $\va'$ \emph{extends} (alternatively is \emph{above}) $\va$. In that case, we use the notation $v'|v$.
\end{definition} 

\begin{proposition}
\label{prop:extension_sav}
Let $K'/K$ be a field extension. Let $v=(\va,A,\m,\kappa)$, resp. $v'=(\va',A',\m',\kappa')$, be a pseudo-absolute value on $K$, resp. on $K'$. Assume that $v'|v$. Then
\begin{itemize}
	\item[(i)] $A'$ is an extension of $A$, namely there is an injective local morphism $A \to A'$; 
	\item[(ii)] the residue absolute value of $\va'$ extends the one of $\va$.
\end{itemize}
Conversely, given any extension $A\to A'$ of valuation rings (of $K$ and $K'$ respectively) endowed with multiplicative semi-norms with kernel the maximal ideal satisfying $(ii)$, the induced pseudo-absolute value on $K'$ extends the induced pseudo-absolute value on $K$. 
\end{proposition}

\begin{proof}
The first statement is a direct consequence of Definition \ref{def:extension_of_sav}. Let $v'=(\va',A',\m',\kappa')$, resp. $v=(\va,A,\m,\kappa)$, denote  the induced pseudo-absolute value on $K'$, resp. on $K$. Then $(ii)$ yields $|x|'=|x|$ for all $x\in A$. If $x\in K\setminus A$, then $|x|'= +\infty$ (otherwise $x \in A'\cap K = A$). Hence the conclusion.
\end{proof}

Let $\mathbf{SVF}$ be the category defined as follows. An object of $\mathbf{SVF}$ is a field endowed with a pseudo-absolute values and morphisms in $\mathbf{SVF}$ are given by the extensions of pseudo-absolute values. Let $\mathbf{VR}$ be the category with objects valuation rings endowed with multiplicative semi-norms with kernel the maximal ideal and with morphisms the extensions satisfying conditions $(i),(ii)$ of Proposition \ref{prop:extension_sav}.

\begin{proposition}
\label{prop:equivalence_extension_sav}
The categories $\mathbf{SVF}$ and $\mathbf{VR}$ are equivalent. 
\end{proposition}

\begin{proof}
It is a consequence of Remark \ref{rem:characterisation_sav} (2) combined with Proposition \ref{prop:extension_sav}.
\end{proof}

\begin{remark}
Proposition \ref{prop:equivalence_extension_sav} allows to safely treat extensions of pseudo-absolute values in the context of the category $\mathbf{VR}$\end{remark}

%Section ok et ref ok.
\section{Algebraic extensions of pseudo-absolute values}
\label{sec:algebraic_extension_of_sav}

In this section, we study extensions of pseudo-absolute values with respect to algebraic extensions of the base field. We first treat the separable case (\S \ref{sub:finite_separable_extension_sav}). Then we extend the results to arbitrary finite extensions (\S \ref{sub:finite_extension_sav}). Finally we introduce elementary Galois theory of pseudo-absolute values (\S \ref{sub:Galois_theory_sav}).

\subsection{Finite separable extension}
\label{sub:finite_separable_extension_sav}

In this subsection, we fix a finite separable extension $K'/K$ and a pseudo-absolute value $v$ on $K$. We study extensions of  $v$ to $L$. Let $A'$ be the integral closure of $A_v$ in $K$. (\cite{BouAC}, Chap. VI, \S 1.3, Théorème 3) implies that $A'$ is the intersection of all the extensions of $A_v$ to $L$. 

\begin{lemma}
$A'$ is a semi-local ring. 
\end{lemma}

\begin{proof}
Let $\m'$ be a maximal ideal $A'$. One the one hand, we have $\m' \cap A_v = \m$ and $A'$ is a  Prüfer domain (cf. Proposition \ref{prop:prop_Prufer_domains} (2)). Hence $A'_{\m}$ is an extension of $A_v$ to $L$. On the other hand, the set of extensions of $A_v$ to $L$ is finite. Whence $A'$ is semi-local.
\end{proof}

For any $\m_w$ in the fibre of $\m'_v$ of the morphism $\Spec(A') \to \Spec(A_v)$, let $A'_w := (A')_{\m_w}$ the localisation in $\m'_w$. From (\cite{BouAC}, Chap. VI, \S 8.6, Proposition 6), extensions of $A_v$ to $L$ are of the form  $A'_w$, for $w$ as above. For any such extension $A_v\to A'_w$, we have a finite extension $\kappa_v \rightarrow \kappa_w$ of residue fields.

\begin{proposition}
\label{prop:formula_extension_sav}
There is a bijective correspondence between the set of pseudo-absolute values on $L$ above $v$ and the set of extensions of the residue absolute value of $v$ with respect to extensions of the form $\kappa_v\to\kappa_w$, where $w$ runs over the set of maximal ideals of $A'$. Furthermore, we have the equality
\begin{align}
\label{eq:formula_extension_sav}
\displaystyle\sum_{\m_w\in \Spm(A')} \frac{1}{|\Spm(A')|} \sum_{i | v} \frac{[\widehat{\kappa_{w,i}}:\widehat{\kappa_v}]_s}{[\kappa_w:\kappa_v]_s} = 1,
\end{align}
where, for all $\m_w\in \Spm(A')$, $i$ runs over the set of extensions of the residue absolute value of $\va_v$ to $\kappa_w$ and $\widehat{\kappa_{w,i}}$ denotes the completion of $\kappa_w$ for any such absolute value. 
\end{proposition}

\begin{proof}
Remark \ref{rem:characterisation_sav} gives the first assertion. We first assume that, for all $\m_w\in\Spm(A')$, the extension $\kappa_v\to\kappa_w$ is separable. Then (\cite{BouAC}, Ch.V \S 8.5 Proposition 5) yields (\ref{eq:formula_extension_sav}). In the case where $\kappa_v\to \kappa_w$ is not separable, denote by $\kappa_v^{\s}$ the separable closure of $\kappa_v$ inside $\kappa_w$. For any $i | v$, (\cite{ChenMori}, Lemma 3.4.2) allows to identify the completion $\kappa_v^{\s}$ with the separable closure of $\widehat{\kappa_v}$ inside $\widehat{\kappa_w}_i$. Since we have an identification of the set of extensions of $v$ to $\kappa_w$ with the set of extensions of $\kappa_v$ to $\kappa_v^{\s}$, (\ref{eq:formula_extension_sav}) is obtained from the previous case.
\end{proof}

\subsection{Arbitrary finite extension}
\label{sub:finite_extension_sav}

In this subsection, we fix a finite extension $L/K$. Let $K'$ denote the separable closure of $K$ inside $L$. Then $L/K'$ is a purely inseparable finite extension and we denote by $q$ its degree. Therefore, for all $x\in L$, we have $x^q \in K'$. 

\begin{proposition}
\label{prop:extension_valuation_purely_inseparable}
Let $v'=(\va',A',\m',\kappa')$ be a pseudo-absolute value on $K'$. Then $A_L := \{x\in L : x^q\in A'\}$ is the unique valuation ring of $L$ extending $A'$. Its maximal ideal is $\m_L := \{x\in A_L : x^q\in \m'\}$. Moreover, the residue field extension $\kappa' \to A_L/\m_L$ is purely inseparable and finite.
\end{proposition}

\begin{proof}
Let $v' : K' \to \Gamma$ be the underlying valuation of $A'$. Then $(x\in L) \mapsto (1/q)v'(x^q)$ is the unique extension of $\nu$ to $L$. The corresponding valuation ring is $A_L = \{x\in L : x^q \in A'\}$ and its maximal ideal is $\m_L = \{x\in L : x^q\in \m'\}$. Let $\overline{a}\in A_L/\m_L$ and $a\in A_L$ be a representative. Then $a^q\in A'$ and represents $\overline{a}^q \in \kappa'$. Hence the conclusion.
\end{proof}

\begin{corollary}
\label{cor:purely_inseparable_extension_sav}
Let $v$ ba a pseudo-absolute value $K$. Then the set of extensions of $v$ to $L$ is in bijection with the set of extensions of $v$ on $K'$ described in Proposition \ref{prop:formula_extension_sav}.  
\end{corollary}

\subsection{Galois theory of pseudo-absolute values}
\label{sub:Galois_theory_sav}

Throughout this subsection, we fix a field $K$.

\begin{proposition}
\label{prop:def_galois_action_sav}
Let $L/K$ be an algebraic extension. For all $x \in M_L$, for all $\tau \in \Aut(L/K)$, the map
\begin{align*}
\fonction{\va_{\tau(x)}}{L}{[0,+\infty]}{a}{|\tau(a)|_{x}}
\end{align*}
defines a pseudo-absolute value on $L$ denoted by $x\circ \tau$. This construction defines a right action of $\Aut(L/K)$ on $M_L$. 
\end{proposition}

\begin{proof}
Let $x \in M_L$ and $\tau\in \Aut(L/K)$. Let us show that $\va_{\tau(x)}$ defines a pseudo-absolute value on $L$. Let $a,b\in L$, by linearity of $\tau$, we have
\begin{align*}
|a+b|_{\tau(x)}=|\tau(a)+\tau(b)|_x \leq |a|_{\tau(x)} + |b|_{\tau(x)}
\end{align*}
as well as $|0|_{\tau(x)}=0$ and $|1|_{\tau(x)}=1$. Assume that $\{|a|_{\tau(x)},|b|_{\tau(x)}\}\neq\{0,+\infty\}$. Then
\begin{align*}
|ab|_{\tau(x)}=|\tau(a)\tau(b)|_x = |\tau(a)|_{x}|\tau(b)|_{x}= |a|_{\tau(x)}|b|_{\tau(x)}.
\end{align*}
Hence $\va_{\tau(x)}\in M_L$.
\end{proof}

\begin{lemma}
\label{lemma:action_galois_extension_integral_closure}
Let $L/K$ be a (possibly infinite) Galois extension with Galois group $G$. Let $B$ denote the integral closure of $A$ in $L$. Then $G$ acts on $B$ and acts transitively on the set of maximal ideals of $B$.
\end{lemma}

\begin{proof}
First, $G$ acts on $B$ since $G$ stabilises $A$. Furthermore, we have inclusions $A \subset B^G \subset K = L^{G}$ and elements of $B^G$ are integral over $A$. Hence $A= B^G$ as $A$ is integrally closed. When $L/K$ is finite, (\cite{stacks}, \href{https://stacks.math.columbia.edu/tag/0BRI}{Lemma 0BRI}) gives the conclusion. In the general case, note that the set of maximal ideals of $B$ is in bijection with the set of extensions of $A$ to $L$. Hence any maximal ideal of $B$ is mapped to $\m$ via the morphism $\Spec(B)\to\Spec(A)$. We conclude by using (\cite{stacks}, \href{https://stacks.math.columbia.edu/tag/0BRK}{Lemma 0BRK}).
\end{proof}

We shall use the following crucial result.

\begin{proposition}[(\cite{Efrat06}, Corollary 15.2.5), (\cite{stacks}, \href{https://stacks.math.columbia.edu/tag/0BRK}{Lemma 0BRK}]
\label{prop:normal_residue_extension}
Let $L/K$ be a normal extension. Let $A'$ be an extension of $A$ on $L$ and denote by $\m'$ its maximal ideal and $\kappa'$ its residue field. Then the extension $\kappa'/\kappa$ is normal and the canonical homomorphism $\{\sigma \in \Aut(L/K) : \sigma(\m')=\m'\} \to \Aut(\kappa'/\kappa)$ is surjective.
\end{proposition}

\begin{proposition}
\label{prop:galois_action_transitive_sav}
Let $v=(|\cdot|,A,\m,\kappa)$ be a pseudo-absolute value on $K$. Assume that the residue field $\kappa$ is perfect. Let $L/K$ be a Galois extension with Galois group $G$. Denote by $M_{L,v}$ the set of extensions of $v$ to $L$. Then $G$ acts transitively on $M_{L,v}$. 
\end{proposition}

\begin{proof}
First, remark that $G$ acts on $M_{L,v}$ as $G$ stabilises $K$. Let $w_1 = (\va_1,A_1,\m_1),w_2=(\va_2,A_2,\m_2)\in M_{L,v}$ and denote by $B$  the integral closure of $A$ in $L$. Lemma \ref{lemma:action_galois_extension_integral_closure} gives the existence of $\sigma \in G$ such that $\sigma(\m_1\cap B) = \m_2\cap B$. Hence the map 
\begin{align*}
\fonction{\va_{\sigma}}{L}{[0,+\infty]}{a}{|\sigma(a)|_2}
\end{align*}
defines an element of $M_{L,v}$ with finiteness ring $A_1$ and residue absolute $\va_{\sigma}$. Proposition \ref{prop:normal_residue_extension} implies that the extension $\kappa_1:=(A_1/\m_1)/\kappa$ is Galois. Hence there exists $\overline{\tau} \in \Aut(\kappa_1/\kappa)$ such that, for all $\overline{a}\in \kappa_1$, $|\overline{a}|_1 = |\overline{a}|_{\sigma}$ (\cite{Neukirch99}, Chapter II, Proposition 9.1). Furthermore, there exists $\tau \in G$ lifting $\overline{\tau}$ and such that $\tau(\m_1) = \m_1$. We then have 
\begin{align*}
|a|_1 = |\sigma(\tau(a))|_2
\end{align*}
for all $a\in K$, i.e. $v_1 = \sigma\tau(v_2)$.
\end{proof}

The following results are an adaptation of (\cite{ChenMori}, Lemma 3.3.5) in the context of pseudo-absolute values. %They will be useful for constructing algebraic coverings of topological adelic curves. %TODO check s'il faut pas dire autre chose.

\begin{lemma}
\label{lemma:max_extension_absolute_values}
Let $F$ be a field endowed with an (usual) absolute value $\va$. We fix an extension of $\va$ to $\overline{F}$ again denoted by $\va$. Let $\{\alpha_1,...,\alpha_d\}$ be a family of pairwise distinct separable elements of $F$ and let $\{n_1,...,n_d\}$ be a family of non-zero integers. Then

\begin{align*}
\displaystyle\max_{j\in\{1,...,d\}} |\alpha_j| = \limsup_{N\to +\infty} \left|\sum_{j=1}^{d} n_j\alpha_j^N \right|^{\frac{1}{N}}.
\end{align*}
\end{lemma}

\begin{proof}
The proof essentially reproduces the arguments of the proof of (\cite{ChenMori}, Lemma 3.3.5). Let $N>0$ be an integer. The triangle inequality yields
\begin{align*}
    \displaystyle\left|\sum_{j=1}^{d} n_j\alpha_j^N \right| \leq d\max_{1\leq j\leq d} n_j \max_{j\in\{1,...,d\}} |\alpha_j|^N.
\end{align*}
Hence
\begin{align*}
    \limsup_{N\to +\infty} \left|\sum_{j=1}^{d} n_j\alpha_j^N \right|^{\frac{1}{N}} \leq \max_{j\in\{1,...,d\}} |\alpha_j|.
\end{align*}

Up to reordering, we may assume that
\begin{align*}
    |\alpha_1| = \cdots |\alpha_l| > |\alpha_{l+1}| \geq \cdots \geq |\alpha_d|,
\end{align*}
where $l\in\{1,...,d\}$. For all $j=1,...,l$, let $\beta_j := \alpha_j/\alpha_1$. Endow $(\widehat{\overline{F}})^l$, where $\widehat{\overline{F}}$ denotes the completion of the algebraic closure of $F$ with respect to $\va$, with the norm
\begin{align*}
    \|(a_1,...,a_l)\| := \displaystyle\max_{j=1,...,l} |a_j|.
\end{align*}
Then $\|(n_1\beta_1^N,...,n_l\beta_l^N)\|=\max_{j=1,...,l}|n_j|>0$ and
\begin{align*}
  \begin{pmatrix}
\sum_{j=1}^l n_j\beta_j^N \\
\vdots \\
\sum_{j=1}^l n_j\beta_j^{N+l-1} 
\end{pmatrix} = \begin{pmatrix}
1 & \dots & 1 \\
\beta_1 & \dots & \beta_l \\
\vdots & \ddots & \vdots \\
\beta_1^{l-1} & \dots & \beta_l^{l-1} 
\end{pmatrix} \begin{pmatrix}
n_1\beta_1^N \\
\vdots \\
n_l\beta_l^N 
\end{pmatrix}.
\end{align*}
As the $\beta_j$ are pairwise distinct, the above Vandermonde matrix is invertible and denote by $A$ its inverse. Then we have the inequality
\begin{align*}
    0<\frac{\max_{j=1,...,l}|n_j|}{\|A\|}\leq \|(\sum_{j=1}^l n_j\beta_j^N,\cdots,\sum_{j=1}^l n_j\beta_j^{N+l-1} )\|,
\end{align*}
where $\|A\|$ denotes the operator norm of $A$ with respect to $\|\cdot\|$. Hence the sequence $(|\sum_{j=1}^l n_j\beta_j^N|)_{N>0}$is bounded from below by a positive real number. Furthermore, we have 
\begin{align*}
    \displaystyle\lim_{N\to\infty} \left|\sum_{j=l+1}^{d} n_j(\frac{\alpha_j}{\alpha_1})^N \right| = 0. 
\end{align*}
Therefore, the sequence $(\left|\sum_{j=1}^{d} n_j(\frac{\alpha_j}{\alpha_1})^N \right|)_{N>0}$ does not converge to $0$. Hence
\begin{align*}
    \displaystyle\limsup_{N\to\infty} \left|\sum_{j=1}^{d} n_j(\frac{\alpha_j}{\alpha_1})^N \right|^{\frac{1}{N}} \geq 1.
\end{align*}
Finally we obtain
\begin{align*}
    \displaystyle\limsup_{N\to\infty} \left|\sum_{j=1}^{d} n_j\alpha_j^N \right|^{\frac{1}{N}} \geq |\alpha_1| = \max_{1\leq j\leq d} |\alpha_j|.
\end{align*}
\end{proof}

\begin{proposition}
\label{prop:max_sav_finite_extension}
Let $v$ be a pseudo-absolute value on $K$. Assume that $\kappa_v$ is perfect and trivially valued if $\mathrm{char}(\kappa_v)>0$. Let $L/K$ be a Galois extension. Let $c\in A_v^s$ (the integral closure of $A_v$ inside $K^s$). Let $P = T^{d}+a_1T^{d-1}+\cdots+a_d \in A_v[T]$ be the minimal polynomial of $c$ over $K$ and let $L$ be the field of decomposition of $P$. Let $O_c := \{\alpha_1,...,\alpha_d\}$ denote the orbit of $c$ under the action of the Galois group of $L/K$. We fix a choice $w_0:=(\va_0,B_0,\m_0,\kappa_0)$ of a pseudo-absolute value on $L$ above $v$. Then
\begin{align*}
\displaystyle\max_{1\leq j \leq d} |\alpha_j|_0 = \max_{w|v} |c|_{w} = \limsup_{N\to \infty}  \left|\sum_{j=1}^{d} \alpha_{j}^N \right|^{\frac{1}{N}}_v.
\end{align*}
\end{proposition}

\begin{proof}
As $L/K$ is Galois, Proposition \ref{prop:galois_action_transitive_sav} yields the first equality. We now prove the second equality. The $\mathrm{char}(k)>0$ case being trivial, we assume that $\mathrm{char}(k)=0$. Let $B$ denote the integral closure of $A_v$ in $L$. Then Lemma \ref{lemma:action_galois_extension_integral_closure} combined with the fundamental theorem of symmetric polynomials yield, for all $j\in\{1,...,d\}$, $\alpha_j \in B$ and $\sum_{j=1}^{d} \alpha_{j}^N\in A_v$. From \S \ref{sub:finite_separable_extension_sav}, $B$ is a Prüfer semi-local ring and denote by $\m_1,...,\m_r$ its maximal ideals and respectively $(B_1,\m_1,\kappa_1),...,(B_r,\m_r,\kappa_r)$ the corresponding localisations. For all $i\in\{1,...,r\}$, for all $j\in\{1,...,d\}$, denote by $\overline{\alpha_j}^{(i)}\in \kappa_i$ the image of $\alpha_j$ in $\kappa_i$. Then $\overline{\alpha_j}^{(i)}$ is a root of the minimal polynomial of $\overline{c}^{(i)}$ over $\kappa_v$, hence it is separable over $\kappa_v$. Hence, for all $i\in\{1,...,r\}$, for all integer $N>0$, the image of $\sum_{j=1}^{d} \alpha_{j}^N\in A_v$ in $\kappa_i$ is of the form 
\begin{align*}
    \sum_{j\in J_i} n_j(\overline{\alpha_{j}}^{(i)})^N,
\end{align*}
where $J_i\subset \{1,...,d\}$ is of cardinality $d_i:= [\kappa_i:\kappa_v]$ such that for all $j\neq j'\in J_i$, $\overline{\alpha_{j}}^{(i)} \neq \overline{\alpha_{j'}}^{(i)}$ and for all $j\in J$, $n_j$ is a non-zero integer in $\kappa_v$. Fixing an extension $\va_i$ of $\va_v$ to $\kappa_i$, we have
\begin{align*}
\displaystyle \max_{w\in E_i} |c|_w = \max_{j\in J_i} |\overline{\alpha_{j}}^{(i)}|_i = \limsup_{N\to \infty}  \left|\sum_{j\in J_i} n_j(\overline{\alpha_{j}}^{(i)})^N \right|^{\frac{1}{N}}_i = \left|\sum_{j\in J_i} n_j(\overline{\alpha_{j}}^{(i)})^N \right|^{\frac{1}{N}}_v = \limsup_{N\to \infty}  \left|\sum_{j=1}^{d} \alpha_{j}^N \right|^{\frac{1}{N}}_v,
\end{align*}
where $E_i$ denotes the set of extensions (in the sense of usual absolute values) on $\va_v$ to $\kappa_i$. The first equality comes from Proposition \ref{prop:galois_action_transitive_sav} (cf. $\kappa_i/\kappa_v$ is Galois). Lemma \ref{lemma:max_extension_absolute_values} provides the second equality. The fact that, for all integer $N>0$, $\sum_{j\in J_i} n_j(\overline{\alpha_{j}}^{(i)})^N\in\kappa_v$ (as a symmetric functions of roots of $\overline{P}^{(i)}$, the image of $P$ in $\kappa_v$) provides the third and last equalities. We can then conclude the proof using the description given in Proposition \ref{prop:formula_extension_sav}.
\end{proof}

%TODO a l'air ok
\section{Transcendental extensions of pseudo-absolute values}
\label{sec:transcendental_extension_sav}

Throughout this section, we fix a field $K$. We will address the problem of extending pseudo-absolute values on $K$ to transcendental extensions. In view of the study of transcendental extensions of valuations, this is a quite complicated problem (cf. e.g. \cite{MacLane36,Vaquie}). Therefore, we will mostly give specific, yet important, examples of such extensions. 

\subsection{Purely transcendental extension of degree 1 and usual absolute value case}

Throughout this subsection, let $K'=K(X)$, where $X$ is transcendental over $K$ and let $v$ be a usual absolute value on $K$. 

Let $v'=(\va',A',\m',\kappa')$ be an extension of $v$ to $K'$. Then by (\cite{BouAC}, Chap VI, \S 10.3, Corollaires 1-3), we have $\rank(v') \in \{0,1\}$. 

For the $\rank(v')=0$ case, $v'$ is an absolute value on $K'$ extending $v$. 

If $\rank(v)=1$, then (\cite{BouAC}, Chap VI, \S 10.3, Corollaires 1-3) yields $\trdeg(\kappa'/K)=0$ and $v'$ is Abhyankar, thus $A'$ is a discrete valuation ring and there exists a closed point $x\in \mathbb{P}^1_{K}$ such that $A'=\{f\in K': \ord(f,z)\geq 0\}$. The residue absolute value of $v'$ is an extension on $v$ to the algebraic extension $\kappa'/K$. 

From the above description, we get the following result.

%TODO homéo?
\begin{proposition}
\label{prop:description_transcendental_extension_degree_1}
Assume that $K$ is complete with respect to the absolute value $v$ and denote by $M_{K',v}$ the set of extensions of $v$ to $K'$. Then we have a homeomorphism.
\begin{align*}
M_{K',v} \cong \mathbb{P}^{1,\Berk}_{K},
\end{align*}
where $\mathbb{P}^{1,\Berk}_{K}$ denotes the Berkovich projective line over $K$.
\end{proposition}  

\begin{proof}
From what is written above, we have a bijection $M_{K',v}\to \mathbb{P}^{1,\Berk}_K$. It is continuous by definition of the topologies and thus we can conclude by compactness of $M_{K',v}$ and $\mathbb{P}^{1,\Berk}_K$.
\end{proof}

\subsection{Composition with a valuation}

Let $K'/K$ be a field extension. Let $\overline{v}=(\overline{\va},\overline{V},\overline{m},\kappa)$ be a pseudo-absolute value on $K$. Let $V'$ be a valuation ring with fraction field $K'$ with residue field $K$, denote by $v'$ the associated valuation of $K'$. For any $a\in V'$, denote by $\overline{a}$ the image of $a$ in $K$. Then $V:=\{a\in V': \overline{a}\in \overline{V}\}$. The results of \S \ref{subsub:rank_rat_rank_etc} imply that $V$ is a valuation ring of $K'$ with residue field $\kappa$. 

\begin{definition}
\label{def:composition_with_a_valuation}
Using the above construction together with Remark \ref{rem:characterisation_sav}, we obtain a pseudo-absolute value $v$ on $K'$ which is denoted by $v=v'\circ\overline{v}$. $v$ is called the \emph{composite pseudo-absolute value} with $v'$ and $\overline{v}$. 
\end{definition}

\begin{lemma}
\label{lemma:composition_with_a_valuation}
We use the above notation. Assume that $K$ is a subfield of $K'$. Then the pseudo-absolute value $v$ is an extension of $\overline{v}$ to $K'$.
\end{lemma}

\begin{proof}
Since $K'/K$ is an extension, we have a section $K \to V' \to K$. Thus, for any $a\in K$, we have $|a|_v = \overline{|a|}$.
\end{proof}

\subsection{Extension by generalisation}

Assume that $K'/K$ is an extension of algebraic functions fields, namely $K'/K$ is finitely generated. Let $X \to \Spec(K)$ be a $K$-variety with function field $K(X)\cong K'$, namely a model of $K'/K$. Let $x\in X$ be a non-singular point. Then Proposition \ref{prop:valuation_rings_specialisation_smooth} implies that there exists a valuation $v$ of $K'/K$ with valuation ring $V$ dominating $\cO_{X,x}$ and residue field isomorphic to $\kappa(x)$. Let $\overline{v_x}$ be a pseudo-absolute value on $\kappa(x)$. Using Definition \ref{def:composition_with_a_valuation}, we obtain a pseudo-absolute value $v_x = v \circ \overline{v_x}$ on $K'$. Moreover, if $\kappa(x)$ is a subfield of $K'$, Lemma \ref{lemma:composition_with_a_valuation} implies that $v_x$ is an extension of $\overline{v_x}$ to $K'$. Note that this is the case when $x$ is a regular rational point.

\begin{definition}
\label{def:extension_by_generalisation}
We use the above notation. The pseudo-absolute value $v_x$ is called the \emph{extension by generalisation} (alternatively the \emph{extension through specialisation}) of $\overline{v}_x$ to $K$ w.r.t. the valuation ring $V$. 
\end{definition} 

\section{Completion of pseudo-valued fields} 
\label{sec:completion_sav}

Let $K$ be a field and $v$ be a pseudo-absolute value on $K$ with residue field $\kappa$. The goal of this section is to construct a pseudo-absolute value (possibly on an extension of $K$) which extends $v$ and whose residue field is $\widehat{\kappa}$, the completion of $\kappa$ w.r.t. the residue absolute value. We make use of the notion of "gonflement" introduced by Bourbaki and which we recall (\S \ref{sub:gonflement}). Then we introduce (non-canonical) completion of a field with respect to a pseudo-absolute value (\S \ref{sub:completion_sav}).

\subsection{Gonflement of a local ring}
\label{sub:gonflement}

We will use results from (\cite{BouAC}, Chap. IX, Appendice 2) we now recall.

\begin{definition}
\label{def:gonflement_elementaire}
Let $(A,\m)$ be a local ring with residue field $\kappa$. An $A$-algebra $A'$ is called an \emph{elementary gonflement} (\emph{gonflement élémentaire} in French) of $A$ if either $A'$ is isomorphic to the $A$-algebra $A]X[:=A[X]_{\m A[X]}$ or there exists a monic polynomial $P\in A[X]$, whose reduction in $\kappa[X]$ is irreducible, such that $A'$ is isomorphic to the $A$-algebra $A[X]/(P)$.
\end{definition}

\begin{lemma}
\label{lemma:elementary_gonflement}
Let $A$ be a local ring and $\iota : A \to A'$ be an elementary gonflement of $A$. Then $A'$ is local. We denote respectively by $\kappa,\kappa'$ the residue fields of $A,A'$. Moreover the following assertion hold.
\begin{itemize}
	\item[(1)] The morphism $\iota$ is flat, injective and local.
	\item[(2)] The residue field extension $\kappa'/\kappa$ is generated by a single element.
	\item[(3)] The maximal ideal of $A$ generates the maximal ideal of $A'$.
	\item[(4)] If $A$ is Noetherian, then so is $A'$.
	\item[(5)] If $A$ is a valuation ring, then so is $A'$ and the value groups of $A$ and $A'$ are equal. In particular, if $A$ is a DVR then so is $A'$.
\end{itemize}
\end{lemma}

\begin{proof}
We first consider the case where $\iota$ is finite, i.e. $A' \cong A[X]/(P)$, where $P\in A[X]$ has irreducible reduction $\overline{P}$ in $\kappa[X]$. Let $\m$ denote the maximal ideal of $A$. Since $A'/\m A' \cong \kappa[X]/(\overline{P})$, $mA'$ is a maximal ideal of $A'$. Let $\p$ be a maximal ideal of $A'$. Since $\iota : A\to A'$ is finite, $\p\cap A = \m$ and thus $\m A' \subset \p$ hence $\p = \m A'$, i.e. $A'$ is local. The residue field extension is generated by the class of $X$ in $\kappa'$. We now prove (5). Assume that $A$ is a valuation ring and denote by $v : A \to \Gamma$ the corresponding valuation. We assume that the valuation is trivial (i.e. $A$ is not a field), otherwise the assertion is trivial itself.  Since $P\in A[X]$ has irreducible image in $\kappa[X]$, it is irreducible in $A[X]$. Indeed, if we could write $P=f\cdot g$,  where $f,g\in A[X]$ are non constant polynomials, then either $f$ or $g$ has coefficients in $\m$ and thus $P\in \m[X]$, which contradicts $P$ is monic since $A$ is not a field. Now $P$ is monic and irreducible in $A[X]$ and $A$ is integrally closed. Therefore $P$ is irreducible in $K[X]$, where $K$ denotes the fraction field of $A$. Thus $A'$ is an integral domain with quotient field $K':= K[X]/(P)$ which is a finite extension of $K$. Now define the map $v' : A[X] \to \Gamma$ sending $a_dX^d+\cdots a_0\in A[X]$ to $\min_{0\leq i\leq d}v(a_i)$. This defines a valuation on $K(X)$ and $A[X]$ is contained in the valuation ring $A]X[$ (cf. Example \ref{example:Gauss_valuations}). Since $P$ is monic, $v(P)=0$ and $v'$ induces a map $v' : A' \to \Gamma$. It is straightforward to check that $v'$ defines a valuation on $A'$, which is non-negative on $A'$. Therefore, if $V'$ denotes the valuation ring of $K'$ of the valuation $v'$, we have $A'\subset V'$. Assume that there exists $a'\in V'$ which does not belong to $A'$. Choose a representative $a_dX^d+\cdots a_0 \in K[X]$ of $a'$. By hypothesis, $v'(a')=\min_{0\leq i\leq d}v(a_i)\geq 0$. Since $a'\notin A'$, there exists an index $j\in \{0,...,d\}$ such that $v(a_j)=\min_{0\leq i\leq d}v(a_i)$ and $a_j\notin A$. Thus $0\leq v(a')=v(a_j) <0$, yielding a contradiction. Finally, $A'$ is a valuation ring. Using (4), we obtain the final part of (5). 

If $\iota$ is not finite, namely $A'\cong A]X[$, since $\m A[X]$ is a prime ideal in $A[X]$, $A'$ is local. Then (1)-(4) follow directly from the definition. If $A$ is a valuation ring, then $A]X[$ is the valuation ring of the Gauss valuation (cf. Example \ref{example:Gauss_valuations}). This concludes to proof.
\end{proof}

\begin{definition}
\label{def:gonflement}
Let $A$ be a local ring. An $A$-algebra $A'$ is called a \emph{gonflement} of $A$ if there exist a well-ordered set $\Lambda$ with a greatest element $\omega$ together with an increasing family $(A'_{\lambda})_{\lambda\in\Lambda}$ (w.r.t. inclusion) such that the following conditions hold:
\begin{itemize}
	\item[(i)] for any $\lambda\in\Lambda$, $A_{\lambda}$ is a local ring and $A'=A'_{\omega}$;
	\item[(ii)] let $\alpha$ be the least element of $\Lambda$, then $A'_{\alpha}$ is isomorphic to $A$;
	\item[(iii)] let $\nu\in \Lambda \setminus \{\alpha\}$, let $S_{\nu}:=\{\lambda \in \Lambda : \lambda < v\}$. If $S_{\nu}$ has a greatest element $\mu$, then $A'_v$ is an elementary gonflement of $A'_{\mu}$. Otherwise, then $A'_{\nu}= \displaystyle\bigcup_{\lambda\in S_{\nu}} A'_{\lambda}$. 
\end{itemize}
\end{definition}

\begin{proposition}
\label{prop:gonflement}
Let $A$ be a local ring and let $A \to A'$ be a gonflement.
\begin{itemize}
	\item[(1)] $A'$ is a local ring and $\m_{A'}=\m_{A}A'$.
	\item[(2)] The ring extension $A \to A'$ is faithfully flat.
	\item[(3)] Assume that $A$ is Noetherian. Then $A'$ is Noetherian and $\dim(A')=\dim(A)$. Moreover, if $A$ is regular, then $A'$ is regular.
	\item[(4)] If $A$ is a DVR, then $A'$ is a DVR. Moreover, the maximal ideal of $A'$ is generated by any generator of the maximal ideal of $A$.
	\item[(5)] If $A$ is a valuation ring, then $A'$ is a valuation ring. Moreover, the value groups of $A$ and $A'$ are equal. 
\end{itemize}
\end{proposition}

\begin{proof}
(1)-(3) are (\cite{BouAC}, Chap. IX, Appendice 2, Proposition 2 et Corollaire). (4) follows from (1) and (3) together with the fact that discrete valuation rings are exactly regular rings of dimension $1$.

Let us prove (5). Let $A$ be a valuation ring with value group $\Gamma$. Let $A'=(A'_{\lambda})_{\lambda\in\Lambda}$ be a gonflement of $A$, where $\Lambda$ is a well-ordered set. Denote respectively by $\alpha,\omega$ the least and greatest elements of $\Lambda$, namely $A=A_{\alpha}$ and $A'=A_{\omega}$. Let 
\begin{align*}
\Lambda' := \{\lambda\in \Lambda : A_{\lambda} \text{ is a valuation ring with value group }\Gamma\}.
\end{align*}
Assume that $\Lambda\setminus\Lambda'\neq\emptyset$. Since $\Lambda$ is well-ordered, there exists a least element $\nu \in \Lambda'$. Let $S_{\nu}:=\{\lambda\in\Lambda: \lambda<\nu\}$. Note that $\alpha<\nu$ and $S_{\nu} \subset\Lambda'$. 

If $S_{\nu}$ has a greatest element $\mu$, then $A_{\nu}$ is an elementary gonflement of $A_{\mu}$ and we obtain a contradiction using Lemma \ref{lemma:elementary_gonflement} (5). 

Now assume that $S_{\nu}$ does not have a greatest element. Then $A_{\nu}=\displaystyle\bigcup_{\lambda\in S_{\nu}} A'_{\lambda}$. Note that, for any $\lambda\leq \lambda'$, the morphism $A_{\lambda}\to A_{\lambda'}$ is injective and local, thus $A_{\nu}$ is a direct limit of Prüfer rings with injective arrows and is thus Prüfer (Prop \ref{prop:prop_Prufer_domains} (6)). Hence $A_{\nu}$ is a local Prüfer domain, i.e. a valuation ring. Since $S_{\nu}\subset \Lambda'$, we deduce that the value group of $A_{\nu}$ is $\Gamma$, providing a contradiction. 
\end{proof}

\begin{theorem}[\cite{BouAC}, Chap. IX, Appendice 2, Corollaire du Théorème 1]
\label{th:gonflement}
Let $A$ be a local ring with residue field $\kappa$. Let $\kappa'/\kappa$ be a field extension. Then there exists a gonflement $A \to B$ such that the field extension $\kappa'/\kappa$ is isomorphic to the field extension $\kappa_B/\kappa$, where $\kappa_B$ denotes the residue field of $B$. 
\end{theorem}

\subsection{Completion}
\label{sub:completion_sav}

\begin{definition}
\label{def:sav_complete}
Let $K$ be a field and $v=(\va,A,\m,\kappa)$ be a pseudo-absolute value on $K$. We say that $K$ is \emph{complete} w.r.t. $v$ if the following conditions hold.
\begin{itemize}
	\item[(i)] The finiteness ring $A'$ is Henselian.
	\item[(ii)] The residue field $\kappa'$ is complete w.r.t. the residue absolute value induced by $v'$.
\end{itemize}
\end{definition}

\begin{proposition}
Let $K$ be a field and $v=(\va,A,\m,\kappa)$ be a pseudo-absolute value on $K$. Assume that $K$ is complete w.r.t. $v$. Then for any algebraic extension $K'/K$, there exists a unique extension $v'|v$ to $K'$. 
\end{proposition}

\begin{proof}
Let $K'/K$ be an algebraic extension. Since $A$ is Henselian, underlying valuation of $v$ extends uniquely to $K'$, denote by $A'$ the corresponding valuation ring of $K'$. Let $\kappa'$ denote the residue field of $A'$. Then $\kappa'/\kappa$ is an algebraic extension and, since $\kappa$ is complete, the residue absolute value of $v$ extends uniquely to $\kappa'$. 
\end{proof}

\begin{definition}
\label{def:sav_complete_extension}
Let $K$ be a field and $v=(\va,A,\m,\kappa)$ be a pseudo-absolute value on $K$. Let $K'/K$ be field extension and $v'=(\va',A',\m',\kappa')\in M_{K'}$ be a pseudo-absolute value on $K'$ extending $v$. We say that the extension $v'|v$ is \emph{complete} if $K'$ is complete w.r.t. $v'$.
\end{definition}

The following proposition allows to construct (non-canonically) complete pseudo-valued fields. 

\begin{proposition}
\label{prop:completion_sav}
Let $K$ be a field and $v=(\va,A,\m,\kappa)$ be a pseudo-absolute value on $K$. Denote by $\widehat{\kappa}$ the completion of $\kappa$ w.r.t. the residue absolute value induced by $v$. Then there exist a field extension $\widehat{K}/K$ and a pseudo-absolute value $\widehat{v}\in M_{\widehat{K}}$ extending $v$ such that $\widehat{v}|v$ is complete (Definition \ref{def:sav_complete_extension}).
\end{proposition}

\begin{proof}
Theorem \ref{th:gonflement} implies that there exists a gonflement $A\to A'$ such that $A'$ has residue field $\widehat{\kappa}$. Moreover, Proposition \ref{prop:gonflement} (5) implies that $A'$ is a valuation ring with same value group as $A$. Now consider $\widehat{A}:=(A')^h$ the Henselisation of $A'$. Then (\cite{stacks}, \href{https://stacks.math.columbia.edu/tag/0BSK}{Section 0BSK}) implies that $\widehat{A}$ is a valuation ring extending $A'$ with same value group as $A'$ and $A$. Moreover, its residue field is $\widehat{\kappa}$. Let $\widehat{K}:=\Frac(\widehat{A})$. Then $\widehat{A}$ is an Henselian valuation ring of $\widehat{K}$ and by considering the unique extension of the residue absolute value induced by $v$ on $\kappa$ to $\widehat{\kappa}$, we obtain a pseudo-absolute value $\widehat{v}$ on $\widehat{K}$. The extension $\widehat{v}|v$ is complete by construction.  
\end{proof}

%Subsections 1 et 2 sont ok, voir si besoin de rajouter plus.
\section{Pseudo-norms}
\label{sec:local_pseudo-norm}
%TODO intro
In this section, we study the higher dimensional counterpart of pseudo-absolute values. Namely, we introduce the analogue of a normed vector space when the base field is equipped with a pseudo-absolute value. Throughout this section, we fix a field $K$.

\subsection{Definitions}

\begin{definition}
\label{def:local_semi_norm}
Let $E$ be a finite-dimensional vector space over $K$ of dimension $d$. For any pseudo-absolute value $v=(\va_v,A_v,\m_v,\kappa_v)\in M_{K}$, we call \emph{pseudo-norm} on $E$ in $v$ any map $\|\cdot\|_v : E \to [0,+\infty]$ such that the following conditions hold:
\begin{itemize}
	\item[(i)] $\|0\|_v = 0$ and there exists a basis $(e_1,...,e_d)$ of $E$ such that $\|e_1\|_v,\cdots,\|e_d\|_v \in \bR_{>0}$;
	\item[(ii)] for any $(\lambda,x)\in K\times E$ such that $\{|\lambda|_v,\|x\|_v\}\neq\{0,+\infty\}$, we have $\|\lambda x\|_v=|\lambda|_v\|x\|_v$;
	\item[(iii)] for any $x,y\in E$, $\|x+y\|_v\leq \|x\|_v+\|y\|_v$.
\end{itemize}
Under these conditions, $(E,\|\cdot\|_v)$ is called a \emph{pseudo-normed} vector space in $v$. Moreover, a basis satisfying condition (i) is called \emph{adapted} to $\|\cdot\|_v$.
\end{definition}

\begin{remark}
In \cite{Sedillotthese}, pseudo-norms are called \emph{local pseudo-norms}. Since in this article we only focus on local aspects, we decided to remove the qualification "local" to ease the reading. 
\end{remark}

\begin{proposition}
\label{prop:local_semi-norm}
Let $(E,\|\cdot\|_v)$ be pseudo-normed finite-dimensional $K$-vector space in $v\in M_K$. Let $d:=\dim_{K}(E)$ and let $(e_1,...,e_d)$ be a basis of $E$ which is adapted to $\|\cdot\|_v$. Then the following assertions hold.
\begin{itemize}
	\item[(1)] $\cE_{\|\cdot\|_{v}} := \{x\in E : \|x\|_v < +\infty\}$ is the restriction of scalars of $\bigoplus_{i=1}^{d}K\cdot e_i$ to $A_v$. Moreover, $\cE_{\|\cdot\|_{v}}$ is free $A_v$-module of rank $d$.
	\item[(2)] $N_{\|\cdot\|_{v}} := \{x\in E : \|x\|_v = 0\}$ is an $A_v$-submodule of $\cE_{\|\cdot\|_{v}}$ and we have the equality $N_{\|\cdot\|_{v}} = \m_v\cE_{\|\cdot\|_{v}}$.
	\item[(3)] $\|\cdot\|_{v}$ induces a norm on the $\widehat{\kappa_v}$-vector space 
	\begin{align*}
	\widehat{E}_{\|\cdot\|_{v}} := \cE_{\|\cdot\|_{v}} \otimes_{A_v} \widehat{\kappa_v} \cong \left(\cE_{\|\cdot\|_{v}}/\m_v\cE_{\|\cdot\|_{v}})\right)\otimes_{\kappa_v} \widehat{\kappa_v}.
	\end{align*}
\end{itemize}
\end{proposition}

\begin{proof}
(i) and (iii) imply that $\cE_{\|\cdot\|_{v}}$ is an Abelian group. (ii) and the fact that $A_v$ is the finiteness ring of $v$ imply that $A_v$ acts on $ \cE_{\|\cdot\|_{v}}$.
\begin{lemma}
\label{lemma:preuve_(1)}
$\cE_{\|\cdot\|_{v}}$ is an $A_v$-module of finite type.
\end{lemma}

\begin{proof}
Let us show by induction on $n$ that, for any $x\in \cE_{\|\cdot\|_{v}}$, there is no linear combination $x=x_1e_1+\cdots+x_de_d$, where $x_1,...,x_d\in K$, such that $|\{i\in\{1,...,d\} : x_i \in K\setminus A_v\}|=n$. We first assume that $n=1$. Let $x=x_1e_1+\cdots+x_de_d \in \cE_{\|\cdot\|_{v}}$, where $x_1,...,x_d\in K$, such that $|\{i\in\{1,...,d\} : x_i \in K\setminus A_v\}|=1$. By symmetry, we may assume that $x_d \notin A_v$. On the one hand, we have $\|x_de_d\|_v=+\infty$ and thus $x_de_d\notin \cE_{\|\cdot\|_{v}}$. On the other hand, we have 
\begin{align*}
x_de_d = x-x_1e_1-\cdots-x_{d-1}e_{d-1} \in \cE_{\|\cdot\|_{v}}.
\end{align*}
Hence a contradiction. Assume that $n>1$ and that the property is satisfied for $k=1,...,n-1$. Let $x=x_1e_1+\cdots+x_de_d \in \cE_{\|\cdot\|_{v}}$, where $x_1,...,x_d\in K$, such that $|\{i\in\{1,...,d\} : x_i \in K\setminus A_v\}|=n$. By symmetry, we may assume that $x_1,...,x_n\notin A_v$ and $x_{n+1},...,x_d\in A_v$. Since $x_n^{-1}\in A_v$, we have
\begin{align*}
x_n^{-1}x_1e_1 +\cdots + x_n^{-1}x_{n-1}e_{n-1} = x_n^{-1}x - e_n - x_n^{-1}x_{n+1}e_{n+1} - \cdots - x_n^{-1}x_de_d\in \cE_{\|\cdot\|_{v}}.
\end{align*}
The induction hypothesis yields a contradiction. Consequently, for any decomposition $x_1e_1+\cdots+x_de_d \in  \cE_{\|\cdot\|_{v}}$, we have $x_1,...,x_d\in A_v$. Hence $\cE_{\|\cdot\|_{v}}$ is of finite type.
\end{proof}
We now prove that $\cE_{\|\cdot\|_{v}}$ is torsion-free. Let $a\in A_v\setminus \{0\}$ and $x\in \cE_{\|\cdot\|_{v}}$ such that $ax = 0$. Since $A_v$ is integral, the inclusion $\cE_{\|\cdot\|_{v}} \subset E$ implies that $ax = 0$ in $E$, thus $x$ is torsion $E$ and $x=0$. Hence $\cE_{\|\cdot\|_{v}}$ is a torsion-free $A_v$-module of finite type, hence projective (cf. Proposition \ref{prop:prop_Prufer_domains} (4)). Since $A_v$ is a local ring, $\cE_{\|\cdot\|_{v}}$  is free (\cite{Matsumura86}, Theorem 2.5). We thus obtain the equality $\cE_{\|\cdot\|_{v}} = \bigoplus_{i=1}^d A\cdot e_i$, namely, $\cE_{\|\cdot\|_{v}}$ is the restriction of scalars of $E$ to $A_v$. This concludes the proof of (1).

We now prove (2). Using the same arguments as in the proof of (1), $N_{\|\cdot\|_{v}}$ is an $A_v$-submodule of $\cE_{\|\cdot\|_{v}}$. We clearly have an inclusion $\m_v\cE_{\|\cdot\|_{v}}\subset N_{\|\cdot\|_{v}}$. To show the inverse inclusion, we use the following lemma. 

\begin{lemma}
Let $x\in N_{\|\cdot\|_{v}}$. Then we can write $x$ under the form $x=x_1e_1+\cdots+x_de_d$, where $x_1,...,x_d\in \m_v$.
\end{lemma}

\begin{proof}
The proof is completely analogue to the one of Lemma \ref{lemma:preuve_(1)}. Let us show by induction on $n\geq 1$ that, for any $x\in N_{\|\cdot\|_{v}}$, there is no linear combination $x=x_1e_1+\cdots+x_de_d$, where $x_1,...,x_d\in A_v$, such that $|\{i\in\{1,...,d\} : x_i \notin \m_v\}|=n$. We first assume that $n=1$. Let $x=x_1e_1+\cdots+x_de_d \in \cE_{\|\cdot\|_{v}}$, where $x_1,...,x_d\in A_v$, such that $|\{i\in\{1,...,d\} : x_i \notin \m_v\}|=1$. By symmetry, we may assume that $x_d \notin \m_v$. On the one hand, we have $\|x_de_d\|_v \neq 0$, and thus $x_de_d\notin N_{\|\cdot\|_{v}}$. On the other hand, we have 
\begin{align*}
x_de_d = x-x_1e_1-\cdots-x_{d-1}e_{d-1} \in N_{\|\cdot\|_{v}}.
\end{align*}
Hence a contradiction. Assume that $n>1$ and that the property is satisfied for $k=1,...,n-1$. Let $x=x_1e_1+\cdots+x_de_d \in N_{\|\cdot\|_{v}}$, where $x_1,...,x_d\in A_v$, such that $|\{i\in\{1,...,d\} : x_i \notin \m_v\}|=n$. By symmetry, we may assume that $x_1,...,x_n \notin \m_v$ and $x_{n+1},...,x_d\in \m_v$. Since $x_n^{-1}\in A_v$, we have
\begin{align*}
x_n^{-1}x_1e_1 +\cdots + x_n^{-1}x_{n-1}e_{n-1} = x_n^{-1}x - e_n - x_n^{-1}x_{n+1}e_{n+1} - \cdots - x_n^{-1}x_de_d\in N_{\|\cdot\|_{v}}.
\end{align*}
The induction hypothesis yields a contradiction. Consequently, for any decomposition$x_1e_1+\cdots+x_de_d \in  N_{\|\cdot\|_{v}}$, we have $x_1,...,x_d\in \m_v$.
\end{proof}
Let $x=x_1e_1+\cdots+x_de_d \in  N_{\|\cdot\|_{v}}$, where $x_1,...,x_d\in \m_v$, let $\delta := \gcd(x_1,...,x_d)\in \m_v$ (cf. $A$ is Bézout). Then we have $x = \delta x'$, where $x'\in \cE_{\|\cdot\|_{v}}$. Hence $N_{\|\cdot\|_{v}} \subset \m_v\cE_{\|\cdot\|_{v}}$. This concludes the proof of (2).

Finally, by definition, $\widehat{E}_{\|\cdot\|_{v}}$ is the extension of scalars of $\cE_{\|\cdot\|_{v}}$ via the morphism $A_v \to \widehat{\kappa_v}$. Thus we have an isomorphism of $\widehat{\kappa_v}$-modules $\widehat{E}_{\|\cdot\|_{v}} \cong \left(\cE_{\|\cdot\|_{v}}/\m_v\cE_{\|\cdot\|_{v}}\right)\otimes_{\kappa_v} \widehat{\kappa_v}$. Furthermore, (1) and (2) imply that the restriction $\|\cdot\|_{v|A_v}$ induce a norm on $\widehat{E}_{\|\cdot\|{v}}$. 
\end{proof}

\begin{notation}
\label{notation:local_semi-norm}
Let $(E,\|\cdot\|_v)$ be pseudo-normed finite-dimensional $K$-vector space in $v=(\va_v,A_v,\m_v,\kappa_v)\in M_K$. 
\begin{itemize}
\item[(1)] In analogy with Notation \ref{notation:sav}, we call
\begin{itemize}
	\item $\cE_{\|\cdot\|_{v}} := \{x\in E : \|x\|_v < \infty\}$ the \emph{finiteness module} of $\|\cdot\|_v$;
	\item $N_{\|\cdot\|_{v}} := \{x\in E : \|x\|_v = 0\}$ the \emph{kernel} of $\|\cdot\|_v$;
	\item $\widehat{E}_{\|\cdot\|_v} := \cE_{\|\cdot\|_v} \otimes_{A_v} \widehat{\kappa_v}$ the \emph{residue vector space} of $\|\cdot\|_v$. By abuse of notation, we denote by $\|\cdot\|_v$ the induced norm on $\widehat{E}_{\|\cdot\|_v}$, called the \emph{residue norm} of $\|\cdot\|_v$.
\end{itemize} 
\item[(2)] By "let $(\|\cdot\|,\cE,N,\widehat{E})$ be a pseudo-norm on the $K$-vector space $E$ in $v$", we mean that $\|\cdot\|$ is a pseudo-norm on $K$ in $v$ with finiteness module $\cE$, kernel $N$ and residue vector space $\widehat{E}$.
\item[(3)] By abuse of notation, by "let $\|\cdot\|_v$ be a pseudo-norm on $E$ in $v$", we mean the pseudo-norm $(\|\cdot\|_v,\cE_v,N_v,\widehat{E}_v)$.
\end{itemize}
\end{notation}

In fact, in analogy with the case of pseudo-valued fields, a pseudo-normed vector space is determined by the objects defined in Notation \ref{notation:local_semi-norm}. More precisely, let
\begin{itemize}
	\item $v=(\va_v,A_v,\m_v,\kappa_v)\in M_{K}$;	
	\item $\cE_v$ be a free $A_v$-module of rank $d<+\infty$;
	\item $\|\cdot\|^{\wedge}_v$ be a norm on $\widehat{E}_v := \cE_v \otimes_{A_v} \widehat{\kappa}_v$.
\end{itemize}
Define the map
\begin{align*}
\fonction{\|\cdot\|_v}{\cE_v}{\bR_{\geq 0}}{a}{\|\overline{a}\|^{\wedge}_v}
\end{align*}
Then $N_v := \{x\in E : \|x\|_v = 0\} = \m_v\cE_v$. By lifting to $\cE_v$ a basis of $\widehat{E}_v$, we get a basis $(e_1,...,e_d)$ of $\cE_v$ such that $\|e_i\|_v > 0$ for all $i=1,...,d$. Then we can extend $\|\cdot\|_v$ to $E:= \cE_v \otimes_{A_v} K$ by setting $\|x\|_v = +\infty$ if $x\notin \cE_v$. 

\begin{proposition}
\label{prop:equivalence_local_semi-norm}
We use the same notation as above. Then $\|\cdot\|_v$ is a locally pseudo-norm in $v$ on $E$. Moreover, this construction is inverse to the one in Proposition \ref{prop:local_semi-norm}.
\end{proposition}

\begin{proof}
By construction of $\|\cdot\|_v$, we have $\|0\|_v=0$. We can see $(e_1,...,e_d)$ as a basis of $E$ such that $0<\|e_i\|_v< +\infty$ for all $i=1,...,d$. Hence $\|\cdot\|_v$ satisfies condition (i) of Definition \ref{def:local_semi_norm}. 

Let $(\lambda,x)\in K \times E$ such that $\{|\lambda|_v,\|x\|_v\}\neq\{0,+\infty\}$. We distinguish three cases. First, assume that $|\lambda|_v \neq +\infty \neq \|x\|_v$. Then by definition of $\|\cdot\|_v$ on $\cE_v$, we have $\|\lambda x\|_v=|\lambda|_v\|x\|_v$. If now $\|x\| = +\infty$, i.e. $x\in E\setminus\cE_v$, then $|\lambda|_v \neq 0$, i.e. $\lambda\in K\setminus \m_v$, hence $\lambda x \in E\setminus \cE_v$ and $\|\lambda x\|_v=+\infty = |\lambda|_v\|x\|_v$. Finally, if $|\lambda|_v = +\infty$, i.e. $\lambda\in K\setminus A_v$, then we have $\|x\|\neq 0$, i.e $x\in E\setminus \m_v\cE_v$, hence $\lambda x \in E\setminus \cE_v$ and $\|\lambda x\|_v=+\infty = |\lambda|_v\|x\|_v$. Thus $\|\cdot\|_{v}$ satisfies condition (ii) of Definition \ref{def:local_semi_norm}.

Since $\|\cdot\|^{\wedge}_v$ is a norm on $\widehat{E_v}$, for any $x,y\in \cE_v$, we have $\|x+y\|_v \leq \|x\|_v+\|y\|_v$. Then the triangle inequality for $\|\cdot\|_v$ on $E$ follows from the fact that $+\infty$ is the maximal element of $[0,+\infty]$.

From the construction, we directly see that it provides an inverse to the one in Proposition \ref{prop:local_semi-norm}. 
\end{proof}
 
\begin{remark}
Similarly to the case of pseudo-valued fields, we can use Proposition \ref{prop:equivalence_local_semi-norm} to define properties of pseudo-normed vector spaces from the corresponding property of the residue norm. For instance, a pseudo-norm on a vector space is called \emph{ultrametric} if so is its residue norm. 
\end{remark} 

To study restrictions and quotients for pseudo-norms, we will make use of the following lemma.

\begin{lemma}
\label{lemma:rescale}
Let $(E,(\|\cdot\|_v,\cE_v,N_v,\widehat{E}_v))$ be pseudo-normed finite-dimensional $K$-vector space in $v=(\va_v,A_v,\m_v,\kappa_v)\in M_K$. Let $(e_1,..,e_d)$ be a basis of $E$ such that, for all $i=1,..,d$, we have $e_i \in \cE_v\setminus N_v$, i.e. $\|e_i\|_v \in \bR_{>0}$. Let $(e'_1,...,e'_d)$ be another basis of $E$. Then, for all $i\in\{1,...,d\}$, there exists $\lambda_i\in K$ such that $\lambda_ie'_i \in \cE_v\setminus N_v$.
\end{lemma}

\begin{proof}
We fix $i\in\{1,...,d\}$. Then there exist $a_1^{(i)},...,a_d^{(i)}\in K$ such that $e'_i = \sum_{j=1}^d a_j^{(i)} e_j$. By symmetry, we may assume that $(a_1^{(i)})^{-1}a_j^{(i)} \in A_v$ for all $j=1,...,d$. Then, for all $j=1,...,d$, we have $(a_1^{(i)})^{-1} e'_i \in \cE_v$. Furthermore, writing
\begin{align*}
(a_1^{(i)})^{-1} e'_i = e_1 + (a_1^{(i)})^{-1}a_{2}^{(i)} e_2 + \cdots + (a_1^{(i)})^{-1}a_{d}^{(i)} e_d,
\end{align*}  
we deduce $(a_1^{(i)})^{-1} e'_i \notin N_v = \m_v\cE_v$.
\end{proof}

\begin{remark}
Lemma \ref{lemma:rescale} shows that, up to scaling, condition (i) in Definition \ref{def:local_semi_norm} is independent on the choice of the basis $(e_1,...,e_d)$.
\end{remark}

\subsection{Algebraic constructions for pseudo-norms}

We now extend the usual algebraic construction for normed vector spaces in our context (cf. e.g. \cite{ChenMori}, \S 1.1 for more details).

\begin{definition}
\label{def:constructions_local_semi-norms}
\begin{itemize}
	\item[(1)] Let $(E,\|\cdot\|_v)$ be a pseudo-normed finite-dimensional $K$-vector space over $K$ in $v\in M_K$. Let $F\subset E$ be a vector subspace of $E$. Let $(e_1,...,e_r)$ be a basis of $F$ enlarged in a basis $(e_1,...,e_r,e_{r+1},...,e_d)$ of $E$. Lemma \ref{lemma:rescale} ensures that we may assume that, for all $i=1,...,d$, we have $e_i \in \cE_v \setminus \m_v$. Then the restriction restriction of $\|\cdot\|_v$ to $F$ is a pseudo-norm in $v$ on $F$ called the \emph{restriction} of $\|\cdot\|_v$ to $F$. By abuse of notation, unless explicitly mentioned, we denote this pseudo-norm by $\|\cdot\|_v$. Moreover, the finiteness module of $\|\cdot\|_v$ (the pseudo-norm on $F$), denoted by $\cF_v$, has $(e_1,...,e_r)$ as a basis.
	
	\item[(2)] Let $(E,\|\cdot\|_v)$ be a pseudo-normed finite-dimensional $K$-vector space in $v\in M_K$. Let $F\subset E$ be a vector subspace of $E$. Let $(e_1,...,e_r)$ be a basis of $F$ enlarged in a basis $(e_1,...,e_r,e_{r+1},...,e_d)$ of $E$. Lemma \ref{lemma:rescale} ensures that we may assume that, for all $i=1,...,d$, we have $e_i \in \cE_v \setminus \m_v$. Using the same notation as in (1), $\cE_v/\cF_v$ is a free $A_v$-module  of rank $d-r$ and is spanned by $e_{r+1},...,e_d$. On the $\widehat{\kappa_v}$-quotient vector space $\widehat{E_v}/\widehat{F_v}$, we consider the quotient norm induced by the residue norm of $\|\cdot\|_v$. Then Proposition \ref{prop:equivalence_local_semi-norm} yields a pseudo-norm in $v$ on $E/F$, called the \emph{quotient} pseudo-norm  and denoted by $\|\cdot\|_{E/F,v}$. Moreover, we have
	\begin{align*}
	\forall \overline{x}\in E/F, \quad \|\overline{x}\|_{E/F,v} = \displaystyle\inf_{x\in \pi^{-1}(\{\overline{x}\})} \|x\|_v,
	\end{align*}
where $\pi : E \to E/F$ denotes the canonical quotient map.
	
	\item[(3)] Let $(E,\|\cdot\|_v)$ be a pseudo-normed finite-dimensional $K$-vector space in $v\in M_K$. Le $\cE_v^{\vee} := \Hom_{A_v}(\cE_v,A_v)$, it is a free $A_v$-module of rank $\dim_{K}(E)$. We denote by $\|\cdot\|_{v,\ast}^{\wedge}$ the dual norm on $(\widehat{E}_v)^{\vee}$ of the residue norm $\|\cdot\|_v^{\wedge}$ of $\|\cdot\|_v$. Then Proposition \ref{prop:equivalence_local_semi-norm} yields a pseudo-norm in $v$ on $E^{\vee} \cong \cE_v^{\vee} \otimes_{A_v} K$, called the \emph{dual} pseudo-norm of $\|\cdot\|_v$ and denoted by $\|\cdot\|_{v,\ast}$. Furthermore, we have 
	\begin{align*}
	\forall \varphi \in E^{\vee},\quad \|\varphi\|_{v,\ast} = \displaystyle\sup_{x\in\cE_v\setminus\m_v\cE_v} \frac{|\varphi(x)|_v}{\|x\|_v}.
	\end{align*}
 	\item[(4)] Let $(E,\|\cdot\|_v)$ and  $(E',\|\cdot\|'_v)$ be two pseudo-normed finite-dimensional $K$-vector spaces in $v=\in M_K$. The data of $\cE_v \otimes_{A_v} \cE'_v$ and respectively of the $\pi$-tensor product and the $\epsilon$-tensor product norms on $\widehat{E}_v\otimes_{\widehat{\kappa}_v}\widehat{E'}_v$ induce pseudo-norms in $v$ on $E\otimes_K E'$ respectively called the $\pi$\emph{-tensor product} and the $\epsilon$\emph{-tensor product} pseudo-norms.
 	\item[(5)] Let $i\geq 1$ be an integer. $(E,\|\cdot\|_v)$ be a pseudo-normed finite-dimensional $K$-vector space over $K$ in $v\in M_K$. We define the $i^{th}\pi$\emph{-exterior power pseudo-norm} $\|\cdot\|_{v,\Lambda^{i}_{\pi}}$ of $\|\cdot\|_v$ on $\Lambda^{i}E$ is defined as the quotient norm of the $\pi$-tensor product norm of $\|\cdot\|_v$ on $E^{\otimes i}$. Likewise, the $i^{th}\epsilon$\emph{-exterior power pseudo-norm} $\|\cdot\|_{v,\Lambda^{i}_{\epsilon}}$ of $\|\cdot\|_v$ on $\Lambda^{i}E$ is defined the quotient norm of the $\epsilon$-tensor product norm of $\|\cdot\|_v$ on $E^{\otimes i}$. In the $i=\dim_{k}(E)$ case, the $i^{th}\pi$-exterior power norm on $\Lambda^{i}E=\det(E)$ is called the \emph{determinant pseudo-norm} of $\|\cdot\|_v$ and we denote it by $\|\cdot\|_{v,\det}$.
 	%\item[(6)] Let $d\geq 2$ be an integer. Let $(E_1,\|\cdot\|_{1}),...,(E_{d},\|\cdot\|_{d})$ be pseudo-normed finite-dimensional $K$-vector space in $v\in M_K$. Let $\psi \in \Psi_{d}$ (cf. \S \ref{sub:reminder_absolute_norms}). Let $E:=E_{1}\oplus\cdots\oplus E_{d}$, $\cE:=\cE_{1}\oplus\cdots\oplus\cE_{d}$ and $\widehat{E}_v:= \widehat{E_{1}}_v\oplus\cdots\oplus\widehat{E_{d}}_v$. Denote by $(\widehat{E_v},\|\cdot\|_{\psi}^{\wedge})$ the $\psi$-direct sum of the residue normed vector spaces $\|\cdot\|_{1}^{\wedge},....,\|\cdot\|_{d}^{\wedge}$ (cf. Definition \ref{def:psi-direct_sum_example}). Then Proposition \ref{prop:equivalence_local_semi-norm} yields a pseudo-norm in $v$ on $E$ called the $\psi$\emph{-direct sum} of $\|\cdot\|_{1},...,\|\cdot\|_{d}$. %TODO voir somme directe
\end{itemize}
\end{definition}

\begin{proposition}
\label{prop:correspondence_dual_quotient}
Let $(E,(\|\cdot\|_v,\cE_v,N_v,\widehat{E}_v))$ be a pseudo-normed finite-dimensional $K$-vector space in $v=(\va_v,A_v,\m_v,\kappa_{v})\in M_K$. Let $G$ be a quotient of $E$ and denote by $\|\cdot\|_{G,v}$ the quotient pseudo-norm on $G$. Then the dual pseudo-norm $\|\cdot\|_{G,v,\ast}$ on $G^{\vee}$ identifies with the restriction of the pseudo-norm $\|\cdot\|_{v,\ast}$ on $E^{\vee}$ to $G^{\vee}$.
\end{proposition}

\begin{proof}
Let $\cG_v$ denote the finiteness module of $\|\cdot\|_G$, it is a quotient of $\cE_v$. Thus we have an inclusion of duals $\cG_v^{\vee}\hookrightarrow \cE_v^{\vee}$. The latter are respectively the finiteness modules of $\|\cdot\|_{G,v,\ast}$ and $\|\cdot\|_{v,\ast}$. To conclude the proof, it is enough to show 
\begin{align*}
\forall \varphi\in \cG_v^{\vee} \setminus \m_v\cG_v^{\vee}, \quad \|\varphi\|_{G,\ast} = \|\varphi\|_{v\ast}.
\end{align*}
This is obtained by lifting the corresponding assertion (\cite{ChenMori}, Proposition 1.1.20) for the residue norms.
\end{proof}

\begin{proposition}
\label{prop:comparaison_double_dual_semi-norm}
Let $(E,\|\cdot\|_v)$ be a pseudo-normed finite-dimensional $K$-vector space in $v\in M_K$ Then the inequality 
\begin{align*}
\|\cdot\|_{v,\ast\ast} \leq \|\cdot\|_v
\end{align*}
holds, where $\|\cdot\|_{v,\ast\ast}$ denotes the dual pseudo-norm of $\|\cdot\|_{v,\ast}$ on $E^{\vee\vee}\cong E$. Moreover, if either $v$ is Archimedean, or if $v$ is non-Archimedean and the pseudo-norm $\|\cdot\|_v$ is ultrametric, then we have 
\begin{align*}
\|\cdot\|_{v,\ast\ast} = \|\cdot\|_v
\end{align*}
\end{proposition}

\begin{proof}
Definition \ref{def:constructions_local_semi-norms} (3) ensures that, for any $x\in \cE_v\setminus \m_v\cE_v$, for any $\varphi\in\cE_v^{\vee}$, we have
\begin{align*}
|\varphi(x)|_v \leq \|\varphi\|_{v,\ast} \cdot \|x\|_v.
\end{align*}
Hence, for any $x\in \cE_v\setminus \m_v\cE_v$, we have
\begin{align*}
\|x\|_{v,\ast\ast} = \displaystyle\sup_{\varphi\in\cE_v^{\vee}\setminus \m_v\cE_v^{\vee}} \frac{|\varphi(x)|_v}{\|\varphi\|_{v,\ast}} \leq \|x\|_v.
\end{align*}
Note the, for any $x\in \m_v\cE_v$, for any $\varphi\in\cE_v^{\vee}$, we have $\varphi(x) \in \m_v$, i.e. $|\varphi(x)|_v=0$, and thus $\|x\|_{v,\ast\ast} = \|x\|_{\ast} = 0$. This concludes the proof of the first statement.

Now assume that the pseudo-absolute value $v$ is either Archimedean or $v$ is non-Archimedean and the pseudo-norm $\|\cdot\|_v$ is ultrametric. Then Proposition 1.1.18 and Corollary 1.2.12 of \cite{ChenMori} give 
\begin{align*}
\forall \overline{x}\in \widehat{E}_v,\quad \|\overline{x}\|_{v,\ast\ast}^{\wedge} = \|\overline{x}\|_{v}^{\wedge}.
\end{align*}
Thus we obtain
\begin{align*}
\forall=x\in \cE_v,\quad \|x\|_{v,\ast\ast} = \|x\|_{v}.
\end{align*}
To conclude the proof, it suffices to see that, for any $x\in E\setminus\cE_v$, we have $\|x\|_{v,\ast\ast}=+\infty$. Let $(e_1,...,e_r)$ a basis of $E$ which is adapted to $\|\cdot\|_{v}$. Then $(e_1^{\vee},...,e_{r}^{\vee})$ is a basis of $E^{\vee}$ which is adapted to $\|\cdot\|_{v,\ast}$. Let $x=x_1e_1+\cdots+x_re_r \in E\setminus\cE_v$. By symmetry, we may assume that $x_1\in K\setminus A_{v}$. By definition, 
\begin{align*}
\|x\|_{v,\ast\ast} = \displaystyle\sup_{\varphi\in \cE^{\vee}_{v}\setminus\m_{v}\cE_{v}^{\vee}} \frac{|\varphi(x)|_{v}}{\|\varphi\|_{v,\ast}}.
\end{align*}
Take $\varphi=e_1^{\vee} \in \cE^{\vee}_{v}\setminus\m_{v}\cE_{v}^{\vee}$. Then $|\varphi(x)|_{v}=|x_1|_{v} = +\infty$ and $\|\varphi\|_{v,\ast}\in\bR_{>0}$. Therefore $\|x\|_{v,\ast\ast}=+\infty$.
\end{proof}

We now generalise the Hadamard inequality to the pseudo-norm case. 

\begin{proposition}
\label{prop:Hadamrd_inequality_local_semi-norm}
Let $(E,\|\cdot\|_v)$ be a pseudo-normed finite-dimensional $K$-vector space in $v\in M_K$. Let $(e_1,..,e_r)$ be a basis of $E$ which is adapted to $\|\cdot\|_v$, namely, for all $i=1,..,r$, we have $e_i \in \cE_v\setminus N_v$, i.e. $\|e_i\|_v \in \bR_{>0}$. Then, for any $\eta\in \det(E)$, we have the equality
\begin{align*}
\|\eta\|_{v,\det} = \inf\left\{\|x_1\|_v\cdots\|x_r\|_v : x_1,...,x_r \in \cE_v \text{ and } \eta = x_1\wedge\cdots\wedge x_r\right\}.
\end{align*}
\end{proposition}

\begin{proof}
We may assume that $\eta\neq 0$. Let $\eta_0 = e_1\wedge \cdots\wedge e_r$. From the construction of tensor product and quotient pseudo-norms above, we obtain that $\|\eta_0\|_{v,\det}\in \bR_{>0}$. Since $\det(E)$ is of dimension 1, there exists $a\in K$ such that $\eta = a\eta_0$.  We first consider the $\|\eta\|_{v,\det}\in\bR_{>0}$ case. Then there exist $x_1,...,x_r\in \cE_v$ such that $\eta = x_1\wedge \cdots\wedge x_r$. From the definition of $\|\cdot\|_{v,\det}$, we directly obtain
\begin{align*}
\|\eta\|_{v,\det} \leq \|x_1\|_v\cdots\|x_r\|_v.
\end{align*}
We assume that $a\in A_v \setminus \m_v$, namely $\|\eta\|_{v,\det} \in \bR_{>0}$. Since $\|\eta\|_{v,\det}$ is computed from the residue norm on $\widehat{E}_v$, the classical Hadamard inequality implies that 
\begin{align*}
\|\eta\|_{v,\det} = \inf\left\{\|x_1\|_v\cdots\|x_r\|_v : x_1,...,x_r \in \cE_v \text{ and } \eta = x_1\wedge\cdots\wedge x_r\right\}.
\end{align*}
Now assume that $a\in \m_{v}$, then $\eta = (ae_1)\wedge\cdots\wedge e_r$ and 
\begin{align*}
\|\eta\|_{v,\det} = \|ae_1\|_{v}\cdots \|e_r\| = 0,
\end{align*}
hence we obtain the Hadamard equality. 

Finally, we treat the $a\in K \setminus A_v$ case. Now $\|\eta\|_{v,\det}=+\infty$. Moreover, there is no possible decomposition $\eta = x_1\wedge \cdots\wedge x_r$, where $x_1,...,x_r\in \cE_v$ (otherwise we would have $\|\eta\|_{v,\det}<+\infty$). Thus the RHS desired equality is $\inf(\emptyset)=+\infty$. 
\end{proof}

\section{Global space of pseudo-absolute values}
\label{sec:global_space_sav}
%TODO intro
We now define various notions of global spaces of pseudo-absolute values on a given field. Such spaces will contain all the relevant pseudo-absolute values in the context of topological adelic curves. Throughout this section, we fix a field $K$. 

\subsection{Definitions}
\label{sub:definitions_global_space_of_sav}

\begin{definition}
\label{def:set_of_sav}
Recall that we denote by $M_K$ the set of pseudo-absolute values on $K$. For all $f\in K$, let $U_f:= \left\{x\in V_K : |f|_x\neq +\infty\right\}$. We endow $M_K$ with the coarsest topology making the maps
\begin{align*}
\fonction{|f|_{\cdot}}{M_K}{[0,+\infty]}{x}{|f|_x}
\end{align*}
continuous, where $f$ runs over $K$. 

%Pour remédier à cela, on se donne un corps $K_0 \subset K$ pour lequel on peut construire l'ensemble $V_{K_0}$, On spécifie ensuite une partie fermée $V\subset K_0$ pour laquelle il existe une section continue $s : V \hookrightarrow V_K$ dont on note l'image $V$ par abus de notation. Un tel ensemble est appelé \emph{ensemble global de pseudo-absolute values} sur $K$, ou plus succinctement \emph{ensemble global de places} de $K$. Un tel corps $K_0$ sera appelé un \emph{corps de définition} pour $V$.
\end{definition}

\begin{theorem}
\label{prop:compactness_set_sav}
$M_K$ is a non-empty, compact Hausdorff topological space. 
\end{theorem}

\begin{proof}
$M_K$ contains the trivial absolute value on $K$ and is thus non-empty. We first show that $M_K$ is Hausdorff. Let $\va_1\neq\va_2$ be two pseudo-absolute values on $K$. By symmetry, we may assume that there exists $f\in K$ such that $0<|f|_1 < |f|_2<+\infty$. Let $t\in ]|f|_1,|f|_2[$. Then $U_1 := \{x\in M_K : |f|_x < t\}$ and $U_2:= \{x\in M_K : |f|_x > t\}$ are both open, disjoint and $\va_i \in U_i$ for $i\in\{1,2\}$. 

We now prove the compactness of $M_K$. Let
\begin{align*}
\fonction{\iota}{M_K}{\prod_{f\in K}[0,+\infty]}{x}{(|f|_x)_{f\in K}}
\end{align*}
which is injective and continuous. From Tychonoff theorem, $E:=\prod_{f\in K}[0,+\infty]$ is compact Hausdorff. Hence it is enough to show that $M_K$ is closed in $E$. Let $I$ be a directed set and let $(\va_i)_{i\in I}$ be a generalised sequence in $M_K$ which converges to $(|f|)_{f\in K} \in E$. We show that the map
\begin{align*}
\fonction{\va}{K}{[0,+\infty]}{f}{|f|}
\end{align*}
defines a pseudo-absolute value on $K$. For all $f\in K$, we have a generalised sequence $(|f|_i)_{i\in I}$ in $[0,+\infty]$ which converges to $|f|$. Hence $|1|=\lim_{i\in I} |1|_i = 1$ and $|0|=\lim_{i\in I} |0|_i = 0$. Let $a,b\in K$. For all $i\in I$, we have $|a+b|_i \leq |a|_i + |b|_i$, hence
\begin{align*}
|a+b| = \lim_{i\in I} |a+b|_i \leq \lim_{i\in I}|a|_i + \lim_{i\in I}|b|_i = |a|+|b|.
\end{align*}
Finally, if $a,b\in K$ are such that $\{|a|,|b|\}\neq \{0,+\infty\}$, then there exists $i_0\in I$ such that 
\begin{align*}
\forall i \in I, i\geq i_0 \Rightarrow  \{|a|_i,|b|_i\}\neq \{0,+\infty\} \Rightarrow |ab|_i = |a|_i|b|_i.
\end{align*}
Whence
\begin{align*}
|ab| = \lim_{i\in I, i\geq i_0}|ab|_i = \lim_{i\in I, i\geq i_0} |a|_i |b|_i = \lim_{i\in I} |a|_i \lim_{i\in I} |b|_i = |a||b|. 
\end{align*}
\end{proof}

\begin{definition}
\label{def:restriction_global_space_sav}
Let $L/K$ be a field extension. There is a restriction map $\pi_{L/K} : M_L \to M_K$. It is continuous (by definition of the topologies of $M_L$ and $M_K$), surjective, and proper (cf. Proposition \ref{prop:compactness_set_sav}).
\end{definition}

\begin{notation}
\label{notation:parts_global_space_of_sav}
Let $K$ be a field. We introduce the following notation.
\begin{itemize}
\item $M_{K,\ar}$ denotes the set of Archimedean pseudo-absolute values on $K$.
\item $M_{K,\um}$ denotes the set of ultrametric pseudo-absolute values on $K$.
\item $M_{K,\sn}$ denotes the set of pseudo-absolute on $K$ values whose kernel is non-zero.
\item $M_{K,\triv}$ denotes the set of pseudo-absolute values on $K$ whose residual absolute value is trivial.
\item $M_{K,\disc}$ denotes the set of discrete pseudo-absolute values, namely the set of pseudo-absolute values on $K$ that correspond either to a discrete non-Archimedean absolute value on $K$ or to a pseudo-absolute value whose finiteness ring is a discrete valuation ring that is not a field.
\end{itemize}
\end{notation}

\subsection{Examples}
\label{sub:examples_global_space_of_sav}

\subsubsection{Number fields}

Let $K$ be a number field and $\cO_K$ be its ring of integers endowed with the norm 
\begin{align*}
\|\cdot\|_{\infty} := \displaystyle\max_{\sigma \hookrightarrow \bC} \va_\sigma,
\end{align*}
where $\sigma$ runs over the set of embeddings of $K$ in $\bC$. 

\begin{proposition}
\label{prop:set_sav_number_field}
$M_K$ is homeomorphic to $\cM(\cO_K,\|\cdot\|_{\infty})$, the analytic spectrum in the sense of Berkovich of the Banach ring $(\cO_K,\|\cdot\|_{\infty})$. 
\end{proposition}

\begin{proof}
Let $\va : K \to [0,+\infty]$ be a pseudo-absolute value on $K$ which is not an absolute value. Denote by $A$ its finiteness ring. Then $A$ corresponds to a non-trivial valuation of $K$ and is thus of rank $1$. Therefore, by Ostrowski theorem, it corresponds to some prime ideal $\p \in \Spec(\cO_K)$. The residue field of $A$ being finite, the residue absolute value induced by $\va$ needs to be trivial. Therefore the restriction of $\va$ to $\cO_K$ corresponds to the extremal point of the branch associated with $\p\in \Spec(\cO_K)$ in the analytic spectrum $\cM(\cO_K,\|\cdot\|_{\infty})$. Conversely, any extremal point of $\cM(\cO_K,\|\cdot\|_{\infty})$ gives rise to a pseudo-absolute value on $K$ which is not an absolute value. These constructions are inverse to each other. The remaining points of $\cM(\cO_K,\|\cdot\|_{\infty})$ correspond directly to absolute values on $K$. Therefore we obtain a bijection $\varphi : M_K \to \cM(\cO_K,\|\cdot\|_{\infty})$ which is continuous by definition of the topologies of $M_K$ and $\cM(\cO_K,\|\cdot\|_{\infty})$. We can conclude the proof as $M_K,\cM(\cO_K,\|\cdot\|_{\infty})$ are both compact Hausdorff.
%modif Klaus
\end{proof}

%modif Jérôme
\subsubsection{Function fields over $\bC$}
\label{subsub:global_space_function_field_over_C} 

\begin{proposition}
\label{prop:archimedean_part_function_field_over_C}
Let $K=\bC(T)$. We have a homeomorphism
%\begin{align*}
%M_{K,\um} \cong \mathbb{P}^{1,\triv}_{\bC},
%\end{align*}
%where $\mathbb{P}^{1,\triv}_{\bC}$ denotes the analytification of $\bP^{1}_{\bC}$ where $\bC$ is trivially valued, and
\begin{align*}
M_{K,\ar} \cong \mathbb{P}^{1}(\bC) \times ]0,1].
\end{align*}

\end{proposition}

\begin{proof}
%Let us prove the first homeomorphism. Recall that we denote by $\va_{\infty}$ the usual absolute value on $\bC$. Let $v$ be a pseudo-absolute value on $K$. Then its residue field is $\bC$ and the residue absolute value is an element of the hybrid spectrum of $\bC$, namely an absolute value of the form $\va_{\infty}^\epsilon$, where $\epsilon\in [0,1]$ (recall that $\va_{\infty}^0 = \va_{\triv}$). The usual absolute values on $K$ are necessarily non-Archimedean and their restriction to $\bC$ is trivial. Hence they are of the form $|\cdot|_{z,c,\um} : f\in K \mapsto \exp(-c\ord(f,z))\in\bR_{>0}$, where $z\in \bC \cup \{\infty\}$ and $c\geq 0$ (using the conventions $\ord(\cdot,\infty) = -\deg(\cdot)$ and $|\cdot|_{z,0,\um} = \va_{\triv}$ for all $z\in \bC\cup\{\infty\}$). Now assume that $v \in M_{K,\um}$ has a non-zero kernel. Then there exists $z\in \bC\cup\{\infty\}$ such that $\va_v = \va_{z,\infty,\um}$, where
%\begin{align*}
%\fonction{|\cdot|_{z,\infty,\um}}{\bC[T]_{(T-z)}}{\bR_{\geq 0}}{f}{
%\left\{\begin{matrix}
%&0 \text{ if } f(z) = 0,\\ 
%&1 \text{ if } f(z) \neq 0. 
%\end{matrix}\right.}
%\end{align*}
%From the above description, we obtain a continuous bijection $\mathbb{P}^{1,\triv}_{\bC} \to M_{K,\um}$. As $\mathbb{P}^{1,\triv}_{\bC}%$  is compact Hausdorff and $M_{K,\um}$ is Hausdorff, it is a homeomorphism.

Let $v=(\va,A,\m) \in M_{K,\ar}$. Then $A$ is a non trivial valuation ring of $\bC(T)$ and thus is of the form
\begin{align*}
A := \{f\in K : \ord(f,z)\geq 0\}
\end{align*}
for some $z\in \bC\cup\{\infty\}$. Then the residue field $A/\m$ is $\bC$ endowed with an Archimedean absolute value, necessarily of the form $\va_{\infty}^{\epsilon}$ where $\epsilon \in ]0,1]$. Hence $v=v_{z,\epsilon,\infty}$, where $z\in \bP^1(\bC)$ (using the notation of Example \ref{ex:sav} (2)). Thus we have a bijection $\varphi : M_{k,\ar} \to \bP^{1}(\bC) \times ]0,1]$. To show that $\varphi$ is continuous, it suffices to prove that its composition with the maps $\pi_1 : \bP^{1}(\bC) \times ]0,1] \to \bP^{1}(\bC)$ and $\pi_2 : \bP^{1}(\bC) \times ]0,1] \to ]0,1]$ are continuous. We have 
\begin{align*}
\fonction{\pi_1 \circ \varphi}{M_{k,\ar}}{\bP^{1}(\bC)}{v_{z,\epsilon,\ar}}{z}.
\end{align*}
It is continuous by definition of the analytic topology on $\bP^{1}(\bC)$. $\pi_2 \circ \varphi : \va_x \in M_{K,\ar} \mapsto \log(|2(x)|)/\log(2)\in ]0,1]$ is also continuous. Furthermore, for all $f\in K$, the map
\begin{align*}
\fonction{\va_{\varphi^{-1}(\cdot)}}{\bP^{1}(\bC) \times ]0,1]}{[0,+\infty]}{(z,\epsilon)}{|f(z)|^{\epsilon}}
\end{align*}
is continuous. Hence $\varphi$ is a homeomorphism.
\end{proof}

\begin{comment} 
\begin{proposition}
\label{prop:set_sav_function_field_over_C}
There is a homeomorphism $M_K\cong \bP^{1,\hyb}_{\bC}$, where $\bP^{1,\hyb}_{\bC}$ denotes the Berkovich hybrid projective line over $\bC$ (cf. \cite{Lemanissier15}).
\end{proposition}

\begin{proof}
Proposition \ref{prop:archimedean_part_function_field_over_C} implies that the disjoint union of the homeomorphisms 
\begin{align*}
\varphi_{\um} : M_{K,\um} \cong \mathbb{P}^{1,\triv}_{\bC}, \quad \varphi_{\ar} : M_{K,\ar} \cong \mathbb{P}^{1}(\bC) \times ]0,1]
\end{align*}
yields a bijection $\varphi : M_K \to \mathbb{P}^{1,\hyb}_{\bC}$. As the densities of $M_{K,\ar}$ in $M_K$ and of $\mathbb{P}^{1}(\bC) \times ]0,1]$ in $\mathbb{P}^{1,\hyb}_{\bC}$ are preserved by $\varphi_{\ar}$, we see that $\varphi$ is continuous. We can then conclude by compactness of $M_{K}$.
\end{proof}
\end{comment}

\subsection{Connection to Zariski-Riemann spaces}

Let $K$ be a field with prime subring $k$. The construction of the Zariski-Riemann space $\ZR(K/k)$ was recalled in \S \ref{sub:Zariski-Riemann_space}. Then there is a map $j : M_K \to \ZR(K/k)$ mapping a pseudo-absolute value $v\in M_{K}$ to its finiteness ring $A_v \in \ZR(K/k)$. The map $j$ is called the \emph{specification map}.

\begin{proposition}
The specification map $j : M_K \to \ZR(K/k)$ is continuous, where $\ZR(K/k)$ is equipped with the Zariski topology.
\end{proposition} 

\begin{proof}
Let $a_1,...,a_r$ be elements of $K$. Let $\cU := \{V \in \ZR(K/k) : a_1,....,a_r \in V\}$ be a basic open subset of $\ZR(K/k)$. Then we have
\begin{align*}
j^{-1}(\cU) = \displaystyle\bigcap_{i=1}^{r}\{v\in M_K : |a_i|_v \in [0,+\infty[\},
\end{align*}
which is open. 
\end{proof}

\section{Local analytic spaces}
\label{sec:local_analytic_spaces}

We now give alternative approaches to (local) analytic spaces. Throughout this section, we fix a field $K$ and a pseudo-absolute $v=(\va,A,\m,\kappa)\in M_K$, namely $\va : K \to [0,+\infty]$ is a pseudo-absolute value with finiteness ring $A$, kernel $\m$ and residue field $\kappa$. 

\subsection{Local model analytic space}
\label{sub:local_model_analytic_space}

Let $X \to \Spec(K)$ be a $K$-scheme, which is assumed to be projective for simplicity. Assume that we have a projective model $\cX$ of $X$ over $A$ (cf. \S \ref{subsub:models_over_Prüfer_domain}). The special fibre $\cX_s := \cX \times_{\Spec(A)} \Spec(\kappa)$ is now a projective $\kappa$-scheme. Denote by $\widehat{\kappa}$ the completion of $\kappa$ w.r.t. the residue absolute value on $\kappa$ induced by the pseudo-absolute value $v$. We call $\widehat{\cX_s} := \cX_{s}\times_{\Spec(\kappa)} \Spec(\widehat{\kappa})$ the \emph{completed special fibre} of the model $\cX$. Now $\widehat{\cX}$ is a projective $\widehat{\kappa}$-scheme and $\widehat{\kappa}$ is a completely (real-)valued field. Henceforth, we can construct the Berkovich analytic space $\widehat{\cX_s}^{\an}$ attached to $\widehat{\cX_{s}}$.  

\begin{definition}
\label{def:model_local_analytic_space}
With the above notation, the space $\widehat{\cX_s}^{\an}$ is called the \emph{local model analytic space} attached to $X$ w.r.t. the projective model $\cX$ over $v$. Note that if $v$ is a usual absolute value on $K$, i.e. when $A=K=\kappa$, $\cX\cong X$ and the corresponding local analytic space is just $(X \otimes_{K} \widehat{K})^{\an}$.
\end{definition}

Using the classical theory of Berkovich analytic spaces, we can adapt the usual notion of metric on a line bundle.

\begin{definition}
\label{def:local_model_pseudo-metric}
Let $X$ be a projective scheme over $\Spec(K)$ and let $L$ be an invertible $\cO_{X}$-module. We call \emph{local model pseudo-metric} in $v$ on $L$ the data $((\cX,\cL),\varphi)$ where:
\begin{itemize}
	\item[(1)] $(\cX,\cL)$ is a flat projective model of $(X,L)$ over $A$ (cf. \S \ref{subsub:models_over_Prüfer_domain}), with special fibre $\cX_s$ and completed special fibre $\widehat{\cX_s}$;
	\item[(2)] $\varphi$ is a metric on $\widehat{\cL_s}$, where $\widehat{\cL_s}$ is the pullback of $\cL$ to completed special fibre $\widehat{\cX_s}$.
\end{itemize}
A local model pseudo-metric $((\cX,\cL),\varphi)$ on $L$ is called \emph{continuous} if the metric $\varphi$ is so. In the case where $\cX$ is fixed, by "let $(\cL,\varphi)$ be a local pseudo-metric on $L$", we mean that $((\cX,\cL),\varphi)$ is a local model pseudo-metric on $L$ which is denoted by $(\cL,\varphi)$.
\end{definition}

Let $X$ be a projective scheme over $\Spec(K)$ and let $L$ be an invertible $\cO_{X}$-module. Let $((\cX,\cL),\varphi)$ be a local model-pseudo metric in $v$ on $L$. The attached local model analytic space $(\cX \otimes_{A} \widehat{\kappa})^{\an}$ is a compact Hausdorff topological space and the metric $\varphi$ defines a supremum seminorm on the $\widehat{\kappa}$-vector space of global sections $H^{0}(\widehat{\cX_s},\widehat{\cL_s})$ defined by
\begin{align*}
\forall \overline{s}\in H^{0}(\widehat{\cX_s},\widehat{\cL_s}),\quad \|\overline{s}\|_{\varphi} := \displaystyle\sup_{x\in X_v^{\an}} |\overline{s}|_{\varphi}(x) \in \bR_{\geq 0}.
\end{align*}

Assume that the model $(\cX,\cL)$ is flat, coherent and that the special fibre $\cX_s$ is geometrically reduced. Then Proposition \ref{prop:models_Prüfer_domains} implies that $H^{0}(\cX,\cL)$ is a free $A$-module of finite rank. Since $\cX_s$ is reduced, (\cite{ChenMori}, Proposition 2.1.16 (1)) implies that the seminorm $\|\cdot\|_{\varphi}$ is a norm on $H^{0}(\widehat{\cX_s},\widehat{\cL_s})$. Proposition \ref{prop:models_Prüfer_domains} (2.ii) implies that $H^{0}(\cX,\cL) \otimes_{A} \kappa$ is a vector subspace of $H^{0}(\cX_s,\cL_s)$. Thus $\|\cdot\|_{\varphi}$ induces a norm on $H^{0}(\cX,\cL) \otimes_{A} \widehat{\kappa}$ and Proposition \ref{prop:equivalence_local_semi-norm} implies that we can lift $\|\cdot\|_{\varphi}$ to a pseudo-norm in $v$ on $H^{0}(X,L)$, which is again denoted by $\|\cdot\|_{\varphi}$ by abuse of notation.

Now we assume that the model $(\cX,\cL)$ is flat and coherent. We denote by $(\widehat{\cX_{s}})_{\red}$ the reduced scheme structure on $\widehat{\cX_{s}}$ and by $(\widehat{\cL_{s}})_{\red}$ the restriction of $\widehat{\cL_{s}}$ to $(\widehat{\cX_{s}})_{\red}$. Then (\cite{ChenMori}, Proposition 2.1.16) implies that the kernel of the supremum semi-norm on $H^{0}(\widehat{\cX_{s}},\widehat{\cL_{s}})$ coincides with the kernel of the natural map $H^{0}(\widehat{\cX_{s}},\widehat{\cL_{s}}) \to H^{0}((\widehat{\cX_{s}})_{\red},(\widehat{\cL_{s}})_{\red})$. Moreover, if $\overline{s} \in H^{0}((\widehat{\cX_{s}})_{\red},(\widehat{\cL_{s}})_{\red})$ lies in the kernel of the supremum seminorm $\|\cdot\|_{\varphi}$, there exists some integer $n\geq 1$ such that $\overline{s}^{\otimes n} = 0$. 

\begin{lemma}
\label{lemma:sup_semi-norm_special_fibre}
Assume that the valuation ring $A$ is non trivial. Let $s\in H^{0}(\cX,\cL)$ be a global section such that $\|s_{|\cX_s}\|_{\varphi}=0$. Then $s\in \m H^{0}(\cX,\cL)$.
\end{lemma}

\begin{proof}
The above paragraph implies that if $\|s_{|\cX_s}\|_{|\varphi}=0$, then there exists some integer $n\geq 1$ such that $s_{|\cX_s}^{\otimes n}=0$. Let $\cX= \cup_{i=1}^{k}U_i$ be an open covering, where $U_i = \Spec(A_i)$ for $i=1,...,k$, and such that $\cL_{|U_i} \cong \cO_{U_i}$. Then for any $i=1,...,k$, there exists $a_i\in A_i$ such that $s_{|U_i} = a_i e_i$, where $e_i$ denotes a generator of $\cL_{|U_i}$. Therefore, for any $i=1,...,k$, $a_i^{n}\in \m A_i$, hence $a_i\in \m A_i$. By glueing, we obtain $s\in H^{0}(\cX,\m\cL)\cong \m H^{0}(\cX,\cL)$.
\end{proof}

Therefore, we can lift the supremum seminorm $\|\cdot\|_{\varphi}$ on $H^{0}(\widehat{\cX_{s}},\widehat{\cL_{s}})$ to a pseudo-norm on $H^{0}(X,L)$. This construction is summarised in the following definition.

\begin{definition}
\label{def:sup-pseudo-norm_model_pseudo-metric}
Assume that $X$ is geometrically reduced if $A=K$. Let $L$ be a line bundle on $X$. Let $(\cX,\cL)$ be a flat and coherent projective model of $(X,L)$ over $A$. Let $(\cL,\varphi)$ be a local pseudo-metric on $L$. The pseudo-norm $\|\cdot\|_{\varphi}$ in $v$ on $H^{0}(X,L)$ is called the \emph{supremum pseudo-norm} attached to $((\cX,\cL),\varphi)$. Its finiteness module is $H^{0}(\cX,\cL)$ and its kernel is $\m H^{0}(\cX,\cL)$. In particular, if $A=K$ (i.e. $v$ is a usual absolute value on $K$), then $\|\cdot\|_{\varphi}$ corresponds to a norm in the classical sense.
\end{definition}

\subsection{General local analytic space}
\label{sub:local_analyticspace}

\subsubsection{Naive attempt}

Let $X\to\Spec(K)$ be a $K$-scheme, one could try mimic directly constructions of Berkovich analytic spaces by replacing absolute values and norms by pseudo-absolute values and pseudo-norms. This naive construction is the following.

Let $f:X \to \Spec(K)$ be a $K$-scheme. The \emph{local analytic space} $X_v^{\an}$ attached to $X$ over $v$ is defined as the set of pairs $x=(p,\va_x)$, where $p\in X$ is a scheme point and $\va_x \in M_{\kappa(p)}$ is a pseudo-absolute value on the residue field $\kappa(p)$ of $p$ such that the restriction of $\va_x$ to $K$ is $\va_v$. There is a \emph{specification morphism} $j_v : X_v^{\an} \to X$ sending any $x=(p,\va_x)\in X_v^{\an}$ to $p$. For any Zariski open subset $U \subset X$, let $U_v^{\an}:= j_v^{-1}(U)$. For such $U$, any regular function $f\in \cO_X(U)$ defines a map $|f|_{\cdot} : U_v^{\an} \to [0,+\infty]$ by sending any point $x\in X^{\an}$ to $|f(j(x))|_{x}$.

We now equip $X_{v}^{\an}$ with the coarsest topology making $j$ together with the maps $|f|_{\cdot}$, for any Zariski open subset $U\subset X$ and any $f\in \cO_X(U)$.

\begin{proposition}
\label{prop:sav_extension_by_generalisation_schemes}
Assume that $X$ is a $K$-variety. Then the space $X_{v}^{\an}$ is not Hausdorff.
\end{proposition}

\begin{proof}
Let $p\in X$ be a regular closed point. Let $v_x=(\va_{x},A_x,\m_x,\kappa_x)\in M_{\kappa(p)}$ be a pseudo-absolute value. Proposition \ref{prop:valuation_rings_specialisation_smooth} yields the existence of a valuation ring $V$ of $K(X)$ dominating $\cO_{X,p}$ with residue field $\kappa(p)$. Then consider the extension by generalisation of $v_x$ to $M_{K}$ induced by $V$ (cf. Definition \ref{def:extension_by_generalisation}), we denote it by $v'_{x}=(\va'_{x},A'_x,\m'_x,\kappa_x)$. Let $U\subset X$ be an open Zariski neighbourhood of $p$, let $f\in \cO_{X}(U)$ be a regular function. By definition of $v'_{x}$ and since $f\in \cO_{X,p} \subset V$, we have
\begin{align*}
|f|'_{x} =  |f(p)|_{x}|.
\end{align*} 
Therefore, any open neighbourhood $V\subset X_{v}^{\an}$ of $x=(p,v_x)$ contains $(\eta,v'_{x})$, where $\eta$ denotes the generic point of $X$. This implies that $X_{v}^{\an}$ is not Hausdorff.
\end{proof}

\subsubsection{Birational approach}
\label{subsub:local_analytic_space_general}

Roughly speaking, in the last paragraph, we saw that, on a variety, most of all pseudo-absolute values on residue fields are specialisations of pseudo-absolute values on the function field. Therefore, one could try to work only with these pseudo-absolute values. 

\begin{definition}
\label{def:local_analytic_space_general}
Let $K'/K$ be an extension of algebraic function fields. Let $v\in M_{K}$ be a pseudo absolute-value. The \emph{local analytic space} attached to $K'$ above $v$ is the set $M_{K',v}$ of pseudo-absolute values of $K'$ extending $v$ on $K$. In other terms, $M_{K',v}$ is the fibre of $v$ via the projection $\pi_{K'/K} : M_{K'} \to M_{K}$. $M_{K',v}$ is equipped with the subspace topology induced by the topology on $M_{K'}$ defined in \S \ref{sec:global_space_sav}. Proposition \ref{prop:compactness_set_sav} implies that $M_{K',v}$ is compact Hausdorff. 
\end{definition}

From now on we fix an extension of algebraic function fields $K'/K$. Then there is the Zariski-Riemann space $\ZR(K'/A)$ (cf. \S \ref{sub:Zariski-Riemann_space}). We have a \emph{specification morphism} $j : M_{K',v} \to \ZR(K'/A)$ mapping any $v'\in M_{K',v}$ to its finiteness ring. 

\begin{lemma}
\label{lemma:local_specification_morphism_continuous}
The specification morphism $j:M_{K',v} \to \ZR(K'/A)$ is continuous, where $\ZR(K'/A)$ is endowed with the Zariski topology.
\end{lemma}

\begin{proof}
Let $a\in K'$ and $U := \{A'\in \ZR(K'/A) : a\in A'\}$ be a basic open set of $\ZR(K'/A)$. By definition, $j{-1}(U)$ is the set of pseudo absolute values $v'$ on $K'$ extending $v$ and with finiteness ring containing $a$. Thus we see that $j^{-1}(U)=|a|_{\cdot}^{-1}([0,+\infty[)$. Since $[0,+\infty[$ is an open subset of $[0,+\infty]$, $j^{-1}(U)$ is open by definition of the topology on $M_{K',v}$.
\end{proof}

\subsubsection{Link with model local analytic spaces}
\label{subsub:link_local_analytic_spaces}

We now link the general local analytic space to the local model analytic spaces introduced in \S \ref{sub:local_model_analytic_space}. Let $\cX$ be an arbitrary projective model of $K'$ over the valuation ring $A$. Denote by $\widehat{\cX_s}$ the local model analytic space attached to $\cX$. Note that in the case where $v$ is a usual absolute value, $\cX$ is a birational model of $K'/K$ and $\widehat{\cX_s}=(\cX \otimes_{K} \widehat{K})^{\an}$.  

Let $v'=(\va',A',\m',\kappa')\in M_{K',v}$. We have a commutative diagram
\begin{center}
% https://tikzcd.yichuanshen.de/#N4Igdg9gJgpgziAXAbVABwnAlgFyxMJZABgBpiBdUkANwEMAbAVxiRAB12BlNGAYwAUAaQDkAShABfUuky58hFGQCMVWoxZtOPfgICC4qTJAZseAkWXk19Zq0Qd2fABpHZZhZdKrqtzQ+1eQT0JaXd5CxQAJmtfDXsQZwB9GjcTOXNFZBifdTstbiCBTgBrOjQ0OlDjUwisgGZYvP9EpOAaSQA9YE46MEkpNRgoAHN4IlAAMwAnCABbJDIQHAgkKLCQGfmkK2XVxHqNrYXEJZWdo9mTmL2kQ+NjpAAWanPEdYerpABWV-37qZfRAvW6Ib6XbaIABsf2ekgokiAA
\begin{tikzcd}
\Spec(K') \arrow[r] \arrow[d] & \cX \arrow[d]  \\
\Spec(A') \arrow[r]           & \Spec(A)                          
\end{tikzcd}.
\end{center}
By valuative criterion of properness, there exists a unique morphism of schemes $\Spec(A') \to \cX$ factorising $\Spec(K') \to \cX$. Denote by $p$ the image of the closed point of $\Spec(A')$ in $\cX_s$. By considering the restriction of the residue absolute value of $v'$ to $\kappa(p)$, we obtain an absolute value on $\kappa(p)$ extending the residue absolute value of $v$. Therefore we have successive valued field extensions $\kappa'/\kappa(p)/\kappa$. Denote by $\widehat{p}$ the $\widehat{\kappa(p)}$-point of $\widehat{\cX_s}$ induced by $\Spec(\widehat{\kappa(p)}) \to \Spec(\kappa(p))$. Then we have extensions $\widehat{\kappa'}/\widehat{\kappa(p)}/\widehat{\kappa}$ of completely valued fields. Thus we obtain a point $x\in \widehat{\cX_s}^{\an}$. Finally, we have obtained a map $\red_{\cX} : M_{K',v}\to \widehat{\cX_s}^{\an}$ such that the diagram
\begin{center}
% https://tikzcd.yichuanshen.de/#N4Igdg9gJgpgziAXAbVABwnAlgFyxMJZABgBpiBdUkANwEMAbAVxiRAFkB9YAaQHJSNAL4ghpdJlz5CKMgEYqtRizYAdVQGMAGgD1g6umBFiJ2PASJzSC6vWatEIdds4ITIDGemXyiuysd1AC0AJQAKfgB6AEEASlFFGCgAc3giUAAzACcIAFskMhAcCCQrEAY6ACMYBgAFSXMZECysZIALHBBbZQcnVSyk0XEQbLykAGZqYqQAJndR-MQy6cQ54YWCqZLEcaEKISA
\begin{tikzcd}
{M_{K',v}} \arrow[d, "\red_{\cX}"'] \arrow[r] & \ZR(K'/A) \arrow[d] \\
\widehat{\cX_s}^{\an} \arrow[r]                     & \cX_s              
\end{tikzcd}
\end{center} 
commutes, where the horizontal maps are the specification morphisms and the right vertical map is the center map defined by the valuative criterion of properness. 

\begin{proposition}
\label{prop:continuity_reduction_local}
Let $\cX$ be a projective model of $K'/A$. Then the map $\red_{\cX} : M_{K',v}\to \widehat{\cX_{s}}^{\an}$ is continuous.
\end{proposition}

\begin{proof}
Denote by $k:\widehat{\cX_s}^{\an} \to \widehat{\cX_s}$ the specification morphism, by $j: \widehat{\cX_s} \to \cX_s$ the extension of scalars and by $i: \cX_s \to \cX$ the closed immersion. Let us first prove that the composition $f := i\circ \circ j \circ k \circ \red_{\cX} : M_{K',v} \to \cX$ is continuous. Let $U=\Spec(B)$ be an open affine subset of $\cX$. By definition, we have 
\begin{align*}
f^{-1}(U) = \{|\cdot|'\in M_{K',v} : B \subset A_{|\cdot|'}\}.
\end{align*}
Let $x_1,...,x_r$ denote generators of $B$ as an $A$-algebra. Then we have
\begin{align*}
f^{-1}(U) := \{|\cdot|'\in M_{K',v} : \displaystyle\max_{i=1,...,r} |x_i|' < + \infty\},
\end{align*} 
which is an open subset of $M_{K'_v}$. Thus $f$ is continuous, hence so is $j \circ k \circ \red_{\cX}$. Since $j$ is open, $k \circ \red_{\cX}$ is continuous. Let $\Spec(\widehat{B})$ be an affine open subset of $\widehat{\cX_s}$. Let $\tilde{f}\in \widehat{B}$. By construction of $\red_{\cX}$ and by definition of the topology on $M_{K',s}$, the map $|f|_{\cdot} : \red_{\cX}^{-1}(\Spec(\widehat{B})^{\an})\to \bR_{\geq 0}$ is continuous. By definition of the topology of $\widehat{\cX_s}^{\an}$, the map $\red_{\cX}$ is continuous.
\end{proof}

In view of results of \S \ref{sub:Zariski-Riemann_space}, we now study the compatibility of the above reduction maps w.r.t. the domination relation between models.

Let $\cX,\cY$ be two projective models of $K'/A$. Assume that we have a morphism of schemes $f:\cY\to\cX$. $f$ induces a morphism $f_s : \widehat{\cY_s}\to\widehat{\cX_s}$ between the special fibres and the corresponding analytification $f_{s}^{\an} : \widehat{\cY_{s}}^{\an}\to\widehat{\cX_{s}}^{\an}$. By uniqueness of the maps $\Spec(A') \to \cX$, $\Spec(A')\to\cY$ from the valuative criterion of properness, we have a factorisation
\begin{center}
% https://tikzcd.yichuanshen.de/#N4Igdg9gJgpgziAXAbVABwnAlgFyxMJZABgBpiBdUkANwEMAbAVxiRAFkB9YAaQHJSNAL4ghpdJlz5CKAIylZVWoxZsAOmoDGADVHiQGbHgJEyi6vWatEIDZoCaopTCgBzeEVAAzAE4QAtkhkIDgQSABM1Ax0AEYwDAAKksYyID5YrgAWOCAWKta2aj4u3Hb2ImLefoGIkSFhiPIg0XGJydJs6Vk5lSC+AUHUoUhNlqo2GsVQpVraFRRCQA
\begin{tikzcd}
{M_{K',v}} \arrow[d, "\red_{\cY}"'] \arrow[rd, "\red_{\cX}"] &     \\
\widehat{\cY_{s}}^{\an} \arrow[r, "f_{s}^{\an}"']                                                & \widehat{\cX_{s}}^{\an}
\end{tikzcd}.
\end{center}
Denote by $\cM$ the collection of all projective models of $K'/A$. Then $\cM$ is an inverse system of locally ringed spaces which induces an inverse system of locally ringed spaces $(\widehat{\cX_{s}}^{\an})_{\cX \in \cM}$. The above construction shows that the reduction maps are compatible with this projective system. Therefore we obtain a commutative diagram
\begin{center}
% https://tikzcd.yichuanshen.de/#N4Igdg9gJgpgziAXAbVABwnAlgFyxMJZABgBpiBdUkANwEMAbAVxiRAFkB9YAaQHJSNAL4ghpdJlz5CKAIzkqtRizYAdVQC0ASgAp+AegCCASlHiQGbHgJEysxfWatEIdfQBOadxABWDLAC23OoAxgAaQqFh3HBCAHrA6nRgImISVtJE8vbUjioubnSe3n6BwarhkRVhooowUADm8ESgAGbeAUhkIDgQSABM1Ax0AEYwDAAKktYyIO5YDQAWOCC5ys6uqu71Zm0dSPI9fYgAzGtOahUEDbsg7RCdiINHSCdpd-uI3b0H5-kgPlqQiAA
\begin{equation}
\label{eq:cd_local_analytic_spaces_Zariski-Riemann}
\begin{tikzcd}
{M_{K',v}} \arrow[d, "\red"'] \arrow[r, "j"] & \ZR(K'/A) \arrow[d, "\cong"] \\
\varprojlim_{\cX}\widehat{\cX_{s}}^{\an} \arrow[r]     & \varprojlim_{\cX}\cX        
\end{tikzcd},
\end{equation}
\end{center}
where the right hand side arrow is a homeomorphism by Theorem \ref{th:Zariski-Riemann_projective_limite}.

\begin{theorem}
\label{prop:local_analytic_space_projective_limit}
We use the above notation. The map $\red : M_{K',v}\to \varprojlim_{\cX\in\cM}\widehat{\cX_{s}}^{\an}$ is a homeomorphism. 
\end{theorem}

\begin{proof}
Note that from Proposition \ref{prop:continuity_reduction_local}, the map $\red$ is continuous. Since $M_{K',v}$ and $\varprojlim_{\cX\in \cM}\widehat{\cX_{s}}^{\an}$ are both compact Hausdorff, it suffices to prove that $\red$ is bijective. 

\begin{comment}
We start by the following lemma.

\begin{lemma}[\cite{BouAC}, Chap. IX, Appendice 1, Proposition 1]
\label{lemma:direct_limit_local_rings}
Let $(A_i)_{i\in I}$ be a direct system of local rings. For any $i\in I$, let $\m_i$ denote the maximal ideal of $A_i$. Then $A := \varinjlim_{i\in I} A_i$ is a local ring with maximal ideal $\m := \varinjlim_{i\in I}\m_i$. Moreover, for any $i \in I$, the canonical morphism $\iota_i : A_i \to A$ is local, and $\varphi_i(\m_i)A = \m$. Finally, we have a field isomorphism
\begin{align*}
A/\m \cong \displaystyle\varinjlim_{i\in I} A_i/\m_i = \bigcup_{i\in I} A_i/\m_i.
\end{align*}
Moreover, if the arrows of the directed system $(A_i)_{i\in I}$ are assumed to be flat, then, for any $i\in I$, $\iota_i : A_i \to A$ is a flat morphism of local rings.
\end{lemma}
\end{comment}

We first prove that $\red:M_{K',v}\to\varprojlim_{\cX\in \cM} \widehat{\cX_{s}}^{\an}$ is injective. Let $v'_1,v'_2 \in M_{K',v}$ be such that $\red(v'_1)=\red(v'_2)$. (\ref{eq:cd_local_analytic_spaces_Zariski-Riemann}) implies that the finiteness rings of $v'_1$ and $v'_2$ are equal. We denote this valuation ring by $A'$ and by $\kappa'$ its residue field. Denote respectively by $\va'_1,\va'_2$ the residue absolute values on $\kappa'$ induced by $v'_1,v'_2$. Now Theorem \ref{th:Zariski-Riemann_projective_limite} (2) implies that 
\begin{align*}
A'=\displaystyle\bigcup_{\cX\in\cM} \cO_{\cX,x_{A',\cX}},
\end{align*} 
where, for any $\cX\in\cM$, $x_{A',\cX}$ denotes the center of $A'$ on $\cX$. Since morphisms between models are dominant morphisms of integral schemes, any morphism of models $\cX' \to \cX$ in $\cM$ induces an injective morphism of local rings $\cO_{\cX,x_{A',\cX}} \to \cO_{\cX',x_{A',\cX'}}$. Therefore (\cite{BouAC}, Chap. IX, Appendice 1, Proposition 1) implies that $A'=\varinjlim_{\cX\in\cM}\cO_{\cX,x_{A',\cX}}$ and that
\begin{align}
\label{eq:local_analytic_space_residue_fields}
\kappa' = \bigcup_{\cX\in\cM} \kappa(x_{V',\cX}). 
\end{align}
Now $\va'_1$ and $\va'_{2}$ are absolute values on $\kappa'$ such that, for any $\cX\in\cM$, their restrictions to $\kappa(x_{V',\cX})$ are equal. Thus (\ref{eq:local_analytic_space_residue_fields}) implies that $\va'_1$ is equal to $\va'_2$, which shows the injectivity of $\red$. 

Let us now prove the surjectivity of $\red$. We will use the following lemma.

\begin{lemma}
\label{lemma:projective_limit_CHaus_spaces}
Let $(X_i)_{i\in I}$ be an inverse system of compact Hausdorff topological spaces.Denote by $X := \varprojlim_{i\in I} X_i$ the inverse limit of the inverse system. Let $Y$ be a compact Hausdorff topological space equipped with a continuous map $g : Y \to X$. For any $i\in I$, let $g_i:Y \to X_i$ denote the composition of $g$ with the projection $f_i : X \to X_i$.
\begin{itemize}
	\item[(i)] Assume that, for any $i\leq j$ in $I$, the arrow $f_{i,j} : X_j \to X_i$ is surjective. Then, for any $i\in I$, $f_i$ is surjective.
	\item[(ii)] If $g_i : Y \to X_i$ is surjective for any $i\in I$, then $g : Y \to X$ is surjective. 
\end{itemize}
\end{lemma}

\begin{proof}
(i) follows from (\cite{Capel54}, 2.6). 

Let us prove (ii). Let $x\in X$ and let, for any $i\in I$, $x_i := f_i(x) \in X_i$. For any $i\in I$, set
\begin{align*}
C_i := \{y\in Y : g_i(y)=x_i\},
\end{align*}
it is a non-empty (since $g_i$ is surjective) closed subset of $Y$, hence compact Hausdorff. It suffices to prove that $\cap_{i\in I}C_i$ is non-empty. By compactness of $Y$, it suffices to prove that any finite intersection of $C_i$ is non-empty. Let $I_0\subset I$ be a finite subset. Let $i_0\in I$ be an upper bound of $I_0$. Let $y\in C_{i_0}$, namely $g_{i_0}(y) =x_{i_0}$. For any $i\in I_0$, since $i\leq i_0$, we have 
\begin{align*}
g_i(y)=f_{i,i_0}(g_{i_0}(y)) = x_i.
\end{align*}
Therefore $y\in \cap_{i\in I_{0}}C_i$, which concludes the proof of (ii).
\end{proof}

Note that that, any morphism of projective models $f:\cY\to\cX$ of $K'/A$ is dominant and closed, hence surjective (\cite{ZariskiSamuelII}, Chap. VI, \S 17, Lemma 5). Therefore, for any such morphism, the induced analytic map $f_{s}^{\an}$ is surjective. Since, for any $\cX\in\cM$, $\widehat{\cX_s^{\an}}$ is a compact Hausdorff space, Lemma \ref{lemma:projective_limit_CHaus_spaces} (i) implies that the projection $p_{\cX}:\varprojlim_{\cX'\in \cM}\widehat{\cX'_{s}}^{\an} \to \widehat{\cX_{s}}^{\an}$ is surjective for all $\cX\in\cM$. Moreover, Lemma \ref{lemma:projective_limit_CHaus_spaces} (ii) shows that it suffices to prove that, for any $\cX\in\cM$, the reduction map $\red_{\cX}:M_{K',v}\to \widehat{\cX_{s}}^{\an}$ is surjective.

Let $\cX\in\cM$ and let $x=(p,\va_x)\in \widehat{\cX_s}^{\an}$, namely $p\in \widehat{\cX_s}$ is a scheme point and $\va_x$ is an absolute value on $\kappa(p)$ extending the absolute value on $\widehat{\kappa}$. Denote by $q$ the image of $p$ via $\widehat{\cX_s} \to \cX$. By restriction of $|\cdot|_x$, we obtain an absolute value on $\kappa(q)$ extending the residue absolute value of $v$ on $\kappa$.

\begin{lemma}
\label{lemma:domintation_valuation_ring_residue_algebraic}
There exists a valuation ring $A'$ of $K'$ such that $A'$ dominates $\cO_{\cX,q}$ and the residue field extension $(A'/\m_{A'})/\kappa(q)$ is algebraic.
\end{lemma}

\begin{proof}
Let 
$$E:=\{\cO_{\cX,q}\subset R \subsetneq K' : R \text{ local and dominates }\cO_{\cX,q} \text{ and } (R/\m_{R})/\kappa(q) \text{ is algebraic}\}.$$ 
Then $E$ is an inductive set and by Zorn's lemma, it has a maximal element $(A',\m',\kappa')$. By maximality, $A'$ is a valuation ring of $K'$. 
\end{proof}

Let $(A',\m',\kappa')$ be a valuation ring satisfying the conditions of Lemma \ref{lemma:domintation_valuation_ring_residue_algebraic}. Then we have a morphism $\Spec(A')\to \cX$ mapping the closed point of $\Spec(A')$ to $p$. Since the extension $\kappa'/\kappa(p)$ is algebraic, there exists an extension $\va'$ of the residue absolute value $\va_x$ on $\kappa(p)$. The data of $A'$ and $\va'$ determines a pseudo-absolute value $v'$ on $K'$ above $v$ such that $\red_{\cX}(v')=(p,\va_x)$. Hence the surjectivity of $\red_{\cX}$. This concludes the proof of the theorem
\end{proof}

\begin{remark}
\label{rem:local_analytic_space_general}
In view of Definition \ref{def:local_analytic_space_general} and Lemma \ref{lemma:local_specification_morphism_continuous}, we see that the definition of local analytic space is birational. Roughly speaking, the "algebraic part" is a certain Zariski-Riemann space instead of a scheme in the classical case.
\end{remark}

\section{Integral structures}
\label{sec:integral_structure}

%D'un point de vue philisophique, on voudrait définir un faisceau des fonctions analytiques sur $M_K$ ainsi que sur toute structure adélique qui donnerait une interprétation globale de la complétion d'un corps muni d'une pseudo-absolute value. Il semble que l'on peut raisonner comme Poineau et montrer que dans les exemples, on retrouve l'exemple intuitif de complétion. 

Throughout this section, we fix a field $K$. In practice, the set $M_K$ might be "too big". In order to mimic constructions arising in global analytic geometry, we introduce boundedness restrictions on the pseudo-absolute values.

\subsection{Definition of integral structures}

\begin{definition}
\label{def:integral strucure} 
An \emph{integral structure} for $K$ is the data $(A,\|\cdot\|_A)$ where:
\begin{itemize}
	\item[(i)] $A$ is an integral subdomain of $K$ with fraction field $K$;
	\item[(ii)] $A$ is Prüfer;
	\item[(iii)] $\|\cdot\|_A$ is a Banach norm on $A$.
\end{itemize}
When no confusion may arise, we use the notation $A$ for an integral structure $(A,\|\cdot\|_A)$.
\end{definition}

\begin{remark}
Denote by $k$ the prime subring of $K$. Let $(A,\|\cdot\|_A)$ be an integral structure for $K$. The latter can be interpreted through the \emph{Zariski-Riemann space} associated to $K$, namely the set $\ZR(K/k)$ of valuation rings of $K$. As $A$ is integrally closed (\cite{Gilmer72}, Chapter IV), it is the intersection of all its valuation overrings (\cite{BouAC}, Chapter VI, \S 1.3 Corollaire 2). Note that any valuation overring of $A$ is of the form  $A_{\p}$ for some $\p\in\Spec(A)$ (cf. Proposition \ref{prop:prop_Prufer_domains} (1)). Hence we have
\begin{align*}
A = \bigcap_{\p\in\Spec(A)} A_{\p}.
\end{align*}
Let $X\subset \ZR(K/k)$ be the subset such that $A = \bigcap_{V\in X} V$. We have a characterisation of such subsets $X$ in terms of affine subsets of $\ZR(K/k)$ (cf. \cite{Olberding21}, Theorem 6.2). Let us also mention that Olberding has done an extensive work in these directions \cite{Olberding07,Olberding10,Olberding11}.
\end{remark} 

The notion of integral structure allows to define refined versions of $M_K$.
 
\begin{definition}
\label{def:global_space_sav}
Let $K$ be a field equipped with an integral structure $(A,\|\cdot\|_A)$. The \emph{global space of pseudo-absolute values} relative to the integral structure $A$ is defined as
\begin{align*}
V_{K,A} := \{\va \in M_K : \va_{|A} \leq \|\cdot\|_{A}\},
\end{align*}
and is equipped with the topology induced by that of $M_K$.
\end{definition}

\begin{notation}
\label{notationpartiesespaceglobaldesemiva}
Let $K$ be a field. Let $(A,\|\cdot\|_A)$ be an integral structure for $K$. Denote by $V$ the associated global space of pseudo-absolute values. We introduce the following notation.
\begin{itemize}
\item $V_{\ar}:=M_{K,\ar}\cap V$ denotes the set of Archimedean pseudo-absolute values of $V$.
\item $V_{\um}:= M_{K,\um} \cap V$ denotes the set of ultrametric pseudo-absolute values of $V$.
\item $V_{\sn}:= M_{K,\sn} \cap V$ denotes the set of pseudo-absolute of $V$ values whose kernel is non-zero.
\item $V_{\triv} := M_{K,\triv} \cap V$ denotes the set of pseudo-absolute values of $V$ whose residual absolute value is trivial.
\item $V_{\disc} := M_{K,\disc} \cap V$ denotes the set of discrete pseudo-absolute values, namely the set of pseudo-absolute values of $V$ that correspond either to a discrete non-Archimedean absolute value on $K$ or to a pseudo-absolute value whose finiteness ring is a discrete valuation ring that is not a field.
\end{itemize}
\end{notation}

\begin{proposition}
\label{prop:compactness_global_space_sav}
With the notation of Definition \ref{def:global_space_sav}, $V$ is a non-empty compact Hausdorff topological space.
\end{proposition}

\begin{proof}
There is a natural map $f:V_{K,A}\to\cM(A)$, where $\cM(A)$ denotes the Berkovich analytic spectrum of $(A,\|\cdot\|_A)$, which is given by restricting elements of $V_{K,A}$ to $A$. By definition of the topologies, $f$ is continuous. Let $x\in\cM(A)$ and denote by $\p_x\in\Spec(A)$ its kernel. Since $A$ is Prüfer, the localisation $A_{\p_x}$ is a valuation ring of $K$ on which $x$ defines a multiplicative semi-norm. Therefore, every element of $\cM(A)$ induces an element of $V_{K,A}$. This construction provides an inverse of $f$ that is a continuous function. The conclusion follows from (\cite{Berko90}, Theorem 1.2.1).
\end{proof}

\begin{remark}
The above proof actually shows that $V$ is homeomorphic to $\cM(A)$.
\end{remark}

\begin{proposition}
\label{prop:archimedean_part_global_space_of_sav}
Let $(A,\|\cdot\|)$ be an integral structure for $K$ and denote by $V$ be the associated global space of pseudo-absolute values. We assume that $V_{\ar}\neq\emptyset$. Recall that we denote by $\va_{\infty}$ the usual absolute value on $\bQ$. Let $\epsilon:V_{\ar}\to]0,1]$ be the function mapping $x\in V_{\ar}$ to the unique $\epsilon(x)\in]0,1]$ such that the residual absolute value on $\kappa_x$ induces the absolute value $\va_{\infty}^{\epsilon(x)}$ on $\bQ$. Then, by extending the definition of $\epsilon$ to $V$ by setting $\epsilon(x):=0$ if $x\in V_{\um}$, $\epsilon$ is a continuous function on $V$.
\end{proposition}

\begin{proof}
From the assumption that $\Omega_{\ar}\neq\emptyset$, $K$ has characteristic zero. By definition of $\epsilon$, for any $x\in V_{\ar}$, we have $\epsilon(x)=\log|2|_{x}/\log 2$. It follows that for any $x\in V$, we have
\begin{align*}
\epsilon(x) = \frac{\max\{0,\log|2|_{x}\}}{\log 2}.
\end{align*}
Since $\log|2|_{\cdot}$ is continuous on $V$, we deduce the continuity of the function $\epsilon$.
\end{proof}

To study the topology of global spaces of pseudo-absolute values, we make use of the constructions in (\cite{Poineau10}, \S 1.3).

\begin{notation}
Let $(A, \|\cdot\|)$ be an integral structure for $K$ and denote $V:= \cM(A)$. For any $x\in V$, we define the interval
\begin{align*}
I_x := \{\epsilon\in\bR_{>0} : \forall f \in A, |f|_{x}^{\epsilon} \leq |f|\}.
\end{align*}
For any $\epsilon \in I_x$, we can then show that $\va_x^{\epsilon}$ defines an element of $V$ that we denote by $x^{\epsilon}$. In the case where $I_x$ has $0$ as a lower bound, we can extend this definition to $\epsilon = 0$ by defining $x^{0}$ as the pseudo-absolute value on $K$ defined by
\begin{align*}
\fonction{|\cdot|_{x}^0}{A}{\bR_{\geq 0}}{f}{
\left\{\begin{matrix}
&0 \text{ if } |f|_x = 0,\\
&1 \text{ if } |f|_x \neq 1.
\end{matrix}\right.}
\end{align*}
\end{notation}

\subsection{Examples of integral structures}

\begin{example}
\label{example:integral_structure}
\begin{itemize}
	\item[(1)] Let $\|\cdot\|$ be any Banach norm on $K$. Then $(K,\|\cdot\|)$ is an integral structure for $K$. This is the case when $K$ is either a complete valued field or a hybrid field, namely the norm $\|\cdot\|$ is of the form $\|\cdot\| = \max\{\va,\va_{\triv}\}$, where $\va$ denotes a non-trivial absolute value on $K$.
	
	\item[(2)] Let $K$ be a number field with ring of integers $\cO_K$. The latter is a Prüfer domain (cf. Example \ref{example:Prufer_rings} (2)) with fraction field $K$. Let $\|\cdot\|_{\infty} := \max_{\sigma \hookrightarrow \bC} \va_\sigma$, where $\sigma$ runs over the set of embeddings of $K$ into $\bC$. Then $(\cO_K,\|\cdot\|_{\infty})$ is an integral structure for $K$.
	
	\item[(3)] Let $U$ be a non-compact Riemann surface and $K:=\cM(U)$ denote its field of meromorphic functions. Then the ring of holomorphic functions $A := \cO(U)$ is a Prüfer domain (cf. Example \ref{example:Prufer_rings} (3)) with fraction field $K$. Let $C\subset U$ be a compact subset. Then defined a norm on $A$ as follows. 
	\begin{align*}
	\forall f\in A,\quad \|f\|_{C,\hyb} := \max\{\|f\|_C,\|f\|_{\triv}\},
\end{align*}	 	
where $\|\cdot\|_C$ denotes the supremum norm on $C$ and $\|\cdot\|_{\triv}$ denotes the trivial norm on $A$. Then $(A,\|\cdot\|_{C,\hyb})$ is a Banach ring and thus $(A,\|\cdot\|_{C,\hyb})$ is an integral structure for $K$.
	
	\item[(4)] Let $R>0$ and let $\overline{D(R)}$ denote the complex closed disc of radius $R$. We denote respectively by $A=\cO(\overline{D(R)})$ and $K=\cM(\overline{D(R)})$ the ring of germs of holomorphic functions and the field of germs of meromorphic functions on $\overline{D(R)}$. Then $A$ is a principal ideal domain with field of fractions $K$ (cf. Example \ref{example:Prufer_rings} (4) and Proposition \ref{prop:holomorphic_functions_closed_disc_PID}) and a Prüfer domain. Let $\|\cdot\|_R$ denote the supremum norm on $\overline{D(R)}$  and define $\|\cdot\|_{R,\hyb}:=\max\{\|\cdot\|_R,|\cdot|_{\triv}\}$, where $|\|\cdot\|_{\triv}$ denotes the trivial norm on $A$. Then $(A,\|\cdot\|_{R,\hyb})$ is a Banach ring and $(A,\|\cdot\|_{R,\hyb})$ defines an integral structure for $K$.
\end{itemize}
\end{example}

\begin{comment}
\begin{lemma}
\label{lemmeproprieteselementairesdeA}
\begin{itemize}
	\item[(i)] $A$ est un anneau intègre de corps de fractions $K$.
	\item[(ii)] $A$ est intégralement clos.
\end{itemize}
\end{lemma}

\begin{proof}
$A$ est un anneau intègre et son corps des fractions est inclus dans $K$, puisque tous les $A_x$ le sont (pour $x\in \ZR(K)$). Let $x_0\in \ZR(K)$, alors tout élément de $K$ peut s'écrire sous la forme $a/b$ avec $a,b\in A_{x_0}$. Let $x\in \ZR(K)$, on a quatre possibilités selon que $a$ ou $b$ appartiennent à $A_x$. Dans le cas où $a\in A_x$ et $b\notin A_x$, on a $a/b = ab^{-1}/1$ et donc s'écrit comme un quotient d'éléments de $A_x$. Dans le cas où $b\in A_x$ et $a\notin A_x$, on a $a/b= 1/a^{-1}b$. Dans le cas où $a,b\notin A_x$, l'égalité $a/b = b^{-1}/a^{-1}$ permet de conclure. Le dernier cas est trivial. Ce qui démontre $(i)$. 

$A$ est par définition une intersection d'anneaux de valuation de $K$. (\cite{BouAC}, Ch.V \S 1.3 Corollaire 2) montre que $A$ est intégralement clos.
\end{proof}
\end{comment}

\subsection{Tame global spaces of pseudo-absolute values}

The following definition is motivated by the explicit study of global spaces of pseudo-absolute values we shall consider in the context of Nevanlinna theory.

\begin{definition}
\label{def:tame_space}
Let $(A, \|\cdot\|)$ be an integral structure on $K$ and denote $V:= \cM(A)$. $V$ is called \emph{tame} if the following conditions are satisfied:
\begin{itemize}
\item[(i)] $v_{\triv}\in V$;
\item[(ii)] for any $x\in V_{\um}$ and any $f\in A$, the inequality
\begin{align*}
|f|_x \leq 1
\end{align*}
is satisfied;
\item[(iii)] $(A, \|\cdot\|)$ is a uniform Banach ring (cf. Definition \ref{def:uniform_Banach_ring}).
\end{itemize}
\end{definition}

\begin{example}
\label{example:tamed_spaces}
The spectra of Example \ref{example:integral_structure} are tame spaces. This is immediate in the case of the spectrum of the ring of integers of a number field. Let us show this fact for the case of the hybrid spectrum of the ring of analytic functions on a non-compact Riemann surface. We use the notation of Example \ref{example:integral_structure} (3). Let $x \in V_{\um}\setminus V_{\sn}$. Then for all $f \in A\setminus\{0\}$, we have $|f|_x \leq \|f\|_{A} = \max\{\|f\|_{C},1\}$. For all $a \in \bC$, we have $|af|_x = |f|_x$ (since the restriction of $\va_x$ to $\bC$ is necessarily trivial). By choosing $a$ of sufficiently small modulus, we obtain $\|af\|_{C} \leq 1$ and the inequality
\begin{align*}
|f|_x \leq |af|_x \leq \|af\|_A = 1.
\end{align*}
In \S \ref{subsub:global_space_Nevanlinna_open_disc}, we will see that the norm $\|\cdot\|_A$ is uniform.
\end{example}

The following proposition ensures that the notion of tame space is compatible with algebraic extensions of the field of fractions, which is fundamental in what follows.

\begin{proposition}
\label{prop:algebraic_extension_of_tame_spaces}
Let $(A, \|\cdot\|)$ be an integral structure on $K$ and assume that $V:= \cM(A)$ is tame. Let $L/K$ be an algebraic extension and let $B$ denote the integral closure of $A$ in $L$. Then $B$ can be equipped with a norm $\|\cdot\|_B$ such that
\begin{itemize}
\item[(1)] $(B,\|\cdot\|_B)$ is a Banach ring;
\item[(2)] $V_L := \cM(B)$ is a tame space;
\item[(3)] the inclusion morphism $(A,\|\cdot\|) \to (B,\|\cdot\|_B)$ is an isometry.
\end{itemize}
Moreover, $V_L$ is "universal" in the following sense. If $v\in M_L$ is such that its restriction to $A$ belongs to $V$, then $v \in V_L$.
\end{proposition}

\begin{proof}
\textbf{Case 1: the extension} $L/K$ \textbf{is finite and Galois.} Let us show the existence of $\|\cdot\|_B$. First note that $B$ is a Prüfer domain (cf. Proposition \ref{prop:prop_Prufer_domains} (2)). Define the set
\begin{align*}
E := \{\va : \va \text{ is a multiplicative seminorm on } B\text{ and }\va_{|A} \in V\}.
\end{align*}
Let us show that the application
\begin{align*}
\fonction{\|\cdot\|_B}{B}{\bR_{\geq 0}}{b}{\sup_{x\in E} |b(x)|}
\end{align*}
defines a Banach norm on $B$. Let $b\in B\setminus\{0\}$, we \emph{a priori} have $\|b\|_B \in [1,+\infty]$ (since the trivial absolute value belongs to $E$). To show that $\|b\|_B < +\infty$, we will use the following lemma.

\begin{lemma}
\label{lemma:Newton_sums_uniformly_bounded}
Let $b\in B$ and denote by $P=(T-\alpha_1)\cdots(T-\alpha_d) \in A[T]$ its minimal polynomial over $K$. Then there exits a constant $M>0$ such that
	\begin{align}
	\label{eq:Newton_sums_uniformly_bounded}
	\displaystyle\forall v\in V,\quad \exists N_0 \geq 0,\quad \forall N>0, \quad\left|\sum_{j=1}^{d} \alpha_{j}^N\right|_v \leq d^{N+N_0}M^N.
	\end{align}
\end{lemma}

\begin{proof}
Write $P=a_dT^{d}+\cdots + a_0$ and denote $M:= \max_{0\leq i\leq n} \|a_i\|. $. We fix $v=(\va_v,A_v,\m_v,\kappa_v)\in V$. For any integer $N>1$, define $\lambda_N := \sum_{j=1}^{d} \alpha_{j}^N$ and 
\begin{align*}
\Lambda := \max_{1\leq j\leq d} |\lambda_j|_v.
\end{align*}
Let $N_0>0$ be an integer such that $\Lambda \leq d^{N_0}$. Let us show by induction on $N$ that (\ref{eq:Newton_sums_uniformly_bounded}) holds for all $N>0$. Let $N$ be an integer. Assume first that $N\leq d$. Then we have
\begin{align*}
|\lambda_N|\leq \Lambda \leq d^{N_0+N}M^N,
\end{align*}
The second inequality comes from $M\geq 1$ (since $v_{\triv}\in V$). We now assume that $N>d$ and that (\ref{eq:Newton_sums_uniformly_bounded}) holds for any $N'<N$. Then, from Newton identities, we have
\begin{align*}
\lambda_N = \displaystyle\sum_{k=N-d}^{N-1}(-1)^{N-1+k} a_{N-k}\lambda_k,
\end{align*}
where
\begin{align*}
a_{j} := \sum_{1\leq n_1 < \cdots < n_j\leq d} \alpha_{n_1}\cdots\alpha_{n_j}
\end{align*}
for any $j\in\{1,...,d\}$. Note that for all $j\in\{1,...,d\}$, $a_j$ is a coefficient of $P$ (up to sign). Hence $|a_j|\leq M$. Then the induction hypothesis yields
\begin{align*}
\displaystyle|\lambda_N|_v \leq dM \max_{N-d\leq k \leq N-1} d^{k+N_0}M^k \leq d^{N+N_0}M^N.
\end{align*}
Hence the conclusion.
\end{proof}

We are now able to show $\|b\|_B < + \infty$. Let $\{\alpha_1,...,\alpha_d\}$ be the Galois orbit of $b$. Fix an arbitrary pseudo-absolute value $v\in V$. Proposition \ref{prop:max_sav_finite_extension} and Lemma \ref{lemma:Newton_sums_uniformly_bounded} give the existence of $M>0$ such that  
\begin{align*}
\max_{w|v} |b|_w = \displaystyle\limsup_{N\to +\infty} \left|\sum_{j=1}^{d} \alpha_{j}^N \right|^{\frac{1}{N}}_v \leq dM.
\end{align*} 
Since both $d$ and $M$ are independent on $v$, we have $\|b\|_B \leq dM < +\infty$. It is immediate to see that $\|\cdot\|_B$ defines a norm on $B$. Since the trivial absolute value on $L$ belongs to $E$, the topology on $B$ induced by $\|\cdot\|_B$ is discrete. In particular, $(B,\|\cdot\|_B)$ is a Banach ring. Furthermore, for any $a \in A$, we have $\|a\|_B = \|a\|$. Hence we have an isometric embedding $(A,\|\cdot\|) \to (B,\|\cdot\|_B)$. This concludes the proof of $(1)$ and $(3)$.

We now show that $V_L := \cM(B)$ is a tame space. Note that $V_L$ contains the trivial absolute value. Let $x=(\va_x,B_x,\m_x)\in V_{L,\um}$. Let $f\in B$ whose Galois orbit is denoted by $\{\alpha_1,...,\alpha_d\}$. Let us show that $|f|_x \leq 1$. Let $v= (\va_v,A_v,\m_v) := \pi_{L/K}(x) \in V$ denote the restriction of $x$ to $K$. Then Lemma \ref{lemma:max_extension_absolute_values} yields 
\begin{align*}
\displaystyle|f|_x \leq \max_{x'\in \pi_{L/K}^{-1}(v)} |f|_{x'} = \displaystyle\limsup_{N\to +\infty} \left|\sum_{j=1}^{d} \alpha_{j}^N \right|^{\frac{1}{N}}_v \leq 1,
\end{align*}
The last inequality comes from the fact that, for any $j\in\{1,...,d\}$, for any $N>0$, $\sum_{j=1}^{d} \alpha_{j}^N \in A$. 
Finally, (\cite{Berko90}, Theorem 1.3.1) implies that the norm $\|\cdot\|_B$ is uniform and $V_L$ is tame.

\textbf{Case 2: the extension} $L/K$ \textbf{is finite separable.} Let $L'/L/K$ be the normal closure of $L/K$ and denote by $B'$ the integral closure of $A$ in $L'$, which is equal to the integral closure of $B$ in $L'$. From the finite Galois case, there exists a norm $\|\cdot\|_{B'}$ on $B'$ such that $(B',\|\cdot\|_{B'})$  (1)-(3) hold. Since $B$ is a subring of $B'$, the restriction of $\|\cdot\|_{B'}$ to $B$, denoted by $\|\cdot\|_B$, induces the discrete topology $B$, hence is a Banach norm. We also have an isometric embedding $(A,\|\cdot\|)\to (B,\|\cdot\|_B)$ and $\cM(B)$ is tame. 

\textbf{Case 3: the extension} $L/K$ \textbf{is infinite and separable.} Let $\cE_{L/K}$ denote the set of finite sub-extensions of $L/K$. $\cE_{L/K}$ is a directed set with respect to the inclusion relation and we have
\begin{align*}
L = \displaystyle\bigcup_{K'\in \cE_{L/K}} K' \quad \text{and} \quad B = \displaystyle\bigcup_{K'\in \cE_{L/K}} A_{K'},
\end{align*}
where, for all $K'\in \cE_{L/K}$, $A_{K'}$ denotes the integral closure of $A$ in $K'$. For any $K'\in \cE_{L/K}$, we denote by $\|\cdot\|_{A_{K'}}$ the norm on $A_{K'}$ constructed in the finite case. If $K''/K'/K$ are sub-extensions in $\cE_{L/K}$, we have the compatibilities
\begin{center}
% https://tikzcd.yichuanshen.de/#N4Igdg9gJgpgziAXAbVABwnAlgFyxMJZABgBpiBdUkANwEMAbAVxiRAAoBBAfWAGkA5AIC+pADpiAPhIDGUCDgmTePfiOEBKEKPSZc+QigCM5KrUYs2XXoNFLZ8xVJU31WnSAzY8BIiaNm9MysiByc4lIOCkruZjBQAObwRKAAZgBOEAC2SABM1DgQSEYeGdl5BUWIxKWZOYgmIIVIZObBbBL4OHSuAPR8wtoUwkA
\begin{tikzcd}
{(A_{K''},\|\cdot\|_{A_{K''}})} & {(A_{K'},\|\cdot\|_{A_{K'}})} \arrow[l, "\iota_{K'/K}"] \\
                               & {(A, \|\cdot\|)} \arrow[u] \arrow[lu]                   
\end{tikzcd}
\end{center}
where the arrows are the isometric embeddings previously constructed. Thus we have a filtered direct system $(A_{K'},\|\cdot\|_{A_{K'}})_{K'\in\cE_{L/K}}$ of Banach $A$-algebras whose arrows are isometric embeddings. One can prove that the direct limit of this direct system exists in $\Ban^{\leq 1}_{A-\alg}$ and using the fact that all $A_{K'}$ are discrete its underlying set can be identified with $B$. $B$ is Prüfer (cf. Proposition \ref{prop:prop_Prufer_domains} (2)). Thus we have an integral structure $(B,\|\cdot\|_B)$ on $L$. Let us show that $V_L := \cM(B) \cong \varprojlim_{K'\in \cE_{L/K}} \cM(A_{K'})$ is tame. From the description of $V_L$ as an inverse limit, we see that $V_L$ contains the trivial absolute value and that the norm $\|\cdot\|_B$ is uniform. Moreover, by definition of $\|\cdot\|_B$, for any $f\in B$, for any $x\in V_{L,\um}$, we have $|f(x)|\leq 1$. 

\textbf{Case 4: general case.} Let $K'/K$ be the separable closure of $K$ in $L$ and let $q$ denote the degree of the purely inseparable extension $L/K'$. From the latter case, we have a norm $\|\cdot\|_{A'}$ on $A'$, the integral closure of $A$ in $K'$, satisfying (1)-(3). Define a map
\begin{align*}
\|\cdot\|_{B} : (b\in B) \mapsto \|b^q\|_{A'} \in\bR_{\geq 0}.
\end{align*} 
Then $\|\cdot\|_{B}$ is a norm on $B$ satisfying (1)-(3). 

To conclude the proof of the proposition, it remains to show that $V_L$ is "universal". Let $v\in M_L$ and denote by $B_v$ its finiteness ring. As a valuation ring of $L$ containing $A$, it is an overring of $B$. By construction of $\|\cdot\|_B$, we have $|b|_v\leq \|b\|_B$ for all $b\in B$, i.e. $v\in V_L$.
\end{proof}

The following propositions ensure that the ultrametric part of global spaces of pseudo-absolute values arising from integral structures enjoy sufficiently nice properties.  

\begin{proposition}
\label{prop:flow_tame_space}
Let $(A, \|\cdot\|)$ be an integral structure on $K$ and assume that $V:= \cM(A)$ is tame. Let $x\in V_{\um}\setminus V_{\sn}$, i.e. an ultrametric absolute value on $K$, and assume that $x$ is  non-trivial. Then, for all $\epsilon \in [0,+\infty[$, $x^{\epsilon}\in V$. Furthermore, the pseudo-absolute value on $K$ defined by 
\begin{align*}
\fonction{|\cdot|_{x}^{\infty}}{A}{\bR_{\geq 0}}{f}{
\left\{\begin{matrix}
&0 \text{ if } |f|_x < 1,\\ 
&1 \text{ if } |f|_x = 1,
\end{matrix}\right.}
\end{align*}
belongs to $V$. We denote it by $x^{\infty}$.
\end{proposition}

\begin{proof}
As $V$ is tame, we have $x^0= v_{\triv}\in V$. For any $\epsilon\in \bR_{>0}$, $x^{\epsilon}$ is an ultrametric and nontrivial absolute value on $K$. For any $f\in A$, we have $|f|_x\leq 1$ and thus $|f|_x^{\epsilon} \leq 1$. Therefore $x^{\epsilon}\in V$ (since $1 \leq \|f\|$). For any $f\in A$,  we have $|f|_x^{\infty} \leq |f|_x \leq 1 \leq \|f\|$. Hence $x^{\infty} \in V$.
\end{proof}

\begin{proposition}
\label{prop:branch_tame_space}
Let $(A, |\cdot|)$ be an integral structure for $K$ and assume that $V:= \cM(A)$ is tame. Let $x\in V_{\um}\setminus V_{\sn}$ and assume that $x$ is  nontrivial. Then the map
\begin{align*}
\fonction{x^{\cdot}}{[0,+\infty]}{V}{\epsilon}{x^{\epsilon}}
\end{align*}
induces an homeomorphism on its image.
\end{proposition}

\begin{proof}
Proposition \ref{prop:flow_tame_space} implies that the map $x^{\cdot}$ is well-defined. Since for any $f\in A$, the map 
\begin{align*}
\fonction{\varphi_f}{[0,+\infty]}{\bR_{\geq 0}}{\epsilon}{|f|_{x}^{\epsilon}}
\end{align*}
is continuous, $x^{\cdot}$ is continuous. Thus $x^{\cdot}$ is a continuous bijection between compact Hausdorff spaces, and consequently a homeomorphism.
\end{proof}

%TODO Ce corollaire est faux, on sait juste que dans le cas où les valeurs absolues ultramétriques sur la corps $K$ sont discrètes, on a une description explicite d'un système fondamental de voisionages dans $V_{\um}$ d'un point d'une branche.
\begin{comment}
\begin{corollary}
Let $K$ un corps muni d'une integral strucure $(A,\|\cdot\|)$ pour lequel on note $V:= \cM(A)$ que l'on suppose sympathique. Let $x\in V\setminus V_{\sn}$ que l'on suppose non triviale. Alors il existe un voisinage ouvert $U$ contenant $x$ tel que tous les points de $U$ correspondent à des valeurs absolues sur $K$ équivalentes à $\va_x$.
\end{corollary}
\end{comment}

The following definition allows to get rid of pseudo-absolute values whose finiteness ring is of rank greater than $1$.

\begin{definition}
\label{def:sav_of_rank_leq1}
A pseudo-absolute value $x\in M_K$ with finiteness ring $A_x$ is called of \emph{rank at most 1} if $A_x$ is a valuation ring of rank $\leq 1$. 

Let $(A,\|\cdot\|)$ be an integral structure for $K$ and denote $V=\cM(A)$. We define
\begin{align*}
    V_{\leq 1} := \{x\in V : x \text{ is of rank at most } 1\}.
\end{align*}
Then $V_{\leq 1}$ is a topological subspace of $V$, with the subset topology. 
\end{definition}

\begin{proposition}
\label{prop:tame_space_rank_leq1}
Let $(A, \|\cdot\|)$ be an integral structure for $K$ and assume that $V:= \cM(A)$ is tame. Let $x=(\va_x,A_x,\m_x,\kappa_x) \in V_{\sn}\cap V_{\leq 1}$. Then the following assertions hold.
\begin{itemize}
    \item[(1)] The residue absolute value on $\kappa_x$ is either Archimedean or trivial.
    \item[(2)] Further assume that $x\in V_{\um}$. Then there exists $v\in V_{\leq 1}$ such that $x = v^{\infty}$ (cf. Proposition \ref{prop:flow_tame_space}).
    %TODO Finir la preuve, ou voir si elle est vraie. \item[(3)] $V_{\leq 1}$ est fermé dans $V$ et est donc un sous-espace compact.
\end{itemize}
\end{proposition}

\begin{proof}
We first show $(1)$. Let $x=(\va_x,A_x,\m_x,\kappa_x)\in V_{\sn}\cap V_{\leq 1}$ and assume $x$ is ultrametric. Since $V$ is tame, for any $a\in A$, we have $|a|_x \leq 1$. Furthermore, the canonical homomorphism $A\to \kappa_x$ is surjective. Hence, for any $\overline{a}\in\kappa_x$, we have $|\overline{a}|_x\leq 1$, i.e. $\kappa_x$ is the valuation ring of the residue absolute value of $x$, i.e. $x$ is residually trivial.

We now prove (2). Let $x=(\va_x,A_x,\m_x)\in V_{\sn}\cap V_{\leq 1}$ be ultrametric. From (1), $\kappa_x$ is trivially valued. Let $v$ be a valuation (necessarily of rank $1$) on $K$ with valuation ring $A_x$ and denote by $\va_v$ the corresponding absolute value on $K$. Note that, since $x,v_{\triv}\in V$, we have $A\subset A_x$ and, for any $a\in A$, $|a|_v \leq 1 \leq \|a\|$. Hence $v\in V_{\leq 1}$ and $x= v^{\infty}$.

%TODO finir la preuve. Il reste à montrer que $V_{\leq 1}$ est un fermé de $V$. On va montrer que $V_{\leq 1}$ est l'adhérence de l'ensemble des éléments de $V_{\um}$ qui sont des valeurs absolues sur $K$. En vertu de la Proposition \ref{propbrancheespacesympathique}, pour toute valeur absolue non triviale $v\in V$, $v^{\infty} \in \overline{I_{v}}$. 
\end{proof}

\begin{remark}
\label{rem:description_V_um_leq1}
With the notation of Proposition \ref{prop:tame_space_rank_leq1}, we have the set-theoretic description of $V_{\um, \leq 1}$ as
\begin{align*}
\bigsqcup_{v \in \cP} [0,+\infty]/\sim,
\end{align*}  
where: 
\begin{itemize}
	\item $\cP$ denotes the set of equivalence classes of nontrivial ultrametric absolute values on $K$;
	\item for any $v\in\cP$, $[0,+\infty]$ denotes the branch introduced in Proposition \ref{prop:branch_tame_space};
	\item $\sim$ denotes the equivalence relation which identifies the extremity $0$ of each branch.
\end{itemize}  
\end{remark}

\begin{proposition}
\label{prop:topological_description_Dedekind}
We use the notation of Remark \ref{rem:description_V_um_leq1}. Assume that $A$ is Dedekind. Then the bijection 
\begin{align*}
V_{\um} = V_{\um,\leq 1} \cong \bigsqcup_{v \in \cP} [0,+\infty]/\sim
\end{align*}  
is an homeomorphism.
\end{proposition}

\begin{proof}
This follows directly for the description of the Berkovich analytic spectrum of a trivially valued Dedekind ring (e.g. \cite{LemanissierPoineau24}, Example 1.1.17).
\end{proof}

%TODO il faut finir la partie sur les corps de fonctions et Nevanlinna sur un disque compact. 
\subsection{Examples}
\label{sub:examples_global_set_sav}

We can explicitly describe global spaces of pseudo-absolute values for various examples.

\subsubsection{Function field over $\bC$}

Let $K=\bC(T)$. Let $R>0$. Denote by $\|\cdot\|_{R}$ the supremum norm of polynomial functions on the closed disc $\overline{D(R)}$ of radius $R$. Let $\|\cdot\|$ denote the hybrid norm $\|\cdot\| := \max \{\|\cdot\|_{\triv},\|\cdot\|_{R}\}$, where $\|\cdot\|_{\triv}$ denotes the trivial norm on $\bC[T]$. Then $(\bC[T],\|\cdot\|)$ defines an integral structure for $K$. We describe the associated global space of pseudo-absolute values $V_{R}$.

\begin{proposition}
\label{prop:space_of_sav_function_field_over_C}
We use the same notation as above.
\begin{itemize}
	\item[(1)] We have homeomorphisms 
\begin{align*}
V_{R,\ar} \cong \overline{D(R)} \times ]0,1], \quad V_{R,\um} \cong \bP^{1,\triv}_{\bC} \cap \{|T|\leq 1\},
\end{align*}	
 where $\bP^{1,\triv}_{\bC}$ is the Berkovich analytic space associated with $\bP^{1}_{\bC}$, where $\bC$ is trivially valued.
 	\item[(2)] $V_{R,\ar}$ is dense in $V_{R}$
\end{itemize}
\end{proposition}

%modif Jérôme
\begin{proof}  
\textbf{(1)} Let $\varphi$ be the map mapping $(z,\epsilon)\in \overline{D(R)}\times]0,1]$ to $v_{z,\epsilon,\ar}\in M_K$. First let us prove that $\varphi$ has image contained in $V_{R,\ar}$. For this purpose, it suffices to show that, for any $(z,\epsilon)\in \overline{D(R)}\times]0,1]$, we have 
\begin{align*}
\forall P \in \bC[T], \quad |P(z)|_{\infty}^{\epsilon} \leq \|P\|.
\end{align*}
Let $(z,\epsilon)\in \overline{D(R)}\times]0,1]$ and fix an arbitrary non-zero polynomial $P\in \bC[T]$. First assume that $|P(z)|_{\infty}\leq 1$. Then $|P(z)|_{\infty}^{\epsilon} \leq 1 = \|P\|_{\triv} \leq \|P\|$. Now assume that $|P(z)|_{\infty}>1$. Then we have the inequalities
\begin{align*}
1 < |P(z)|_{\infty}^{\epsilon} < |P(z)|_{\infty} \leq \|P\|_R = \|P\|. 
\end{align*} 
Therefore $\varphi$ has image contained in $V_{R,\ar}$.

Let us now prove that $\varphi$ has exactly image $V_{R,\ar}$. We first consider the case where $z\notin D(\overline{R})$ and $\epsilon =1$. Then if $P(T):= T+1 \in \bC[T]$, we have 
\begin{align*}
 \left\{\begin{matrix}
|P(z)|_{\infty} = |z|_{\infty} + 1 > 1 = \|P\|_{\triv},\\
|P(z)|_{\infty} = |z|_{\infty} + 1 > R + 1 = \|P\|_R.
\end{matrix}\right.
\end{align*}
Therefore $v_{z,1,\ar} \notin V_{R,\ar}$.

Now let $z\notin \overline{D(R)}$ and $\epsilon\in]0,1[$. Then we construct a polynomial $Q\in\bC[T]$ such that $|Q(z)|^{\epsilon}_{\infty} > \|Q\|$. For any $a\in \bR_{>0}$, denote 
\begin{align*}
P_a(T) := \frac{T}{a} \in \bC[T].
\end{align*}
Then for any $a>0$, we have $\|P_a\|_{R}= R/a$ and
\begin{align*}
|P_a(z)|^{\epsilon}_{\infty}>\|P_a\| \Leftrightarrow  \left\{\begin{matrix}
\frac{|z |_{\infty}^{\epsilon}}{a^{\epsilon}} > 1 = \|P_a\|_{\triv},\\
\frac{|z |_{\infty}^{\epsilon}}{a^{\epsilon}} > \frac{R}{a} = \|P_a\|_{R}.
\end{matrix}\right. 
\end{align*}
Let $u : (a\in ]0,|z|_{\infty}[) \mapsto R/a^{1-\epsilon}$. Then $u$ is a continuous function. Note that 
\begin{align*}
R < |z|_{\infty} \Leftrightarrow \displaystyle\lim_{a \to |z|^{-}_{\infty}}u(a) =  \frac{R}{|z|_{\infty}^{1-\epsilon}} < |z|_{\infty}^{\epsilon}.
\end{align*}
Moreover, for any $a\in]0,|z|_{\infty}[$, we have 
\begin{align*}
|P_a(z)|_{\infty}^{\epsilon}>\|P\|_{R} \Leftrightarrow u(a) < |z|_{\infty}^{\epsilon}.
\end{align*}
Since $\lim_{a \to |z|^{-}_{\infty}}u(a)<|z|_{\infty}^{\epsilon}$, by continuity, there exists $a\in ]0,|z|_{\infty}[$ such that $u(a) < |z|_{\infty}^{\epsilon}$ and we denote $Q:=P_a$. From what precedes, we have $|Q(z)|_{\infty}^{\epsilon}>\|Q\|$, thus $v_{z,\epsilon,\ar}\notin V_{R,\ar}$.

Combining the two above paragraphs, $\varphi$ defines a bijection between $\overline{D(R)} \times ]0,1]$ and $V_{R,\ar}$. It is a homeomorphism by definition of the topologies of its domain and codomain.

We now prove the second homeomorphism. Let $|\cdot|\in V_{R,\um}$. As $|\cdot|_{|\bC[T]}\leq 1$, we deduce that the restriction of $|\cdot|$ is the trivial absolute value. We distinguish two cases. The first one is the case where $|\cdot|$ is a usual absolute value. Then $|\cdot|$ is of the form $|\cdot|_{z,c,\um} : f\in K \mapsto \exp(-c\ord(f,z))\in\bR_{>0}$, where $z\in \bC \cup \{\infty\}$ and $c\geq 0$ (using the conventions $\ord(\cdot,\infty) = -\deg(\cdot)$ and $|\cdot|_{z,0,\um} = \va_{\triv}$ for all $z\in \bC\cup\{\infty\}$). The second case is when $|\cdot|$ has a non-zero kernel. Then there exists $z\in \bC\cup\{\infty\}$ such that $|\cdot| = \va_{z,\infty,\um}$, where
\begin{align*}
\fonction{|\cdot|_{z,\infty,\um}}{\bC[T]_{(T-z)}}{\bR_{\geq 0}}{f}{
\left\{\begin{matrix}
&0 \text{ if } f(z) = 0,\\ 
&1 \text{ if } f(z) \neq 0. 
\end{matrix}\right.}
\end{align*}
Thus $\va$ can be identified with an element of $\bP^{1,\triv}_{\bC}$. Moreover, the condition $|\cdot|_{|\bC[T]}\leq 1$ implies that $|\cdot|$ can be identified with an element of $\bP^{1,\triv}_{\bC} \cap \{|T|\leq 1\}$, this yields a map $\psi : V_{R,\um} \to \bP^{1,\triv}_{\bC} \cap \{|T|\leq 1\}$ which is continuous by definition of the topologies. Conversely, by definition, any element $x\in \bP^{1,\triv}_{\bC} \cap \{|T|\leq 1\}$ gives rise to a ultrametric pseudo-absolute value $\va \in M_{K,\um}$. As $x\in \{|T|\leq 1\}$, we deduce $|\cdot|_{|\bC[T]}\leq 1 \leq \|\cdot\|$ and therefore $|\cdot|\in V_{R,\um}$. This construction yields an inverse to $\psi$. To conclude, it suffices to remark that $V_{R,\um}$ compact Hausdorff and that $\bA^{1,\triv}_{\bC} \cap\{|T|\leq 1\}$ is Hausdorff.

\textbf{(2)} We refer to the proof of Proposition \ref{prop:global_space_of_sav_compact_disc} (2).
\end{proof}

\subsubsection{Nevanlinna theory: open disc}
\label{subsub:global_space_Nevanlinna_open_disc}

A motivation for this work is to study problems arising from Nevanlinna theory. Let $0<R\leq +\infty$ and let $K_R:=\cM(D(R))$ be the field of meromorphic functions on $D(R):=\{z\in\bC : |z|< R\}$, where $D(+\infty) = \bC$. Let $R'<R$ and denote by $\|\cdot\|_{R'}$ the restriction to $\cO(D(R))$ of the hybrid norm on the disc $\overline{D(R')}$ (cf. Example \ref{example:integral_structure} (3)). Then $(\cO(D(R)),\|\cdot\|_{R'})$ defines an integral structure for $K$. Denote by $V_{R,R'}$ the corresponding global space of pseudo-absolute values. In what follows, we explicitly describe $V_{R,R',\leq 1}$. 

\begin{proposition}
\label{prop:description_Archimedean_sav_Nevanlinna_open_disc}
Let $v=(\va,A,\m,\kappa)\in V_{R,R',\ar}$. Then there exist $z\in \overline{D(R')}$ and $\epsilon\in ]0,1]$ such that $v=v_{z,\epsilon,\ar}$ (cf. Example \ref{ex:sav} (3)). Furthermore, we have a homeomorphism  $V_{R,R',\ar} \cong ]0,1]\times \overline{D(R')}$ which maps any $v_{z,\epsilon,\ar}\in V_{R,R',\ar}$ to $(\epsilon,z)\in ]0,1]\times \overline{D(R')}$.
\end{proposition}

\begin{proof}
Recall that $\va_{\infty}$ denote the usual complex Archimedean value. Note that, since $v\in V_{R,R'}$, we have an inclusion $\cO(D(R))\subset A$. In particular, $\bC$ is a subfield of $A$ and thus $\bC\subset A^{\times}$. We have arrows $\bC \subset A^{\times}  \twoheadrightarrow \kappa$ and an extension $\bC \to \kappa$. Since $v$ induces an Archimedean absolute value on $\kappa$, the Gelfand-Mazur theorem (\cite{BouAC}, Chapitre VI, \S 6, n\textsuperscript{o} 4, Théorème 1) ensures that $\kappa = \bC$ and that there exists $\epsilon\in]0,1]$ such that $\va = \va_{\infty}^\epsilon$ on $\kappa$. %En outre, $\m':=\m \cap \cO(D(R)) \in \Spec(\cO(D(R)))$ donc $\cO(D(R))_{\m'}$ est un anneau de valuation pour $K$ contenu dans $A$, donc est égal à $A$ (\cite{BouAC}, Chapitre VI, \S 1, n\textsuperscript{o} 2, Théorème 1). 

From the inclusion $\bC\hookrightarrow \cO(D(R))^{\times}$ we deduce that the $\bC$-algebra morphism  $\cO(D(R)) \to \bC$ is surjective whose kernel is the maximal ideal $\m' := \m\cap\cO(D(R))$ of $\cO(D(R))$. Proposition \ref{prop:ring_of_analytic_functions_open_Riemann_surface} implies that $\m'$ is principal and that there exists $z\in D(R)$ such that $\m'$ is the set of functions in $\cO(D(R))$ vanishing in $z$. Finally, the restriction of $v$ to $\cO(D(R))$ is the map $(f\in \cO(D(R))) \mapsto |f(z)|_{\infty}^{\epsilon}\in[0,+\infty]$. The condition $v\in V_{R,R'}$ ensures that $z\in \overline{D(R')}$. Hence the conclusion of the proof of the first statement of the proposition.

For any $x\in V_{R,R',\ar}$, denote by $z(x)$ the unique $z \in \overline{D(R')}$ such that $x = v_{z(x),\epsilon(x),\ar}$. Let $\varphi : V_{R,R',\ar} \to ]0,1]\times \overline{D(R')}$ be the map mapping any $x\in V_{R,R',\ar}$ to $(\epsilon(x),z(x))\in ]0,1]\times \overline{D(R')}$. Then $\varphi$ is bijective. Let us show $\varphi$ is a homeomorphism. To show that $\varphi$ is continuous, it is enough to show that so are the induced maps $\epsilon : x\in V_{R,R',\ar} \mapsto \epsilon(x) \in ]0,1]$ et $z : x\in V_{R,R',\ar} \mapsto z(x) \in \overline{D(R')}$. The first map is continuous (cf. \ref{prop:archimedean_part_global_space_of_sav}). The map $z$ is equal to the composition $V_{R,R',\ar} \to V_{R,\ar} \to \overline{D(R')}$, where $V_{R,\ar}$ is the Archimedean part described in Proposition \ref{prop:space_of_sav_function_field_over_C}. It remains to prove that $\varphi^{-1}$ is continuous. For this purpose, it suffices to show that, for any $f\in \cO(D(R))$, the map
\begin{align*}
\fonction{|f|\circ \varphi^{-1}}{]0,1]\times \overline{D(R')}}{\bR_{+}}{(\epsilon,z)}{|f(z)|_{\infty}^{\epsilon}}
\end{align*}
is continuous. This fact being certainly true, it concludes the proof.
\end{proof}

\begin{remark}
Let $K=\cM(U)$ denote the field of meromorphic functions on a (connected) open Riemann surface $U$. On can adapt the above proof to describe the Archimedean pseudo-absolute values on $K$: they are the pseudo-absolute values of the form $(f \in K)\mapsto |f(z)|_{\infty}^{\epsilon} \in [0,+\infty]$, for some $z\in U$ and $\epsilon \in ]0,1]$.
\end{remark}

We now study the ultrametric pseudo-absolute values in $V_{R,R'}$. We first exhibit some of them. For any $z\in D(R)$ and $c>0$, denote by $v_{z,c,\um}$ the absolute value $|\cdot|_{z,c,\um} : (f\in K_R) \mapsto \exp(-c\ord(f,z)) \in \bR_{>0}$ (it is an element $V_{R,R'}$ since for all $f\in\cO(U)$, we have $|f|_{z,c,\um} \leq 1$). Example \ref{ex:sav} (4) provides, for any $z\in D(R)$, a pseudo-absolute value $v_{z,\infty,\um}$ defined by 
\begin{align*}
\fonction{|\cdot|_{z,\infty,\um}}{A_z}{\bR_{\geq 0}}{f}{
\left\{\begin{matrix}
&0 \text{ si } f \in \m_z,\\ 
&1 \text{ si } f\notin \m_z,
\end{matrix}\right.}
\end{align*}
where $A_z$ and $\m_z$ denote respectively the set of elements of $K$ without pole and vanishing in $z$. Finally, for any $z\in D(R)$, we denote by $v_{z,0,\um}$ the trivial absolute value on $K$.

\begin{proposition}
\label{prop:description_sav_ultrametric_Nevanlinna_open_disc}
$V_{R,R',\leq 1, \um}$ is the set of pseudo-absolute values $v$ on $K$ such that there exist $z\in D(R)$ and $c\in [0,+\infty]$ such that $v= v_{z,c,\um}$.
\end{proposition}

\begin{proof}
Let $v\in V_{R,R',\leq 1}$. We first assume that $v\in V_{R,R'}\setminus V_{R,R',\sn}$, i.e. $v$ is a usual absolute value on $K$. Let $A \subset K$ denote the valuation ring of $v$. We have the inclusion $\cO(D(R))\subset A$ since $v\in V_{R,R'}$. Moreover, $v$ is a valuation of rank $1$, hence $v$ is either trivial or equivalent to the valuation $\ord(\cdot,z)$, for some $z\in D(R)$ (cf. Proposition \ref{prop:absolute_values_on_fields_of_meromorphic_functions}). Thus there exist $z\in D(R)$ and $c\in \bR_{+}$ such that $v = v_{z,c,\um}$.

We now assume that $v\in V_{R,R',\sn,\um}$. Since $V_{R,R'}$ is tame (Définition \ref{def:tame_space}) and, for any $x\in V_{\leq 1,\sn}$, we have $\kappa(x) = \bC$ and a surjection $\cO(D(R)) \to \kappa(x)$, Proposition \ref{prop:tame_space_rank_leq1} ensures that $x$ is of the form $v = v_{z,\infty,\um}$, for some $z\in D(R)$.
\end{proof}

\begin{proposition}
\label{prop:global_space_sav_Nevanlinna_open_disc}
We have homomorphisms $V_{R,R',\ar} \cong \overline{D(R')}\times]0,1]$ and
\begin{align*}
V_{R,R',\leq 1,\um} \cong \bigsqcup_{v \in \cP} [0,+\infty]/\sim,
\end{align*}
with the notation of Remark \ref{rem:description_V_um_leq1}.
\end{proposition}

\begin{proof}
This is a consequence of Propositions \ref{prop:description_Archimedean_sav_Nevanlinna_open_disc} and \ref{prop:description_sav_ultrametric_Nevanlinna_open_disc}. The only part needing a justification is the fact that the nontrivial ultrametric absolute values of $K_R$ are discrete. This comes from the first part of the proof of Proposition \ref{prop:description_sav_ultrametric_Nevanlinna_open_disc}.
\end{proof}

\subsubsection{Nevanlinna theory: compact disc}
\label{subsub:example_Nevanlinna_compact_disc}

Let $R>0$ and denote by $K_R:=\cM(\overline{D(R)})$ the field of (germs) of meromorphic functions on the closed disc $\overline{D(R)} := \{z\in\bC : |z|_{\infty} \leq R$. Let $A_R=\cO(\overline{D(R)})$ denote the ring of (germs of holomorphic functions on $\overline{D(R)}$. Equip $A$ with the hybrid norm $\|\cdot\|_R$ (cf. Example \ref{example:integral_structure} (4)). Then $(A_R,\|\cdot\|_R)$ is an integral structure for $K_R$. We now describe the corresponding global space of pseudo-absolute values $V_{R}$.

Firstly, $V_{R,\ar} \cong ]0,1]\times \overline{D(R)}$ where $(\epsilon,z)\in ]0,1]\times \overline{D(R)}$ is mapped to the pseudo-absolute value $v_{z,\epsilon,\ar}$ defined in Example \ref{ex:sav} (3) (see also Proposition \ref{prop:description_Archimedean_sav_Nevanlinna_open_disc}).

We now describe $V_{R,\um}$. By Propositions \ref{prop:holomorphic_functions_closed_disc_PID} and \ref{prop:holomorphic_functions_closed_disc_ideals}, a non-trivial ultrametric absolute value in $V_{R,\um}$ is given by a valuation $\ord(\cdot,z)$ for some $z\in \overline{D(R)}$. Hence it is of the form $|\cdot|_{z,c,\um} : (f\in K_R) \mapsto \exp(-c\ord(f,z))\in\bR_{>0}$ for some $c>0$ and $z\in\overline{D(R)}$. The remaining elements of $V_{R,\um}$ are of the form
\begin{align*}
\fonction{|\cdot|_{z,\infty,\um}}{\cO(\overline{D(R)})_{(T-z)}}{\bR_{\geq 0}}{f}{
\left\{\begin{matrix}
&0 \text{ si } f \in (T-z),\\ 
&1 \text{ si } f\notin (T-z),
\end{matrix}\right.}
\end{align*}
for some $z\in \overline{D(R)}$

%modif Jérôme
\begin{proposition}
\label{prop:global_space_of_sav_compact_disc}
\begin{itemize}
	\item[(1)] We have homeomorphisms 
\begin{align*}
V_{R,\ar} \cong ]0,1]\times \overline{D(R)}, \quad V_{R,\um} \cong \bigsqcup_{z \in \overline{D(R)}} [0,+\infty]/\sim,
\end{align*}	
where $\sim$ denotes the equivalence relation which identifies the extremity $0$ of each branch.
	\item[(2)] $V_{R,\ar}$ is dense in $V_{R}$.
\end{itemize}
\end{proposition}

\begin{proof}
The homeomorphisms in (1) follow from the above paragraphs together with Proposition \ref{prop:topological_description_Dedekind}. 

As usual, we denote by $\va_{\infty}$ the usual absolute value on $\bC$. We now prove (2). Let $(v_n)_{n\geq 0}$ be a sequence in $V_{R,\ar}$ converging to some $v\in V_{R}$. (1) implies that, for any integer $n\geq 0$, $v_n=v_{z_n,\epsilon_n,\ar}$, for some $z_n\in \overline{D(R)}$ and $\epsilon_n\in ]0,1]$. By compactness of $\overline{D(R)}$ and $[0,1]$, up to extracting subsequences, we may assume that there exists $(z,\epsilon)\in \overline{D(R)}\times[0,1]$ such that $\lim_{n\to+\infty} z_n=z$ and $\lim_{n\to+\infty} \epsilon_n=\epsilon$. Let us prove that $v=v_{z,\epsilon,\ar}$ if $\epsilon>0$ and $v=v_{z,\infty,\um}$ if $\epsilon=0$. 

We first assume that $\epsilon>0$. Since $|2|_v>1$, $v$ is Archimedean. Moreover, $|(T-z)|_{v} = \lim_{n\to+\infty}|z_n-z|^{\epsilon_n}_{\infty} = 0$. Thus $(T-z)$ belongs to the kernel of $v$ and $v$ is of the form $v_{z,\epsilon',\ar}$, for some $\epsilon'\in]0,1]$. $|2|_v=2^{\epsilon}=2^{\epsilon'}$ yields $\epsilon=\epsilon'$. 

We now consider the $\epsilon=0$ case. Thus $|2|_v=1$ and $v$ is ultrametric. Consider the sequence $u=(\epsilon_n\log|z_n-z|_{\infty})_{n\geq 0}$. up to extracting a subsequence, we may assume that $u$ converges to $-c\in [-\infty,0]$. Let us prove that $v=v_{z,c,\um}$. By hypothesis, we have $|(T-z)|_v=e^{-c}$ (with the convention $e^{-\infty}=0$). Moreover, for any $z'\neq z$ in $\overline{D(R)}$, we have $|(T-z')|=\lim_{n\to+\infty}|z_n-z'|^{\epsilon_n}_{\infty} = 1$. From the aforementioned description of $V_{R,\um}$, we see that $v=v_{z,c,\um}$. 
\end{proof}

\section{Global analytic spaces}
\label{sec:global_analytic_spaces}

We now give alternative approaches to global analytic spaces. This is a global counterpart to \S \ref{sec:local_analytic_spaces}. Throughout this section, we fix a field $K$.

\subsection{Model global analytic space}
\label{sub:model_global_analytic_space}

In this subsection, we assume that we have an integral structure $(A,\|\cdot\|)$ for $K$ (Definition \ref{def:integral strucure}). We denote by $V:=\cM(A,\|\cdot\|)$ the corresponding global space of pseudo-absolute value. We will make heavy use of the theory developed by Lemanissier-Poineau in \cite{LemanissierPoineau24}. 

We further assume that $(A,\|\cdot\|)$ is a geometric base ring (cf. \S \ref{subsub:global_analytification}). 
Let $X\to \Spec(K)$ be a projective $K$-scheme and let $\cX \to \Spec(A)$ be a projective model of $X$ over $A$. Assume that $\cX$ is a coherent model, namely $\cX \to \Spec(A)$ is finitely presented. Then one can consider the $A$-analytic space associated with $\cX$, denoted by $\cX^{\an}$ (cf. \S \ref{subsub:global_analytification}).

\begin{definition}
\label{def:global_model_analytic_space}
With the above notation, the $A$-analytic space $\cX^{\an}$ is called the \emph{global model analytic space} attached to the model $\cX$. It is a compact Hausdorff topological space (\cite{LemanissierPoineau24}, Lemme 6.5.1 and Proposition 6.5.3). 
\end{definition} 

From now on, we fix a coherent model $\cX$ of $X$ over $A$ and denote by $X^{\an}$ the associated global analytic space. The global analytic space $\cX$ comes with two maps of locally ringed spaces $p : \cX^{\an} \to V:=\cM(A)$ and $j : \cX^{\an} \to \cX$. In this thesis, we will only use the topological properties of analytic spaces. 

Let $v=(\va_v,A_v,\m_v,\kappa_v) \in V$. Then $\cX_v := \cX \otimes_{A} A_v$ is a coherent model of $\cX$, which is flat if $\cX$ is flat itself. Thus $(\cX_v \otimes_{A} \widehat{\kappa_v})^{\an}$ is a local model analytic space in $v$ of $X$ in the sense of Definition \ref{def:model_local_analytic_space}.

\begin{proposition}[\cite{LemanissierPoineau24}, Proposition 4.5.3]
\label{prop:LP_4.4.8}
We have a homeomorphism 
\begin{align*}
(\cX_v \otimes_{A} \widehat{\kappa_v})^{\an} \cong p^{-1}(v).
\end{align*}
\end{proposition}

As we will need results about finite extension of the base field, we now describe the behaviour of model global analytic with respect to such extension. We further assume the space $V_A$ is tame (Definition \ref{def:tame_space}). Let $K'/K$ be a finite extension. Recall that from Proposition \ref{prop:algebraic_extension_of_tame_spaces}, if $A'$ denotes the integral closure of $A$ in $K'$, there is natural tame integral structure $(A',\|\cdot\|')$ on $K'$ induced by $(A,\|\cdot\|)$. We assume that $(A',\|\cdot\|')$ is a geometric base ring. %TODO enlever cette hypothèse. 
Then $\cX' := \cX \otimes_{A} A'\to \Spec(A')$ is a finitely presented scheme over $\Spec(A')$. We denote by $
\cX'^{\an}$ its analytification. It is called the \emph{extension of scalars} of $\cX$ w.r.t. $K'/K$. %TODO finir

\subsection{General global analytic space}

We now try to introduce a general approach to global analytic spaces. Let $K'/K$ be an algebraic function field. Intuitively, the global analytic space attached to any birational model $X$ of $K'/K$ should be a fibration over $M_K$ whose fibres correspond to the local analytic spaces attached to $X$. For the same reasons that were explained in Proposition \ref{prop:sav_extension_by_generalisation_schemes}, an approach mimicking directly Berkovich spaces is not suitable if one wants to obtain Hausdorff spaces.

\subsubsection{Definition}

\begin{definition}
\label{def:global_analytic_spaces_general}
Let $K'/K$ be an algebraic function field. The \emph{global analytic space} attached to $K'$ is defined to be the global space of pseudo-absolute values $M_{K'}$. 
\end{definition}

\begin{remark}
\label{rem:global_analytic_space_general}
As in the local case (cf. \S \ref{subsub:local_analytic_space_general}), we see that the definition of global analytic space is birational in nature. Namely, it does not depends on any birational model of $K'/K$. The natural "specification morphism" is 
\begin{align*}
M_{K'} \to \ZR(K').
\end{align*}
\end{remark}

\subsubsection{Link with model global analytic spaces}%limite projective de modèles cohérents

We fix an integral structure $(A,\|\cdot\|_{A})$ for $K$, we assume that $(A,\|\cdot\|_{A})$ is a geometric base ring. . Then we have a \emph{specification morphism} $V'_{A} \to \ZR(K'/A)$ which is continuous by a similar argument to Lemma \ref{lemma:local_specification_morphism_continuous}. 

Theorem \ref{th:Zariski-Riemann_projective_limite} (1) gives a homeomorphism 
\begin{align}
\label{eq:iso_global_analytic_space_Zariski-Riemann}
\ZR(K'/A) \cong \displaystyle\varprojlim_{\cX} \cX,
\end{align}
where $\cX$ runs over the projective models of $K'/A$. Let us prove that the isomorphism (\ref{eq:iso_global_analytic_space_Zariski-Riemann}) can be refined in our context. 

\begin{lemma}
\label{lemma:domination_flat models_global_analytic_spaces}
Let $\cX$ be a projective model of $K'/A$. Then there exists a flat and coherent projective model $\cX'$ of $K'/A$ dominating $\cX$.
\end{lemma}

\begin{proof}
Consider the sheaf of ideals $(\cO_{\cX})_{\mathrm{tor}}$ of $\cO_{\cX}$ as an $\cO_{\Spec(A)}$-module. Then let $\cX'$ denote the closed subscheme of $\cX$ defined by $(\cO_{\cX})_{\mathrm{tor}}$. It is a projective model of $K'/A$. Moreover, since $\cO_{\cX}/(\cO_{\cX})_{\mathrm{tor}}$ is a torsion-free $\cO_{\Spec(A)}$, hence $\cX' \to \Spec(A)$ is flat. By the flattening theorem of Raynaud-Gruson (\cite{RaynaudGruson71}, Corollaire 3.4.7), $\cX'$ is also coherent. 
\end{proof}

In view of Lemma \ref{lemma:domination_flat models_global_analytic_spaces}, we see that any two projective models of $K'/A$ are dominated by a flat and coherent projective model of $K'/A$. Therefore we have a homeomorphism 
\begin{align*}
\ZR(K'/A) \cong \displaystyle\varprojlim_{\cX} \cX,
\end{align*}
where $\cX$ runs over the flat and coherent projective models of $K'/A$. We will now prove an analytic analogue of this homeomorphism.

Let $V_{A}':= \{v'\in M_{K'} : v'_{|K} \in V_A\}$. Let $\cX$ be a flat and coherent projective model of $K'/A$. Let $v'=(|\cdot|',A',\m',\kappa')\in V'_{A}$. Denote by $v$ the restriction of $v'$ to $K$ and let $\cX_v := \cX \otimes_{A} A_v$. Then the construction of the map $\red_{\cX_v} : M_{K',v} \to \widehat{\cX_v}^{\an}:=(\cX_v \otimes_{A_v} \widehat{\kappa_v})^{\an}$ in \S \ref{subsub:link_local_analytic_spaces} yields a point $x\in \widehat{\cX_v}^{\an}$. Let $\pi: \cX^{\an} \to V_A$ denote the structural morphism. By Proposition \ref{prop:LP_4.4.8}, we have a homeomorphism $\cX_v^{\an} \cong p^{-1}(v)$. Thus we obtain a map $\red_{\cX} : V'_A \to \cX^{\an}$. Moreover the arguments in \S \ref{subsub:link_local_analytic_spaces} adapt \emph{mutatis mutandis} to show that the construction of the maps $\red_{\cX}$ is compatible with the domination relation between projective models of $K'/A$, so that we obtain a commutative diagram

\begin{center}
% https://tikzcd.yichuanshen.de/#N4Igdg9gJgpgziAXAbVABwnAlgFyxMJZABgBpiBdUkANwEMAbAVxiRAFkB9YAaQHJSNAL4ghpdJlz5CKAIzkqtRizYAdVQC0ASgAp+AegCCASlHiQGbHgJEysxfWatEIdfQBOadxABWDLAC23OoAxgAaQqFh3HBCAHrA6nRgImISVtJE8vbUjioubnSe3n6BwarhkRVhooowUADm8ESgAGbeAUhkIDgQSABM1Ax0AEYwDAAKktYyIO5YDQAWOCC5ys6uqu71Zm0dSPI9fYgAzGtOahUEDbsg7RCdiINHSCdpd-uI3b0H5-kgPlqQiAA
\begin{tikzcd}
V'_{A} \arrow[d, "\red"'] \arrow[r, "j"] & \ZR(K'/A) \arrow[d, "\cong"] \\
\varprojlim_{\cX\in \cM}\cX^{\an} \arrow[r]     & \varprojlim_{\cX \in \cM}\cX        
\end{tikzcd},
\end{center}
where $\cM$ denotes the collection of flat and coherent projective models of $K'/A$.

\begin{theorem}
\label{th:link_global_analytic_spaces}
We use the above notation. The map $\red : V'_{A} \to \varprojlim_{\cX\in\cM}\cX^{\an}$ is a homeomorphism. 
\end{theorem}

\begin{proof}
First note that the arguments in the proof of Proposition \ref{prop:continuity_reduction_local} adapt directly in our setting to obtain the continuity of the map $\red_{\cX}$, for any $\cX\in \cM$, and thus the continuity of $\red$. By compactness of $V'_{A}$ and $\varprojlim_{\cX}\cX^{\an}$, it suffices to prove that $\red$ is bijective. 

For the injectivity of $\red$, the elements in the proof of Theorem \ref{prop:local_analytic_space_projective_limit} adapt \emph{mutatis mutandis} (using Theorem \ref{th:Zariski-Riemann_projective_limite} (2) and (\cite{BouAC}, Chap. IX, Appendice 1, Proposition 1)). 

For the surjectivity of $\red$, Lemma \ref{lemma:projective_limit_CHaus_spaces} implies that it suffices to prove the surjectivity of $\red_{\cX}$ for any $\cX\in\cM$. Let $\cX\in\cM$ and let $x\in \cX^{\an}$. Denote by $p$ the image of $x$ via the specification morphism $\cX^{\an} \to \cX$. Then Lemma \ref{lemma:domintation_valuation_ring_residue_algebraic} implies that there exists a valuation ring $(A',\m',\kappa')$ of $K'$ such that $A'$ dominates $\cO_{\cX,p}$ and the residue field extension $\kappa'/\kappa(p)$ is algebraic. Thus we can extend the absolute value on $\kappa(p)$ induced by $x$ to $\kappa'$ in order to obtain a pseudo-absolute value $v'\in M_{K'}$. Moreover, by construction, the restriction of $v'$ to $K$ gives an element of $V_A$, hence $v'\in V'_A$. This concludes the proof of the theorem. 
\end{proof}

\begin{comment}
\subsubsection{Future works}

In this article, we will not study more deeply such spaces and their properties. This should be done in a subsequent work. Let us mention a few conjectural facts.

\begin{itemize}
	\item[(1)] There should be a description as a projective limit of global model analytic spaces over models of Prüfer rings, as in Proposition \ref{prop:local_analytic_space_projective_limit}, without specifying . This should be linked to the description of affine subsets of $\ZR(K')$ in terms of spectra of Prüfer domains with fraction field $K'$. 
	\item[(2)] The structure sheaf should be compatible with the projective limit in (1).
	\item[(3)] Once having specified an adelic structure on $K$ (cf. Chapter \ref{chap:topological_adelic_curves}), one should be able to adapt the constructions of Chapter \ref{chap:higher_dimensional_adelic_geometry} to study general graded linear series of $K'/K$ instead of specifying a birational model of $K'/K$. 
\end{itemize}
\end{comment}

\bibliographystyle{alpha}
\bibliography{biblio}
\end{document}